\documentclass[11pt]{amsart}
\usepackage{amsmath,amsfonts}
\usepackage{amssymb}
\usepackage{amscd}
\usepackage{amsthm}
\usepackage{yhmath}
\usepackage{subfigure}

\usepackage[all]{xy}
\usepackage{color}

\numberwithin{equation}{section}

\newtheorem{theorem}[equation]{Theorem}

\newtheorem{proposition}[equation]{Proposition}

\newtheorem{lemma}[equation]{Lemma}

\newtheorem{corollary}[equation]{Corollary}

\theoremstyle{remark}
\newtheorem{remark}[equation]{Remark}

\theoremstyle{definition}
\newtheorem{definition}[equation]{Definition}

\newtheorem{defn}[equation]{Definition}

\newtheorem{example}[equation]{Example}

\def\Xint#1{\mathchoice
	{\XXint\displaystyle\textstyle{#1}}%
	{\XXint\textstyle\scriptstyle{#1}}%
	{\XXint\scriptstyle\scriptscriptstyle{#1}}%
	{\XXint\scriptscriptstyle\scriptscriptstyle{#1}}%
	\!\int}
\def\XXint#1#2#3{{\setbox0=\hbox{$#1{#2#3}{\int}$}
	\vcenter{\hbox{$#2#3$}}\kern-.5\wd0}}

\def\dashint{\Xint-}

\newcommand{\N}{\mathbb N}

\newcommand{\R}{\mathbb R}

\renewcommand{\L}{{\mathcal L}}

\newcommand{\id}{\operatorname{id}}

\newcommand{\supp}{\operatorname{Supp}}
\newcommand{\p}{\partial}
\newcommand{\e}{\epsilon}
\newcommand{\Lap}{\operatorname{I}_Q}

\newcommand{\vol}{\operatorname{vol}}
\renewcommand{\d}{\operatorname{d}}
\renewcommand{\L}{{\mathcal L}}

\newcommand{\norm}[1]{\left\Vert#1\right\Vert}

\newcommand{\al}{\alpha}

\def\eps{\epsilon}

\def\Om{\Omega}

\newcommand{\Lip}{\operatorname{L}}
\newcommand{\Jac}{\operatorname{J}}

\begin{document}

\title[Conformality in sub-Riemannian manifolds] 
{Conformality and $Q$-harmonicity in sub-Riemannian manifolds \\
{\rm   \tiny (more detailed version of a submitted paper)}}

\author{Luca Capogna}
\address[Capogna]{Worcester Polytechnic Institute,
100 Institute Road,
Worcester, MA 01609, USA.}\email{lcapogna@wpi.edu}

\author{Giovanna Citti}
\address[Citti]{Dipartimento di Matematica, Universit\'a di Bologna, Piazza Porta S. Donato 5,
40126 Bologna, Italy}\email{giovanna.citti@unibo.it}

\author{Enrico Le Donne}
\address[Le Donne]{
Department of Mathematics and Statistics, University of Jyv\"askyl\"a, 40014 Jyv\"askyl\"a, Finland.}
\email{ledonne@msri.org}

\author{Alessandro Ottazzi}
\address[Ottazzi]{
School of Mathematics and Statistics, University of New South Wales, 
Sydney,
NSW 2052 
Australia.}
\email{alessandro.ottazzi@gmail.com}

\keywords{Conformal transformation, quasi-conformal maps, subelliptic PDE, harmonic coordinates, Liouville Theorem, Popp measure, morphism property, regularity for $p$-harmonic functions, sub-Riemannian geometry\\
L.C. was partially funded by NSF awards  DMS 1449143  and DMS 1503683.
G. C. was partially funded by  the People
Programme (Marie Curie Actions) of the European Union's Seventh
Framework Programme FP7/2007-2013/ under REA grant agreement n¡607643.
E.L.D. was supported by the Academy of Finland, project no. 288501.   
A.O. was partially supported by the Australian Research CouncilÕs Discovery Projects funding scheme, project no. DP140100531. 
}

\renewcommand{\subjclassname}{%
 \textup{2010} Mathematics Subject Classification}

\date{March 15, 2016}

\begin{abstract}  
We prove the equivalence of several natural notions of conformal maps between  sub-Riemannian manifolds. Our main contribution is in the setting of those manifolds that support a  suitable regularity theory for subelliptic $p$-Laplacian operators. For such manifolds we prove a Liouville-type theorem, i.e.,  $1$-quasiconformal maps are smooth.
In particular, we prove that contact manifolds support the suitable regularity.
 The main new technical tools are a sub-Riemannian version of $p$-harmonic coordinates and a  technique of propagation of regularity from horizontal layers.
   \end{abstract}

\maketitle


\newpage
\tableofcontents

\newpage

\section{Introduction}

In this paper we establish the equivalence of several notions of conformality in the setting of   sub-Riemannian manifolds.
  In particular,   we show  that $1$-quasiconformal homeomorphisms   (see below for the definition) 
are in fact  smooth conformal diffeomorphisms, 
provided that there exist certain regularity estimates for weak solutions of a class of quasilinear degenerate elliptic PDE, i.e., the subelliptic $p$-Laplacian, see \eqref{QLap}.  
 Moreover, we also show that such PDE regularity estimates hold in the important special case of sub-Riemannian contact manifolds, thus fully establishing the Liouville theorem in this setting. 
Up to now the connection between regularity of quasiconformal maps and the $p$-Laplacian, and  the equivalence of different  definitions  of conformality, were only well understood in the Euclidean, Riemannian, and Carnot-group settings. The general sub-Riemannian setting presents genuinely new difficulties, e.g., subRiemannian manifolds are not locally bi-Lipschitz equivalent to their tangent cones, Hausdorff measures are not smooth, there is a need to construct adequate coordinate charts that are compatible both with the nonlinear PDE and with the sub-Riemannian structure, and,  last but not least, the issue of optimal regularity for $p$-harmonic functions  is still an open problem in the general subRiemannian setting.

The motivations for our work stem from outside of the field of subRiemannian geometry: one of the main applications for Liouville-type theorems, like the one we establish in this paper, comes from a Mostow-type approach   to rigidity questions. Indeed, our main result, Theorem \ref{thm:contact}, implies smoothness of $1$-quasiconformal boundary extensions of mappings between strictly pseudoconvex smooth open subsets of $\mathbb C^n$.
Following an idea of Cowling, it seems plausible that this regularity may lead to a new approach to proving Fefferman's smooth extension theorem for bi-holomorphisms (see \cite{MR0350069}), in the same spirit as Mostow rigidiy.   From a related  point of view, our work also adds a new contribution to the study of quasiconformal transformations between CR manifolds. This line of research was initiated by Kor\'anyi and Reimann in 1985 (see \cite{Koranyi-Reimann85, MR1055843,Puqi:Tang}) to investigate problems such as the question of which compact strongly pseudoconvex 3-dimensional CR manifolds are embeddable (see for instance \cite{MR1157290} and references therein). For a further sample of possible applications of our result see the remark after Theorem \ref{thm:contact} .

The issue of  regularity of $1$-quasiconformal homeomorphisms in the  Euclidean case was first studied in 1850 in Liouville's work, where the initial regularity of the conformal homeomorphism was assumed to be $C^3$. In 1958, the regularity assumption was lowered to $C^1$ by Hartman \cite{Hartman58} and then, in conjunction with the proof of the De Giorgi-Nash-Moser Regularity Theorem, further decreased to the Sobolev spaces $W^{1,n}$, in the work of Gehring \cite{Gehring_rings} and Re{\v{s}}etnjak  \cite{Resetnjak}. The role of the De Giorgi-Nash-Moser Theorem in Gehring's proof consists in providing adequate $C^{1,\alpha}$ estimates for solutions of the Euclidean $n$-Laplacian, that are later bootstrapped to $C^{\infty}$ estimates by means of elliptic regularity theory. It is worthwhile to observe that the quest for the optimal regularity assumption in this problem is still open (see \cite{Iwaniec_Martin_book}). 

The regularity  of $1$-quasiconformal maps in the Riemannian case is considerably more difficult than the Euclidan case. It was finally settled in 1976 by Ferrand  \cite{Lelong-Ferrand76,Lelong-Ferrand77,Lelong-Ferrand79}, in occasion of her work on Lichnerowitz's conjecture and was modeled after Re{\v{s}}etnjak's original proof. See also related work of Bojarski, Iwaniec, Kopiecki \cite{BIK}. More recently, inspired by Taylor's regularity proof for isometries via harmonic coordinates, Liimatainen  and Salo \cite{Salo} provided a new proof for the regularity of $1$-quasiconformal maps between Riemannian manifolds. Their argument is based on the notion of $n$-harmonic coordinates, on the morphism property for $1$-quasiconformal maps, and on the $C^{1,\alpha}$ regularity estimates for the $n$-Laplacian on manifolds. The proofs in the present paper are modeled on Taylor's approach, as developed   in \cite{Salo}.

The introduction of conformal and quasiconformal maps in the sub-Riemannian setting goes back to the proof of Mostow's Rigidity Theorem \cite{Mostow}, where such maps arise as boundary limits of quasi-isometries between certain Gromov hyperbolic spaces. Because the class of spaces that arises as such boundaries in other geometric problems includes sub-Riemannian manifolds that are not Carnot groups, it becomes relevant to study conformality and quasiconformality in this more general environment.

In the sub-Riemannian setting the regularity is currently known only in the special case of  $1$-quasiconformal maps in Carnot groups, see \cite{Pansu,Koranyi-Reimann85,Puqi:Tang,Capogna-Cowling,Cowling-Ottazzi-Conformal}. Since such groups arise as tangent cones of sub-Riemannian manifolds then the regularity of $1$-quasiconformal maps in Carnot groups setting is an analogue of the Euclidean case as studied by Gehring and Re{\v{s}}etnjak.
To further describe the present work we begin with the introduction of the notions of conformal and quasiconformal maps between sub-Riemannian manifolds (see Section \ref{sec:SR:def} for the relevant definitions).

\begin{defn}[Conformal map]\label{conf}
A smooth diffeomorphism  between two sub-Riemannian manifolds is {\em conformal} if its differential maps horizontal vectors into horizontal vectors, and its restrictions to the horizontal spaces are  similarities\footnote{A map $F:X\to Y$ between metric spaces is called a {\em similarity} if there exists a constant $\lambda>0$ such that 
  $d(F(x),F(x'))=\lambda d(x,x')$, for all $x,x'\in X$. }.
\end{defn}

The notion of quasiconformality can be formulated with minimal regularity assumptions  in arbitrary metric spaces.

\begin{defn}[Quasiconformal map]\label{def:quasiconformal}
A  \emph{quasiconformal} map between two metric spaces $(X,d_X)$ and $(Y,d_Y)$ is a homeomorphism $f:X\to Y$ for which there exists a constant $K\geq 1$ such that for all $p\in X$
\begin{displaymath}
H_f(p):=\limsup_{r \to 0}\frac{
\sup \{d_Y(f(p),f(q)):  {d_X(p,q)\leq r}   \}}
{
\inf \{d_Y(f(p),f(q)) : {d_X(p,q)\geq r}\}}\leq K.
\end{displaymath}
\end{defn}

We want to clarify what is the correct definition of $1$-quasiconformality since in the literature there are several equivalent definitions of quasiconformality  
associated to possibly different bounds for different types of distortion (metric, geometric, or analytic).

In order to state our results we  need to recall a few basic notions and introduce some notation.
We consider the 
  following 
metric quantities 
$$
\Lip_f(p)  := \limsup_{q\to p} \dfrac{d(f(p), f(q))}{d(p,q)} \qquad\text{ and }\qquad
\ell_f(p):= \liminf_{q\to p} \dfrac{d(f(p), f(q))}{d(p,q)}.$$
The quantity $\Lip_f(p)$ is sometimes denoted by ${\rm Lip}_{ f}(p) $ and is called the pointwise Lipschitz constant.
Given an equiregular sub-Riemannian manifold $M$,
we denote by $Q$ its
  Hausdorff dimension with respect to the 
Carnot-Carath\'eodory    distance, and we write $\nabla_{\rm H}$ for the horizontal gradient, see Section~\ref{sec:SR:def} for these definitions. 
We denote by $\vol_M$  the Popp measure on $M$
and denote by $\Jac_f^{\rm Popp}$
the Jacobian    of a map $f$ between equiregular sub-Riemannian manifolds when these manifolds are equipped with their Popp measures (see Section~\ref{jacobians}).
By $W_{\rm H}^{1,Q}(M)$ we indicate the space of functions $u\in L^Q(\vol_M)$ such that $|   \nabla_{\rm H} u |\in L^Q(\vol_M)$.
We use the  standard notation ${\rm Cap}_Q$ and ${\rm Mod}_Q$
for  capacity and  modulus  (see Section~\ref{sec:geodef}).
We also consider the nonlinear pairing
$$
\Lap(u,\phi;U):=\int_{U} |\nabla_{\rm H} u|^{Q-2} \langle \nabla_{\rm H} u, \nabla_{\rm H} \phi\rangle \d\vol_M,
$$
with $u,\phi \in W_{\rm H}^{1,Q}(U)$ and $U\subset M$ an open subset. 
For short, we write $\Lap(u,\phi)$ for $\Lap(u,\phi;M)$ and denote by ${\rm E}_Q(u)=\Lap(u,u;M)$ the {\it $Q$-energy} of $u$.
The functional $\Lap(u,\,\cdot\,)$  defines the weak form of the $Q$-Laplacian when acting on the appropriate function space, see Section~\ref{sec:QLap}.
Given a quasiconformal  homeomorphism $f$  between two equiregular sub-Riemannian manifolds,
we denote by $\mathcal N_p(f)$ the   Margulis-Mostow differential of $f$ and by $(\d_{\rm H} f)_p$
its horizontal differential (see Section~\ref{sec:tange}).

\begin{theorem}\label{theorem0}
Let $f$ be  a quasiconformal  map  between two equiregular sub-Riemannian manifolds of Hausdorff dimension $Q$. 
The following are equivalent:
\begin{equation}\label{H}
\text{
$H_f(p)=1$ for a.e. $p$;
}
 \end{equation}
\begin{equation}\label{H=}
\text{
$H_f^=(p):=\limsup_{r \to 0}\frac{
\sup \{d(f(p),f(q)):  {d(p,q)= r}   \}}
{
\inf \{d(f(p),f(q)) : {d(p,q)= r}\}}=1$ for a.e. $p$;
}
 \end{equation}
\begin{equation}\label{HS}
\text{
 $(\d_{\rm H} f)_p$ is a similarity for a.e. $p$;
 }
 \end{equation}
 \begin{equation}\label{S}
\text{
 $\mathcal N_p (f)$ is a similarity for a.e. $p$;
 }
 \end{equation}
\begin{equation}\label{L}
\text{
 $\ell_{   f}(p)=\Lip_{  f}(p)$ for a.e. $p$,
 i.e., the limit $\lim_{q\to p} \dfrac{d(f(p),f(q))}{d(p,q)}$ exists for a.e. $p$;
 }
 \end{equation}
\begin{equation}\label{LN}
\text{
$\ell_{\mathcal N_p (f)}(e)=\Lip_{\mathcal N_p (f)}(e)$ for a.e. $p$;
 }
 \end{equation}
\begin{equation}\label{JP}
\text{
$\Jac_f^{\rm Popp}(p)= \Lip_f(p)^Q$ for a.e. $p$;
}
 \end{equation}
\begin{equation}\label{MP}
\text{The $Q$-modulus  (w.r.t. Popp measure) is preserved:
  }
 \end{equation}
\hspace{3cm}
${\rm Mod}_Q(\Gamma) = {\rm Mod}_Q(f(\Gamma)),\qquad\forall \, \Gamma \text{ family   of curves in } M;$
 \begin{equation}\label{LP}
\text{
   \text{The operators $\Lap$  (w.r.t. Popp measure) are preserved:}
  }
 \end{equation}
\hspace{3cm}  
$\Lap(v, \phi; V)=\Lap(v\circ f , \phi \circ f;  f^{-1}(V))
,\quad\forall \, 
V\subset N \text{ open}, \forall\, v, \phi \in W_{\rm H}^{1,Q}(V)
.
$
\end{theorem}

\begin{defn}[1-quasiconformal map]\label{def:one-quasiconformal}
We say that a quasiconformal map between two equiregular sub-Riemannian manifolds is {\it 1-quasiconformal} if any of the conditions in Theorem \ref{theorem0} holds.
\end{defn}

The equivalence of the definitions in Theorem \ref{theorem0} have as consequences some invariance properties that are crucial in the proofs of this paper.

\begin{corollary}\label{EP-CP}
Let $f$ be a $1$-quasiconformal map between equiregular sub-Riemannian manifolds of Hausdorff dimension $Q$. Then
\begin{itemize}
\item[(i)]
the $Q$-energy (w.r.t. Popp measure) is preserved:  
\begin{equation}\label{EP}
{\rm E}_Q(v) = {\rm E}_Q(v \circ f)   , \qquad  \forall   \,v\in W^{1,Q}_{H}(N) ; 
 \end{equation}
\item[(ii)] the $Q$-capacity  (w.r.t. Popp measure) is preserved:
 \begin{equation}\label{CP}
 {\rm Cap}_Q(E,F)={\rm Cap}_Q(f(E),f(F))   , \qquad  \forall \,E,F\subset M \text{ compact}.
  \end{equation}
  \end{itemize}
\end{corollary}
The proofs of Theorem~\ref{theorem0} and Corollary~\ref{EP-CP} are given in Section~\ref{defs1qc}. A detailed guide to such proofs in contained at the beginning of that section.

While the Hausdorff measure may seem to be the natural volume measure to use in this context, there is a subtle and  important reason for choosing  the Popp measure rather than the Hausdorff measure. Indeed, the latter may not be smooth, even in equiregular sub-Riemannian manifolds, see \cite{Agrachev_Barilari_Boscain:Hausdorff}. However, we show that for $1$-quasiconformal maps the corresponding Jacobians coincide.
The following statement is a consequence of Theorem~\ref{theorem0} and Proposition~\ref{equivalence_of_jacobians}.
\begin{corollary}\label{analytic-definition}
Let $f$ be a $1$-quasiconformal map between equiregular sub-Riemannian manifolds of Hausdorff dimension $Q$. Then for almost every $p$
$$
\ell_f(p)^Q=\Lip_f(p)^Q=\Jac_f^{{\rm Popp}}(p)=\Jac_f^{{\rm Haus}}(p).
$$
Moreover, the inverse map $f^{-1}$ is  $1$-quasiconformal. 
\end{corollary}

Since the Popp measure is smooth, the associated $Q$-Laplacian operator $L_Q $ will involve smooth coefficients
and consequently it is plausible to conjecture the existence of  a regularity theory of  $Q$-harmonic functions (see Section~\ref{sec:QLap} for the definitions).
In fact such a theory exists 
in the important subclass of contact manifolds (see Section~\ref{sec:contact}).
The following result is the {\it morphism property} for   $1$-quasiconformal maps, and it is proved in Section~\ref{morphism-prop}.  The $Q$-Laplacian operator $L_Q$ is defined in \eqref{QLap}.
\begin{corollary}[Morphism property]\label{morphism-property} 
Let $f:M\to N$ be a $1$-quasiconformal map between equiregular sub-Riemannian manifolds  of Hausdorff dimension $Q$ equipped with their Popp measures.  The following hold:
\begin{itemize}
\item[(i)] The $Q$-Laplacian is preserved:

 If $v\in    W_{\rm H}^{1,Q}(N)$,
then
$L_Q ( v\circ f) \circ f^*  =   L_Q v $, where $ f^*$ denotes the pull-back operator on functions.


\item[(ii)] The $Q$-harmonicity  is preserved:

 If $v$ is a $Q$-harmonic function on $N $,
then
$v\circ f $ is a $Q$-harmonic function on $M $.
\end{itemize}

 \end{corollary}
 
 One of the main results of this paper is the following 
 Liouville-type theorem. Its proof is the content of Section~\ref{contactsmooth}.
 
 \begin{theorem}\label{thm:contact} 
Every $1$-quasiconformal map between sub-Riemannian  contact manifolds is conformal. 
 \end{theorem}

For some related results in the setting of CR 3-manifolds see \cite{Puqi:Tang}.
As a sample application of Theorem \ref{thm:contact}  we recall 
the following two examples of contact sub-Riemannian manifolds.
Consider the roto-translation group (i.e., the isometry group of the Euclidean plane) and the Heisenberg group, both equipped with their standard 
sub-Riemannian structures.
They are locally quasiconformal (in fact they admit a global map that is locally bi-Lipschitz, with an explicit correspondence that can be found in \cite{Faessler_Koskela_LeDonne}).  Theorem \ref{thm:contact} shows that such local equivalence cannot be extended to $1$-quasiconformality. In fact, if this was the case, then the two spaces would be locally conformal,
but they are not, in view of a recent result of Boarotto \cite{MR3459235}.

Theorem \ref{thm:contact}   follows from a more general theorem.
 In fact, it continues to hold 
  in the class of sub-Riemannian  manifolds that 
 support a regularity theory for  $Q$-harmonic functions.
 This class includes every sub-Riemannian manifold that is locally contactomorphic to a Carnot group of step $2$ or, equivalently, every 
 Carnot group of step $2$ with a sub-Riemannian metric that is not necessarily left-invariant.
 We remark that there are examples of step-2 sub-Riemannian manifolds that are not  contactomorphic to any Carnot group, see \cite{LOW}. 
 In order to describe in detail the more general result we introduce  the following definition.

\begin{definition}\label{support} Consider an equiregular  sub-Riemannian  manifold $M$ of Hausdorff dimension $Q$, with horizontal bundle of dimension $r$,
endowed with a smooth volume form. 
We say that $M$ {\em supports regularity for  $Q$-harmonic functions} if 
the following two properties hold:
\begin{enumerate} \item  
For every $g=(g^1,...,g^r)\in C^\infty(M,\R^r)$, $U\subset\subset M$ and for every $\ell >0$, there exist constants $\alpha\in (0,1), C>0$ 
such that
for each weak
 solution $u$ of  the equation $L_Q u= X_i^* g^i$ on $M$
with
$||u||_{W^{1,Q}_{\rm H}( U)} <\ell$, one has
\begin{equation*}
 ||u||_{C^{1,\alpha}_{\rm H}( U)} \le C.
\end{equation*}
\item
For every $g=(g^1,...,g^r)\in C^\infty(M,\R^r)$, $U\subset \subset M$
 and for every $\ell , \ell' >0$,
there exists a constant $C>0$ 
such that
for each  weak
 solution $u$ of  the equation $L_Q u= X_i^* g^i$ on $M$
with
$||u||_{W^{1,Q}_{\rm H}( U)} <\ell$ 
and
$\tfrac{1}{\ell'}<|\nabla_{\rm H} u|<\ell'$ on $U$,
one has
\begin{equation*} 
||u||_{W^{2,2}_{\rm H}(U)} 
\le C.
\end{equation*}
\end{enumerate}
\end{definition}
In a similar fashion one can reformulate the definition to the case of $p-$harmonic functions.
%

 In view of the work of Uraltseva \cite{MR0244628} (but see also\cite{MR0474389,Tolksdorf,DiBenedetto}) every Riemannian manifold supports regularity for $Q$-harmonic functions. Things are less clear in the sub-Riemannian setting.
The H\"older regularity of weak solutions of quasilinear PDE  $\ \sum_{i=1}^{r} X_I^* A(x,\nabla_H u)=0$, modeled on the subelliptic $p$-Laplacian, for $1<p<\infty$,  and for their parabolic counterpart,  is well known, see \cite{CDG, ACCN}.  However, in this generality the higher regularity of solutions is still an open problem.
The only  results  in the literature are for the case of left-invariant sub-Riemannian structures on step two Carnot groups. Under these assumptions one has that solutions in the range $p\ge 2$ have H\"older regular horizontal gradient. This is a formidable achievement in itself, building on  contributions by several authors \cite{Capogna-cpam,MR2085543,MR2257051,MR2336058,  MR2545517,MR2531368, MR2527629,MR3444525}, with the final result being established eventually by Zhong in \cite{Zhong}.  Beyond the Heisenberg group one has some promising results due to Domokos and Manfredi \cite{MR2424544, MR2556756, MR2676172} in the range of $p$ near $2$. In this paper we build on these previous contributions, particularly on Zhong's work \cite{Zhong} to include the dependence on $x$ and prove that contact sub-Riemannian manifolds support regularity for $Q$-harmonic functions (see Theorem \ref{zhong-main}). Note that very recently Zhong and Mukherjee have announced the extension of the regularity result to the range $1<p\le 2$, while  Garofalo, Markasheva and the first two named authors  \cite{CCGM} have extended Zhong's theorem to the parabolic case.
The novelty of our approach is that we use a Riemannian approximation scheme to regularize the $Q$-Laplacian operator, thus allowing to approximate its solutions with smooth functions. In carrying out this approximation the main difficulty is to show that the regularity estimates do not blow up as the approximating parameter approaches the critical case. Our main result in this context, proved in Section \ref{contactsmooth}, is the following.

\begin{theorem}  Sub-Riemannian contact manifolds {\em support regularity for  $p$-harmonic functions} for every $p\ge 2$.
\end{theorem}

%
%
\endremark

%
%
%

The regularity hypotheses in Definition~\ref{support} 
have two important consequences.
First, 
it allows us to construct horizontal $Q$-harmonic
coordinates.
Second, together with the existence of such coordinates,
it
eventually leads to an initial $C^{1,\alpha}$ regularity for  
 $1$-quasiconformal 
 maps (see Theorem \ref{duballe} (ii)).
 When this basic regularity is present, 
one can use classical PDE arguments 
to derive smoothness without the additional hypothesis of Definition~\ref{support} (see Theorem \ref{duballe} (i)).

\begin{theorem}\label{duballe}

Let $f:M\to N$ be a $1$-quasiconformal map between equiregular sub-Riemannian manifolds  of Hausdorff dimension $Q$,
 endowed with   smooth volume forms. 
 
 (i)  If $f$ is bi-Lipschitz  and in $C^{1,\alpha}_{\rm{H, loc}}(M,N)\cap W^{2,2}_{\rm H,loc}(M,N)$, then $f$ is   conformal
.
 

(ii)  If  $M$ and $N$
 support regularity for  $Q$-harmonic functions (in the sense of Definition~\ref{support}),
 then 
  $f$ is bi-Lipschitz  and in $C^{1,\alpha}_{\rm H, loc}(M,N)\cap W^{2,2}_{\rm H,loc}(M,N)$,    and hence conformal.
\end{theorem}
The function spaces in Theorem \ref{duballe} are defined componentwise, see Section~\ref{sec:regular}.
Theorem~\ref{duballe} (i) is proved in Section~\ref{unapizza}.
Theorem~\ref{duballe} (ii) is proved in Section~\ref{sec:smoothness_conclusion}.

The above theorem provides
 the following result.

\begin{corollary}\label{theorem1}
Let $f$ be  a  homeomorphism  between two equiregular sub-Riemannian manifolds each
supporting the regularity estimates in Definition~\ref{support}. 
The map $f$ is conformal if and only if it is 1-quasiconformal.
\end{corollary}

\begin{remark} Note that in view of the regularity theory in the  work of Uraltseva \cite{MR0244628}, we recover Ferrand's theorem (and the Liimatainen--Salo argument): every $1$-quasiconformal map between Riemannian manifolds is conformal. \end{remark}

The proof   of the first part of Theorem \ref{duballe}  rests on the morphism property for $1$-quasiconformal   maps (see Theorem \ref{morphism-property}) and on Schauder's estimates, as developed by Rothschild and Stein \cite{Roth:Stein} and Xu \cite{MR1135924}. The second part is one of the main contributions of this paper and is based on the construction of ad-hoc systems of coordinates, the {\it horizontal $Q$-harmonic coordinates}, that play an analogue role to that of the $n$-harmonic coordinates in the work of Liimatainen--Salo \cite{Salo}. However, in contrast to the Riemannian setting, only a subset of the coordinate systems (the {\it horizontal} components) can be constructed so that they are  $Q$-harmonic, but not the remaining ones. This yields a potential obstacle, as $Q$-harmonicity is the key to the smoothness of the map. We remedy to this potential drawback by producing an argument  showing that if an ACC map has suitably regular horizontal components then such regularity is transferred to all the other components (see Proposition~\ref{smoothfromhorizontal}). This method was introduced in \cite{Capogna-Cowling} in the special setting of Carnot groups, where $Q$-harmonic horizontal coordinates arise naturally as the exponential coordinates associated to the first layer of the stratification.
\bigskip

Looking ahead, it seems plausible to conjecture that the Liouville theorem holds in any equiregular sub-Riemannian manifold. Our work shows in fact  that this is implied by the regularity theory for $p-$Laplacians and the latter is widely expected to hold for general systems of H\"ormander vector fields. However the latter remains a challenging open problem.

\bigskip
We conclude this introduction with a comparison between our work and the Carnot group case as studied in \cite{Capogna-Cowling}. In the latter setting one has that all the canonical exponential horizontal coordinates happen to be  also  smooth $Q$-harmonic (in fact they are also harmonic).  Moreover, a simple argument based on the existence of dilations and 
the $1$-quasiconformal invariance of the conformal capacity 
(see \cite{Pansu}) yields the bi-Lipschitz regularity for $1$-quasiconformal maps immediately, without having to invoke any PDE result. As a consequence {\it the Liouville theorem in the Carnot group case can be proved relying on a much weaker regularity theory than the one above, i.e., one has just to use the $C^{1,\alpha}$ estimates for the $Q$-Laplacian in the simpler case where the gradient is bounded away from zero and from infinity}. This result goes back to  the earlier work of the first-named author \cite{Capogna} in the Carnot group setting. In our more general, non-group setting, there are no canonical $Q$-harmonic coordinates, and so one has to invoke the PDE regularity to construct them. Similarly, the lack of dilations makes it necessary to rely on the PDE regularity also to show bi-Lipschitz regularity.

\noindent{\it Acknowledgements.} 
We are  grateful to Xiao Zhong for many useful conversations. 
%
%
%


\section{Preliminaries}
\subsection{Sub-Riemannian geometry}\label{sec:SR:def}

A {\em sub-Riemannian manifold} is a connected, smooth manifold $M$ endowed with   a      subbundle  $HM$ of the tangent bundle $TM$ that bracket generates $TM$ and 
 a smooth section of positive-definite quadratic forms 
 $g$     on $HM$,
see \cite{Montgomery}. 
The form $g$ is locally completely determined by any orthonormal frame 
$X_1, \ldots, X_r $   of $ HM $. The bundle $ HM $ is called {\em horizontal 
distribution}.
The section
$g$ is called {\em sub-Riemannian 
metric}.

 Analogously to the Riemannian setting, one can endow a sub-Riemannian manifold $M$ with a metric space structure by defining the {\it Carnot-Carath\'eodory}    distance: For any pair $x,y\in M$ set  \begin{multline} d(x,y)=\inf\{\delta>0 \text{ such that there exists a curve } \gamma\in C^\infty([0,1]; M) \text{ with endpoints }x,y \\\text{ such that }\dot\gamma\in H_\gamma M  \text{ and }|\dot\gamma|_g\le \delta\}.\notag\end{multline}

Consider a sub-Riemannian manifold $ M$ with horizontal distribution $HM$ and denote by $\Gamma(HM)$ the smooth sections of $HM$, i.e., the vector fields tangent to $HM$.
For all $k\in \N$, consider
$$H^kM:={\bigcup_{q\in M} 
\rm span}\{[Y_1,[Y_2,[\dots [Y_{l-1},Y_l]]]]_q\;:\; l\leq k,Y_j\in\Gamma(HM),   j=1,\dots,l\}.
$$
The
bracket generating condition (also called {\it H\"ormander's finite rank hypothesis})
is expressed by the existence of $s\in \N$ such that
$H^sM=TM$. 
\begin{definition}\label{equiregular} 
A sub-Riemannian manifold $ M$ with horizontal distribution $HM$ is {\it equiregular} if, for all $k\in \N$, 
each set $H^kM$ defines a subbundle of $TM$.
\end{definition}

%

Consider the metric space $(M,d)$ where $ M$ with horizontal distribution $\Delta$  is an equiregular sub-Riemannian manifold and $d$ is the corresponding  Carnot-Carath\'eodory    distance. 
As a consequence of Chow-Rashevsky
Theorem
 such a  distance is always finite  and induces on $M$ the original topology.
 As a result of Mitchell \cite{Mitchell}, the Hausdorff dimension of $(M,d)$ coincides with 
  the Hausdorff dimension of its tangents spaces.

Let  $ X_1, \ldots, X_r  $ be an orthonormal frame of the horizontal distribution of a    sub-Riemannian manifold $M$. 
We define the {\em horizontal gradient}  of a function $u:M\to \R$ with respect to   $X_1,\dots,X_r$ as
\begin{equation}\label{hor_grad}
\nabla_{\rm H} u:= (X_1 u)X_1+\ldots +(X_r u)X_r.
\end{equation}
\remark
Let $X_1^\prime,\dots,X_r^\prime$ be another frame of the same distribution. Let $B$ be the matrix such that 
$$
X_j^\prime(p)=\sum_{i=1}^r B_j^i(p)X_i(p).
$$
Then the horizontal gradient $\nabla_{\rm H}^\prime u$ of $u$ with respect to $X_1^\prime,\dots,X_r^\prime$ is
\begin{align*}
\nabla_{\rm H}^\prime u(p)&= \sum_j(X_j^\prime u(p))X_j^\prime(p)\\
&=\sum_j(\sum_i B_j^i(p)X_i(p)u) \sum_k B_j^k(p) X_k(p)\\
&=\sum_i\sum_j\sum_k B_j^i(p)B_k^j(p)^T X_iu(p)X_k(p)\\
&= (B(p) B(p)^T)_k^i X_iu(p)X_k(p).
\end{align*}

\remark
If $X_1,\dots,X_r$ and $X_1^\prime,\dots,X_r^\prime$ are two frames that are orthonormal with respect to  a sub-Riemannian structure on the distribution, then $\nabla_{\rm H}^\prime u=\nabla_{\rm H} u$. Indeed, in this case the matrix $B(p)$ would be in $O(r)$ for every $p$.


\subsection{PDE preliminaries}

In this section we collect some of the PDE results that will be used later in the paper.
 Let ${X_1,\dots,X_r}$ be an   orthonormal   frame  of the horizontal bundle of a sub-Riemannian manifold $M$.
For each $i=1,\ldots,r$ denote by $X_i^*$ the {\em adjoint} of $X_i$ with respect to a smooth volume form $\vol$, i.e.,
 $$\int_M u X_i \phi \d \vol = \int_M X_i^* u \phi \d\vol,$$ 
 for every compactly supported $\phi$ for which the integral is finite. In any system of coordinates, the smooth volume form can be expressed in terms of the Lebesgue measure $ \mathcal{L}$ through a smooth density $\omega$, i.e.,  $\d\vol= \omega \d\mathcal{L}$. If in local coordinates we write $X_i=\sum_{k=1}^n b^i_k\p_k$,
  then one has 
\begin{equation}\label{sdj:vec}
X_i^*u=  - \omega^{-1} (X_i (\omega u))   - u\p_k b^i_k   .
\end{equation}

Next we define some of the function spaces that will be used in the paper.

\begin{definition} Let ${X_1,\dots,X_r}$ be an   orthonormal   frame  of the horizontal bundle of a sub-Riemannian manifold $M$ and consider an open subset $\Omega\subset M$.  For  any   $k\in \N$,  and $\alpha\in (0,1)$   we define the $C^{k,\alpha}_{\rm H}$ norm 
$$ \|u\|^2_{C^{k,\alpha}_{\rm H}(\Omega)}:=\sup_\Omega ( \sum_{|I|\le k-1} |X^I u|^2)+ \sup_{p,q\in \Omega \text{ and }p\neq q} \frac{\sum_{|I|=k}  |X^Iu(p)- X^Iu(q)|^2}{d(p,q)^{2\alpha}},$$ 
where,  for each $m=0,\ldots,k$ and each 
$m$-tuple $I=(i_1,\ldots,i_m)\in \{1,\ldots,r\}^m$, 
we have denoted  by $X^I$ 
the   $m$-order  operator $X_{i_1}\cdots X_{i_m}$
 and   we   set $|I|=m$. 
We write
$$
C^{k,\alpha}_{\rm H}(\Om)=\left\{u:\Omega\to \R : \  X^I u \text{ is continuous in }\Om \text{ for } |I|\leq k\ \text{ and }    \|u\|_{C^{k,\alpha}_{\rm H}(\Om)}<\infty\right\}.
$$
  A function $u$ is in $C^{k,\alpha}_{\rm H,loc}(\Omega)$, if for any $K\subset \subset \Omega$ one has $\|u\|_{C^{k,\alpha}_{\rm H}(K)}<\infty$.
  \end{definition}
  
  \begin{definition}\label{Sobolev spaces}
Let ${X_1,\dots,X_r}$ be an   orthonormal   frame  of the horizontal bundle of a sub-Riemannian manifold $M$ and consider an open subset $\Omega\subset M$.  For $k\in \N$ and for any multi-index $I=(i_1,...,i_k)\in \{1,...,r\}^k$ we define $|I|=k$ and $X^I u = X_{i_1}...X_{i_k} u$. For $p\in [1,\infty)$ we define the {\it horizontal Sobolev space} $W_{\rm H}^{k,p}(\Omega)$ 
  to be the space of all $u\in L^p(\Omega)$  whose distributional derivatives $X^Iu$ are also in $L^p(\Omega)$ for all multi-indexes $|I|\le k$. This space can also be defined as the closure   of the space of $C^\infty(\Omega)$  functions   with respect to the norm
\begin{equation}\label{sobolev}
\| u \|_{W_{\rm H}^{k,p}}^p := \| u\|_{L^p(\Omega)}^p + \int_{\Omega} [\sum_{|I|=1}^k (X^I u)^2]^{p/2}  \;\mathrm d{\rm vol},
\end{equation}
see \cite{Garofalo-Nhieu98}, \cite{fssc-ms} and references therein.
A function $u\in L^p(\Omega)$ is in the local Sobolev space $W^{k,p}_{\rm H,loc}(\Omega)$ if, for any $\phi\in C^\infty_c(\Omega)$, one has $u\phi\in W_{\rm H}^{k,p}(\Omega)$. 
\end{definition}

\subsection{Schauder estimates} 
Here we discuss Schauder estimates for second order, non-divergence form subelliptic linear operators.
Given an orthonormal   frame  ${X_1,\dots,X_r}$ of the horizontal bundle of $M$, one defines the {\em subLaplacian} on $M$ of a function $u$ as
 \begin{equation}\label{sublaplacian}
L_2 u:=\sum_{i=1}^r X_i^* X_i u.
 \end{equation}
 One can check that such an operator does not depend on the choice of the orthonormal   frame, but only on the sub-Riemannian structure of $M$ and the choice of the volume form.

Let $\Omega $ be an open set of $M$. A function $u:\Omega\to \R$ is called $2$-{\em harmonic} (or, more simply, {\em harmonic}) if $L_2 u = 0$ in $\Omega$, in the sense of distribution.
H\"ormander's celebrated  Hypoellipticity Theorem  \cite{Hormander} implies that harmonic functions are smooth.
 
A well known result of Rothschild and Stein \cite{Roth:Stein}, yields {\em Schauder estimates} for subLaplacians,  that  is
 if $L_2 u\in C^\alpha_{\rm H}(\Omega)$, then  
for any $K\subset\subset \Omega$, there exists a constant $C$ depending on $K,\alpha$ and the sub-Riemannian structure such that  
$$
\|u\|_{C^{2,\alpha}_{\rm H}(K)} \leq C  \|L_2 u\|_{C^\alpha_{\rm H}(\Omega)}.
$$
In particular we shall use that
\begin{equation}\label{schauder-consequence}
\|u\|_{C^{1,\alpha}_{\rm H}(\overline{B}_{\epsilon/2})}\leq C\|L_2 u\|_{C^\alpha_{\rm H}(B_\epsilon)}.
\end{equation}
The Schauder estimates have been extended to subelliptic operators with low regularity by a number of authors. For our purposes we will 
consider operators of the form 
$$L_{a(x)}u(x):= \sum_{i,j=1}^r a_{ij}(x)X_i X_j u(x),$$
where $a_{ij}$ is a symmetric matrix such that  for some constants $\lambda,\Lambda>0$ one has 
\begin{equation}\label{elliptic}
\lambda |\xi|^2 \le a_{ij}(x) \xi_i \xi_j \le \Lambda |\xi|^2
\end{equation} for every $x\in M$ and for all $\xi\in \R^r$.   We recall a version of the classical  Schauder estimates as established in \cite{MR1135924}

\begin{proposition}\label{schauder}
Let $u\in C^{2,\alpha}_{\rm H,loc} (M)$ for some $\alpha \in (0,1)$. Let $a_{ij}\in C^{k,\alpha}_{\rm H,loc}(M)$. If $L_a u\in C^{k,\alpha}_{\rm H,loc}(M)$, then 
$u\in C^{k+2,\alpha}_{\rm H,loc}(M) $ and for every $U\subset \subset M$ there exists a positive constant $C=C(U,\alpha, k,X)$ such that
$$||u||_{C^{k+2,\alpha}(U)} \le C ||Lu||_{C^{k,\alpha}(M)}.$$
\end{proposition}

In a similar spirit, the Schauder estimates hold for any operator of the form $Lu= \sum_{i,j=1}^r a_{ij}(x)X_i^* X_j u$ where $X_i^*$ denotes the adjoint of $X_i$ with respect to some fixed smooth volume form.

Next, following an argument originally introduced by Agmon, Douglis and Nirenberg \cite[Theorem A.5.1]{ADN1} in the Euclidean setting,  we show that one can  lift the burden  of the a-priori regularity hypothesis from the Schauder estimates.
\begin{lemma}\label{ADN} Let  $\alpha \in (0,1)$ and assume that $u\in W^{2,2}_{\rm H,loc}
(M)\
$ is a function  that satisfies for a.e. $x\in M$ 
$$L_{A(x)} u(x)= \sum_{i,j=1}^r a_{ij}(x) X_i X_j u (x) \in C^{\alpha}_{\rm H,loc}(M).$$
If $a_{ij}\in C^{\alpha}_{\rm H,loc}(M)$, then $u$ is in fact a $C^{2,\alpha}_{\rm H,loc}(M)$ function.
\end{lemma}
\begin{proof}
The strategy in \cite{ADN1} consists in setting up a bootstrap argument through which the integrability of the weak second order derivatives $X_iX_j u$ of the solution increases until, in a finite number of steps, one achieves that they are continuous.  At this point ones invokes a standard extension of a classical result of Hopf  \cite{MR1545250}  or \cite[page 723]{ADN1} (for a proof  in the subelliptic setting see for instance  Bramanti et al., \cite[Theorem 14.4]{MR2604962}) which yields the last step in regularity, i.e., if $X_iX_j u$ are continuous then  $u\in C^{2,\alpha}_{\rm H,loc}.$

For a fixed $p_0\in M$ consider the {\it frozen coefficients} operator $$L_{A(p_0)}w=\sum_{i,j=1}^r a_{ij}(p_0)X_iX_j w.$$ For sake of simplicity we will write $L_p, L_{p_0}$ for $L_{A(p)},L_{A(p_0)}$.  Denote by $\Gamma_{p_0}(p,q)$ the fundamental solution of $L_{p_0}$. For fixed $r>0$, consider a smooth function $\eta\in C^{\infty}_0(B(p_0,2r))$ such that $\eta=1$ in $B(p_0,r)$. For any $p\in M$ and any smooth function $w$  one has
$$\eta(p)w(p)= \int \Gamma_{p_0}(p,q) L_{p_0}(\eta w)(q) \d \vol (q).$$
Differentiating the latter 
along two horizontal vector fields $X_i, X_j$  $i,j=1,\ldots ,r$ one obtains that   for any $p\in B(p_0,r)$ 
$$X_{i,p} u(p)= \int \bigg[X_{i,p} \Gamma_{p_0}(p,q) L_q(u\eta)  + X_{i,p}\Gamma_{p_0}(p,q) ( L_{p_0}-L_q)u\eta (q) \bigg]\d\vol(q),$$
and
$$X_{i,p} X_{j,p} u(p)= \int \bigg[X_{i,p} X_{j,p}\Gamma_{p_0}(p,q) L_q(u\eta)  + X_{i,p}X_{j,p}\Gamma_{p_0}(p,q) ( L_{p_0}-L_q)u\eta (q) \bigg]\d\vol(q)+  C(p_0)L_p(u\eta),$$
where $X_{i,p}$ denotes differentiation in the variable $p$ and $C$ is a H\"older continuous function arising from the principal value of the integral.

Setting $p=p_0$ one obtains the identity
\begin{equation}\label{a5.3} X_{i,p} X_{j,p} u(p)= \int \bigg[X_{i,p} X_{j,p}\Gamma_{p}(p,q) L_q(u\eta)  + X_{i,p}X_{j,p}\Gamma_{p}(p,q) ( L_{p}-L_q)u\eta (q) \bigg]\d\vol(q)+  C(p_0)L_p(u\eta),
\end{equation}
where the differentiation in the first term in the integrand is intended in the first set of the argument variables only. The next task is to show that identity  \eqref{a5.3} holds also for functions in $W^{2,2}_{\rm H}$, in the sense that the difference between the two sides has $L^2$ norm zero. To see this we consider a sequence of smooth approximations $w_n\to u\in W^{2,2}_{\rm H}$ in $W^{2,2}_{\rm H}$ norm. To guarantee convergence we observe that in view of the work in \cite{Roth:Stein} and \cite{nagelstwe}, the expression $ X_{i,p}X_{j,p}\Gamma_{p}(p,q) $ is a Calderon-Zygmund kernel. To prove our claim it is  then sufficient to invoke the boundedness between Lebesgue spaces  of  Calderon-Zygmund operators in the setting of  homogenous spaces (see \cite{MR0499948}), and \cite{Deng-Han}).

Our next goal is to show an improvement in the integrability of the second derivatives of the solution $u\in W^{2,2}_{\rm H}$. We write
$$X_{i,p} X_{j,p} u(p)= I_1+I_2+I_3+I_4$$
where 
$$I_1(p)=\int X_{i,p} X_{j,p} \Gamma_{p}(p,q) \eta(q) L_qu(q) \d\vol(q)+C(p) L_pu(p),$$
$$I_2(p)=\int X_{i,p} X_{j,p} \Gamma_{p}(p,q)  \sum_{i,j=1}^r a_{ij}(q) X_i \eta(q) X_j u(q)+ u(q) \sum_{i,j=1}^r a_{ij}(q) X_iX_j \eta(q) \d\vol(q),$$
$$I_3(p)=\int X_{i,p} X_{j,p} \Gamma_{p}(p,q) \sum_{i,j=1}^r (a_{ij}(p)-a_{ij}(q)) X_i X_j (\eta u) \d\vol(q).$$
Since $Lu\in C^\alpha$ and in view of the continuity of singular integral operators in H\"older spaces (see Rothschild and Stein \cite{Roth:Stein}) then $I_1\in C^{\alpha}$ and we can disregard this term in our argument.  

Next we turn our attention to $I_2$ and $I_3$.  Since $u\in W^{2,2}$ then Sobolev embedding theorem  \cite{HK} yields $\nabla_{\rm H} u\in L^{\frac{2Q}{Q-2}}_{loc}$ and as a consequence of the continuity of   Calderon-Zygmund operators in homogenous spaces one has $I_2\in L^{\frac{2Q}{Q-2}}$. 

In view of the estimates on the fundamental solution for sublaplacians by Nagel, Stein and Wainger  \cite{nagelstwe}, one has that 
$$|X_i X_j \Gamma_p (p,q)| \sup_{i,j} |a_{ij}(p)-a_{ij}(q)|  \le C(K)d(p,q)^{\alpha-Q},$$
for every $q\in K\subset \subset M$. One can then bound  $I_3$ with fractional integral operators $${\bf \mathcal I_\alpha}(\psi)(p):=\int d(p,q)^{\alpha-Q}\psi(q) d vol (q).$$ In the context of homogenous spaces (see for instance \cite{MR0499948}), these operators are bounded between the Lebesgue spaces $L^\beta\to L^\gamma$ with $\frac{1}{\beta}-\frac{1}{\gamma}=\frac{\alpha}{Q}$, whenever $1<\beta<\frac{\alpha}{Q}$. When $1+\frac{\alpha}{Q}>\beta>\frac{\alpha}{Q}$ one has that $\mathcal I_\alpha$ maps continuously $L^\beta$ into the Holder space $C^{\beta-\frac{\alpha}{Q}}_{\rm H}$.

In view of such continuity we infer that  $I_3\in L^{2\kappa}$ with $\frac{Q}{Q-2}>\kappa=\frac{Q}{Q-2\alpha} >1$. 

In conclusion, so far we have showed that if $u\in W^{2,2}_{\rm H,loc}(M)$ is a solution of $L_pu(p)\in C^{\alpha}_{\rm H}$ then  one has the integrability gain $u\in W^{2, 2 \frac{Q}{Q-2\alpha}}_{\rm H,loc}(M)$.  Iterating this process  for a finite number of steps, in the manner described in \cite[page 721-722]{ADN1}, one can increase the integrability exponent until it is larger than $\alpha/Q$ and at that point the fractional integral operators maps into a H\"older space and one finally has that $X_iX_ju$ are continuous. As described above, to complete the proof one now invokes  Bramanti et al., \cite[Theorem~14.4]{MR2604962}.

\end{proof}

\subsection{Subelliptic $Q$-Laplacian and $C^{\infty}$ estimates for non-degeneracy}\label{sec:QLap}
%
Denote by $Q$ the Hausdorff dimension of $M$. For $u\in W^{1,Q}_{\rm H,loc}(M)$, define the $Q$-Laplacian  $L_Q u$ by means of the following identity
\begin{equation}\label{QLap}
\int_M L_Q u  \,\phi \,\d\vol = \int_M |\nabla_{\rm H} u|^{Q-2} \langle \nabla_{\rm H} u, \nabla_{\rm H} \phi\rangle \d\vol, \qquad \text{ for any } \phi \in W^{1,Q}_{{\rm H},0}().
\end{equation}     If $u\in W^{2,2}_{\rm H,loc} (M)\cap W^{1,Q}_{\rm H,loc}(M)$ one can then  write almost everywhere in $M$\begin{equation}\label{QLap:forte}
L_Q u= X_i^*  (|\nabla_{\rm H} u|^{Q-2}X_i u).
\end{equation}

\begin{definition}
[$Q$-harmonic function]
Let $M $ be an equiregular sub-Riemannian manifold  of Hausdorff dimension $Q$. Fixed a measure $\vol$ on $M$, a function $u\in W^{1,Q}_{\rm H,loc}(M)$
is called {\em $Q$-harmonic} if
$$ \int_{M} |\nabla_{\rm H} u|^{Q-2} \langle \nabla_{\rm H} u, \nabla_{\rm H} \phi\rangle \d\vol =0, \qquad  \forall\phi\in W^{1,Q}_{{\rm H},0}(M).$$
\end{definition}

%
%
%
\begin{proposition}\label{smoothness} Let $M$ be an equiregular sub-Riemannian manifold  endowed with a smooth volume form $\vol$.  Let $u\in W^{1,Q}_{\rm H,loc}(M)$ be  a weak solution of $L_Qu=h$ in $M$,  with 
$h\in C^\alpha_{\rm H,loc}(M)$ and
$|\nabla_{\rm H}u|$  not vanishing in $M$. If 
$u\in C^{1,\alpha}_{\rm H,loc}(M)\cap W^{2,2}_{\rm H,loc}(M)$, then $u \in C^{2,\alpha}_{\rm H,loc}(M).$ 
\end{proposition}
\begin{proof} 
In coordinates, let $\omega\in C^\infty$ such that $\d\vol=\omega  d\mathcal{L}$, where $\mathcal{L}$ is the Lebesgue measure.
 Since  $u\in W^{2,2}_{\rm H,loc}(M)$,  then,  a.e.~in $M$, 
 the $Q$-Laplacian can be expressed in non-divergence form

\begin{equation}\label{non-div}
(L_Qu)(x)= \alpha_{ij}(x,\nabla_{\rm H} u)  X_i X_j u + g(x,\nabla_{\rm H} u) = h(x),
\end{equation}
where
 $$\alpha_{ij}(x,\xi)=- |\xi|^{Q-4} (\delta_{ij}+ (Q-2)) \xi_i \xi_j$$ and
 $$g(x,\xi)=-\omega(x)^{-1} X_i  \omega (x) |\xi|^{Q-2}\xi_i+ \p_kb_k^i(x) |\xi|^{Q-2} \xi_i.$$
%
Set $a_{ij}(x)=\alpha_{ij}(x,\nabla_{\rm H} u)$.
 Since $u\in C^{1,\alpha}_{\rm H,loc}(M)$, we have
$$a_{ij}(\cdot) \, \text{ and }\,  g(\,\cdot\,,\nabla_{\rm H} u)  \in C^\alpha_{\rm H,loc}(M).$$
 In view of the non-vanishing of $\nabla_{\rm H} u$, one can invoke  Lemma~\ref{ADN}, to obtain $u\in C^{2,\alpha}_{\rm H,loc}(M).$   
\end{proof}

\section{Definitions of $1$-quasiconformal maps}\label{defs1qc}

 In this whole section we prove Theorem \ref{theorem0} and the corollaries thereafter. In particular, we show the equivalence of the definitions \eqref{H} - \eqref{LP} of $1$-quasiconformal maps, and show how \eqref{EP} and \eqref{CP}  
are consequences. 
To help the reader, we provide the following road map. The nodes of the graph indicate the definitions in  Theorem \ref{theorem0}, the tags on the arrows are the labels of  Propositions, Corollaries and Remarks in the present section. 

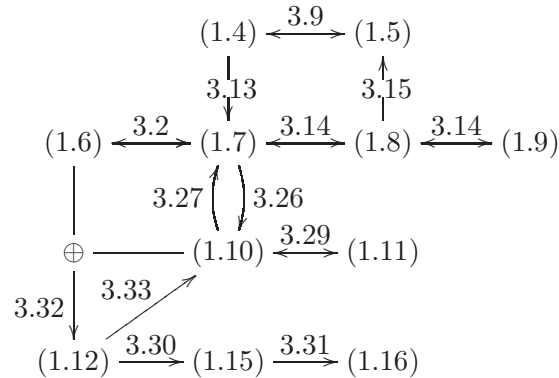
\begin{displaymath}
\xymatrix{
	&	\eqref{H} \ar[d]|{\,\,\ref{tan1qcsr}
}\ar@{<->}[r]^ {\ref{H_H=}}& \eqref{H=} 
	&	\\
	 \ar[r]^{\ref{rmk:similarity}}
\eqref{HS}
	\ar@{-}[d]   
	\ar@{<->}[r]&	\eqref{S}
	\ar@/^/[d]^{\ref{S_JP}}
	 \ar@{<->}[r]^ {\ref{L=LN}}  
&  \ar[u]|{\,\,\ref{L_H}} \eqref{L} \ar@{<->}[r]^ {\ref{L=LN}}& \eqref{LN}   	  \\
\ar[d]_{\ref{JP_LP}}
\ar@{-}[r]
\oplus
&	 \eqref{JP} 	\ar@/^/[u]^{\ref{JP_S}} 
 \ar@{<->}[r]^{{\ref{MP_L}} }&
  \eqref{MP} 	
  &	\\
 \eqref{LP} 
 \ar[ru]^{{\ref{LP_JP}}} 
 \ar[r]^{{\ref{rmk:E_I}}} & \eqref{EP}  \ar[r]^{\ref{EP_CP}} &		\eqref{CP} 
 &
 }
 \end{displaymath}

%
%
%
%
%
%
%
%
%
%
%
%

\subsection{Ultratangents of $1$-quasiconformal maps}

We refer 
 the reader who is not familiar with the notions of nonprincipal ultrafilters and ultralimits to Chapter 9 of Kapovich's book \cite{Kapovich_book}.
 Roughly speaking, taking ultralimits with respect to a nonprincipal ultrafilter is a consistent way of using the axiom of choice to select an accumulation point of any bounded sequence of real numbers.
Let $\omega$  be a nonprincipal ultrafilter.
Given a sequence $X_j$ of metric spaces with base points $\star_j\in X_j$, we shall consider the {\em based ultralimit metric space}
$$(X_\omega, \star_\omega):=(X_j, \star_j)_\omega := \lim_{j\to \omega} (X_j, \star_j).$$
We recall briefly the construction. Let 
$$X^\N_b:=\left\{ (p_j)_{j\in \N} : p_j\in  X_j , \sup\{ d(p_j, \star_j):j\in\N\}<\infty\right\}.$$
For all $ ( p_j ) _{j},  (q_j ) _{j}\in X^\N_b$, set
$$d_\omega ( ( p_j ) _{j},  ( q_j ) _{j}):= \lim_{j\to \omega} d_j(p_j, q_j),$$
where $ \lim_{j\to \omega}$ denotes  the $\omega$-limit of a sequence  indexed by $j$.
Then $X_\omega$ is the metric space obtained by taking the quotient of $(X^\N_b, d_\omega)$
by the semidistance $d_\omega$. We denote by $[p_j]$ the equivalence class of $ ( p_j ) _{j}$. The base point  $\star_\omega$ in $X_\omega$ is $[\star_j]$.

Suppose $f_j:X_j\to Y_j$ are 
maps between 
metric spaces, 
$\star_j\in X_j$
are base points,  and 
we have the property that
$  ( f_j(p_j) ) _j 
 \in Y^\N_b$,
for all 
$  ( p_j ) _j 
\in X^\N_b$.
 %
Then the ultrafilter $\omega$ assigns a limit map
$f_\omega := 
\lim_{j\to \omega} f_j :
(X_j, \star_j)_\omega 
\to
(Y_j, f_j(\star_j))_\omega $
as
$f_\omega  ([p_j]) := [f_j(p_j)]$.

Let $X$ be a metric space with distance $d_X$.
We fix a nonprincipal ultrafilter $\omega$, a base point $\star\in X$, and a sequence of positive numbers
$\lambda_j\to \infty$ as $j\to \infty$.
We define the {\em ultratangent at} $\star$ of $X$ as
$$T_\omega(X,\star):=\lim_{j\to \omega} (X,\lambda_j d_X, \star).$$
Moreover, given $f:(X,d_X)\to (Y,d_Y)$, we call the {\it ultratangent map} of $f$ at $\star$ the limit, whenever it exists, of the maps $f:(X,\lambda_jd_X,\star)\to(Y,\lambda_j d_Y,f(\star))$, denoted $T_\omega(f,\star)$.

\begin{lemma}\label{Lemma:tan1qc}
Let  $X$ and $Y$ be  geodesic metric spaces and let $ f:X\to Y$ be  a  quasiconformal map satisfying $H_f(\star) = 1$ at some point $\star \in X$.
Fix  a nonprincipal ultrafilter  $\omega$ and  dilations factors $\lambda_j\to \infty$.
If the ultratangent map $f_\omega=T_\omega(f,\star)$
exists, 
 then for
$p,q\in T_\omega(X,\star)$
$$
d(\star_\omega, p) = d(\star_\omega, q)\implies d(f_\omega(\star_\omega), f_\omega(p)) = d( f_\omega(\star_\omega), f_\omega( q)).$$
\end{lemma}

\proof
Take $p=[p_j], q=[q_j]\in T_\omega (X, \star)$ with 
$d(\star_\omega, p) = d(\star_\omega, q)=:R.$
Namely,
$$\lim_{j\to \omega} \lambda_j d(\star,p_j) = \lim_{j\to \omega} \lambda_j d(\star,q_j) = R.$$

Set $r_j:= \min\{  d(\star,p_j), d(\star,q_j)\}.$
Fix $j$ and suppose $r_j=d(\star,p_j)$ so  $r_j\leq d(\star,q_j)$. Since $Y$ is geodesic, there exists $q'_j\in X$ along a geodesic between $\star$ and $q_j$ with 
$$d (\star,q'_j) = r_j\quad \text{  and } \quad d (q_j,q'_j)  =d (\star,q_j) - r_j.$$ 

We claim that $[q'_j]=[q_j]$. Indeed,
\begin{eqnarray*}
d_\omega ([q'_j],[q_j]) &=&  \lim_{j\to \omega} \lambda_j d(q'_j,q_j)\\
&=& \lim_{j\to \omega} \lambda_j (  d (\star,q_j) - r_j )\\
&=& \lim_{j\to \omega} \lambda_j  d (\star,q_j)  -   \lambda_j  d (\star,p_j)\\
&=& R-R=0. 
\end{eqnarray*}
Reasoning similarly with $p_j$'s, we may conclude that $p=[p'_j]$ and $
q=[q_j']$ with
$  d(\star,p'_j)= d(\star,q'_j)=r_j.$
Hence, by definition of $f_\omega$ we have
$f_\omega  (p) = f_\omega  ([p'_j]) = [f_j(p'_j)]$
and
$f_\omega  (q) = f_\omega  ([q'_j]) = [f_j(q'_j)]$.
We then calculate 
\begin{eqnarray*}
\dfrac{d_\omega (f_\omega(\star_\omega), f_\omega(p)) }{d_\omega (f_\omega(\star_\omega), f_\omega(q))}
&=&\dfrac{\lim_{j\to \omega} \lambda_j d(f(\star), f(p_j')) }{\lim_{j\to \omega} \lambda_j d(f(\star), f(q'_j))}\\
&=&\dfrac{\lim_{j\to \omega}  d(f(\star), f(p_j')) }{\lim_{j\to \omega}  d(f(\star), f(q'_j))}\\
&\leq&\lim_{j\to \omega}\dfrac{\sup \{d(f(\star), f(a)) : 
d(\star, a)\leq
r_j\}}{\inf\{d(f(\star), f(b)) : 
d(\star, b)\geq
r_j \}}\\
&=&1.
\end{eqnarray*}
Arguing along the same lines one obtains ${d_\omega (f_\omega(\star_\omega), f_\omega(q))} \leq {d_\omega (f_\omega(\star_\omega), f_\omega(p)) }$ and hence the statement of the lemma follows.
\qed

\subsection{Tangents of quasiconformal maps in sub-Riemannian geometry}\label{sec:tange}
We recall now some known results due to Mitchell \cite{Mitchell} and Margulis, Mostow \cite{Margulis-Mostow}, which are needed to show that every $1$-quasiconformal map induces at almost every point a $1$-quasiconformal isomorphism of the relative ultratangents. For the sake of our argument, we rephrase their results using the convenient language of ultrafilters.

Let $M$ be an equiregular sub-Riemannian manifold. From \cite{Mitchell}, for every $p\in M$ the ultratangent $T_\omega(M,p)$ is isometric to a Carnot group, denoted $\mathcal N_p(M)$,  also called nilpotent approximation of $M$ at $p$.
Each horizontal vector of $M$ at $p$ has a natural identification with an horizontal vector of $\mathcal N_p(M)$ at the identity. 
Such identification is an isometry between the horizontal space $H_pM$ and the horizontal space of 
$\mathcal N_p(M)$ at the identity, both equipped with the  scalar products given by respective sub-Riemannian structures.
Next, consider $f:M\to N$ a quasiconformal map between equiregular sub-Riemannian manifolds $M$ and $N$. By the work of Margulis and Mostow \cite{Margulis-Mostow}, there exists at almost every $p\in M$ the ultratangent map $T_\omega(f,p)$ that is a group isomorphism
$$
{\mathcal N}_p(f): {\mathcal N}_p(M)\to {\mathcal N}_{f(p)}(N)
$$
that commutes with the group dilations, and
 it is independent on the ultrafilter $\omega$ and the sequence $\lambda_j$.
Part of Margulis and Mostow's result is that the map $f$ is almost everywhere differenziable along horizontal vectors. 
Hence, for almost every $p\in M$ and for all horizontal vectors $v$ at $p$, 
we can consider the
push-forwarded vector, which we
denote by 
$(\d_{\rm H} f)_p (v)$.
We call the map
$$(\d_{\rm H} f)_p :H_p M \to H_{f(p)}N$$
the {\em horizontal differential} of $f$ at $p$.

\begin{remark}\label{rmk:similarity}
With the above identification, we have 
\begin{equation}\label{horizontal-push-forward}
(\d_{\rm H} f)_p(v) = \mathcal N_p(f)_* v, \qquad \forall v\in H_pM,
\end{equation}
so
$(\d_{\rm H} f)_p$ is a restriction of 
$\mathcal N_p(f)_*$. Vice versa, 
$(\d_{\rm H} f)_p$ 
completely determines $\mathcal N_p(f)$, since $\mathcal N_p(f)$
is a homomorphism and 
$H_p M$ generates the Lie algebra of $\mathcal N_p(M)$.
In particular, 
$(\d_{\rm H} f)_p$ is a similarity if and only if 
$\mathcal N_p(f)$ is a similarity with same factor.
Hence, Conditions~\eqref{S} and \eqref{HS} are equivalent.
\end{remark}

Next we introduce some expressions that can be used to quantify the distortion.
\begin{align*}
\underline\L_f(p)&:= \liminf_{r \to 0}\frac{
\sup \{d_N(f(p),f(q)):  {d_M(p,q)\leq r}   \}}{r},\\
\overline\L_f(p)&:= \limsup_{r \to 0}\frac{
\sup \{d_N(f(p),f(q)):  {d_M(p,q)\leq r}   \}}{r},\\
\overline\L_f^=(p)&:= \limsup_{r \to 0}\frac{
\sup \{d_N(f(p),f(q)):  {d_M(p,q)= r}   \}}{r},\\
\underline\L_f^=(p)&:= \liminf_{r \to 0}\frac{
\sup \{d_N(f(p),f(q)):  {d_M(p,q)= r}   \}}{r},
\end{align*}
\begin{eqnarray*}
\norm{\mathcal N_p(f)}&:=&
\max \{d(e, \mathcal N_p(f) (y) ) : d_{\mathcal N_p(M)}(e,y)\leq1\}\\
&=& \max \{d(e, \mathcal N_p(f) (y) ) : d_{\mathcal N_p(M)}(e,y)=1\}.
\end{eqnarray*}
\begin{remark}\label{realization}
There exists a horizontal vector at $p$ such that $\norm{X}=1 $ and
$\norm{f_*X} =  
\norm{\mathcal N_p(f)}
$,
which in other words means 
that $X$ is in the first layer of the Carnot group 
${\mathcal N_p(M)}$,
$
d_{\mathcal N_p(M)}(e,\exp(X))=1$, and 
$d_{\mathcal N_{f(p)}(N)}(e, \mathcal N_p(f) (\exp(X)) ) =
\norm{\mathcal N_p(f)}
$.
\end{remark}

The following holds.
\begin{lemma}\label{uguaglianze}
Let  $M$ and $N$ be (equiregular) sub-Riemannian manifolds and let $ f:M\to N$ be  a  quasiconformal map. Let $p$ be  a point of differentiability for $f$. 
We have
$$\Lip_f(p)=\norm{\mathcal N_p(f)}=\Lip_{\mathcal N_p(f)}(e)=\overline\L_f(p)=\underline\L_f(p)=\overline\L_f^=(p)=\underline\L_f^=(p).$$

\end{lemma}
\begin{proof} 

Proof of $\Lip_f(p)\leq \norm{\mathcal N_p(f)}$.
Let $p_j\in M$ such that $p_j\to p$ and
$$\Lip_{ f}(p) = \lim_{j\to \infty} \dfrac{d(f(p), f(p_j))}{d(p,p_j)}.$$
Let $\lambda_j:=1/d(p,p_j)$, so $\lambda_j\to \infty$. We fix now any nonprincipal ultrafilter $\omega$ and consider ultratangents with respect to dilations $\lambda_j$.
Hence, 
\begin{align*}
\Lip_{ f}(p) &= \lim_{j\to \infty} \lambda_j d(f(p), f(p_j))
\\&= 
d_\omega([f(p)],  [f(p_j)])\\
&= 
d_\omega(\mathcal N_pf([p]),\mathcal N_pf( [p_j]))
\\
&
\leq
\norm{\mathcal N_p f}
d_\omega([p], [p_j])\\&=
\norm{\mathcal N_p f}\lim_{j\to\omega}\lambda_j d(p,p_j) \\&=\norm{\mathcal N_p f}
.
\end{align*}

Proof of $\Lip_f(p)\geq \norm{\mathcal N_p(f)}$.
  Take $y\in {\mathcal N_p(M)}$ with 
$d (e,y)= 1$
that realizes the maximum in $\norm{\mathcal N_p(f)}$.  
 Choose a sequence  $q_j\in M$ such that $[q _j]$
represents the point $ y$.
Let $\lambda_j\to \infty$ be the dilations factors for which we calculate the ultratangent. 
Since
$$1=d (e,y)=\lim_{j\to \omega} \lambda_j d (p,q_j),$$
then, up to passing to a subsequence of indices, $d(p,q_j)\to 0$. Moreover, 
\begin{align*}
\Lip_{ f}(p) &\geq 
\limsup_{j\to \infty} 
\dfrac{ d(f(p), f(q_j))}{d(p,q_j)}
\\&= 
\limsup_{j\to \infty} 
\lambda_j  d(f(p), f(q_j)) 
\\&= 
d_\omega([f(p)],  [f(q_j)])\\
&= 
d_\omega(e,\mathcal N_pf( y))
\\
& =\norm{\mathcal N_p f}
.
\end{align*}


Proof of $\overline\L_f(p)\leq \norm{\mathcal N_p(f)}$.
There exists $r_j\to 0$ and $p_j\in M$ with $d_M(p,p_j)\leq r_j$ such that
$$
\overline\L_f(p)=\lim_j \frac{d_N(f(p),f(p_j))}{r_j}.
$$
 Then, using $1/r_j$ as scaling for the ultratangent, we have
$d_\omega ([p], [p_j] ) \leq 1$ and
$\overline\L_f(p)=\lim_j \frac{1}{r_j}d_N(f(p),f(p_j))
= d_\omega ([f(p)], [f(p_j)] ) 
\leq \norm{\mathcal N_p(f)}$.
 
Proof of $\norm{\mathcal N_p(f)}\leq \underline\L_f(p)$. Take $y\in {\mathcal N_p(M)}$ with 
$d_{\mathcal N_p(M)}(e,y)\leq1$
that realizes the maximum in $\norm{\mathcal N_p(f)}$.  
Choose subsequences $s_j\to 0 $ 
that realizes the limit in the definition of $\underline\L_f(p)$, i.e., 
so that
$$
\underline\L_f(p)=\lim_j 
\frac{
\sup \{d_N(f(p),f(q)):  {d_M(p,q)\leq s_j}   \}}{s_j}.$$


We use $1/s_j$ as scaling factors for the ultratangent space. 
For any $\mu\in (0,1)$ choose a sequence  $q_j\in M$ such that $[q _j]$
represents the point $\delta_\mu(y)$. Therefore,
we have that
$$
\lim_j \frac{d_M(p,q_j)}{s_j}
=d_{\mathcal N_p(M)}(e,\delta_\mu(y))\leq \mu d_{\mathcal N_p(M)}(e,y)\leq \mu <1.
$$
For $j$ big enough we then have $d_M(p,q_j)<s_j$. So 
$$
d_N(f(p),f(q_j))\leq \sup \{d_N(f(p),f(q)): d_M(p,q)\leq s_j\},
$$
whence, dividing both sides by $s_j$ and letting $j\to \infty$, we get
$$d_\omega ( [f(p)] ,[f(q_j)] )  \leq \underline\L_f(p), $$
which, in view of the homogeneity of  $\mathcal N_p(f)$, yields  
$$
\mu \norm{\mathcal N_p(f)} =d_{\mathcal N_p(M)}(e,\mathcal N_p(f)(\delta_\mu y))\leq \underline\L_f(p).
$$
Since the last inequality holds for all $\mu \in (0,1)$, the conclusion follows.


 Proof of $ \underline\L_f(p)\geq \overline\L_f^=(p).$
  Since $$     \sup \{d_N(f(p),f(q)):  {d_M(p,q)\leq r}   \} \geq \sup \{d_N(f(p),f(q)):  {d_M(p,q)= r}   \},$$ one has
\begin{eqnarray*}
\underline\L_f(p)&=&
\liminf_{r \to 0} \frac{\sup \{d_N(f(p),f(q)):  {d_M(p,q)\leq r}   \}}{r}\\
&\geq& \limsup_{r\to 0} \frac{\sup \{d_N(f(p),f(q)):  {d_M(p,q)= r\}}}{r} = \overline\L_f^=(p).
\end{eqnarray*}

 Proof of $ \underline\L_f(p)\leq\overline\L_f^=(p).$
Choose a  sequence $r_j\to 0$ such that 
$$
  \frac{
\sup \{d_N(f(p),f(q)):  {d_M(p,q)\leq r_j}   \}}{r_j}=\frac{
\sup \{d_N(f(p),f(q)):  {d_M(p,q)= r_j}   \}}{r_j},
$$
and so in particular 
\begin{eqnarray*}
\underline\L_f(p)
&=&
 \liminf_j \frac{
\sup \{d_N(f(p),f(q)):  {d_M(p,q)\leq r_j}   \}}{r_j}\\
&\leq&\limsup_j \frac{
\sup \{d_N(f(p),f(q)):  {d_M(p,q)= r_j}   \}}{r_j}\leq \overline\L_f^=(p).
\end{eqnarray*}

 Proof of $ \norm{\mathcal N_p(f)}\leq\underline\L_f^=(p).$
  Take $y\in {\mathcal N_p(M)}$ with 
$d (e,y)= 1$
that realizes the maximum in $\norm{\mathcal N_p(f)}$.  
Choose subsequences $s_j\to 0 $ 
that realizes the limit in the definition of $\underline\L_f^=(p)$, i.e., 
so that
$$
\underline\L_f^=(p)=\lim_j 
\frac{
\sup \{d_N(f(p),f(q)):  {d_M(p,q)= s_j}   \}}{s_j}.$$
We use $1/s_j$ as scaling factors for the ultratangent space. 
For any $\eps>0 $ choose a sequence  $q'_j\in M$ such that $[q' _j]$
represents the point $\delta_{1+\eps}(y)$. Therefore,
we have that
$$
1+\eps=d(e, \delta_{1+\eps}(y))=
\lim_{j\to \omega} \frac{d(p,q'_j)}{s_j}.
$$
For $j$ big enough we then have $d(p,q'_j)\in (s_j, (1+2\eps)s_j)$. Since $M$ is a geodesic space, we consider a point $q''_j\in M$ such that
$d(p,q''_j)= s_j$ and lies in the geodesic between $p$ and $q'_j$, consequently 
$d(q'_j,q''_j)\leq 2\eps s_j$.

Set $y_\eps\in  {\mathcal N_p(M)}$ the point being represented by the sequence $q''_j$.
We have $d(\delta_{1+\eps} y , y_\eps) <2\eps$.
From which we get that $y_\eps \to y$, as $\eps\to 0$.
We then bound
$$\underline\L_f^=(p)
\geq 
\lim_j 
\frac{
 d(f(p),f(q''_j)) }{s_j}
 = d (\mathcal N_p(f)(y_\eps), e).
 $$
 Since  $d (\mathcal N_p(f)(y_\eps), e)$ is continuous at $\eps=0$ and converges to $\norm{\mathcal N_p(f)}$, as $\eps\to 0$, we obtain the desired estimate.

  To conclude the proof of the proposition, one observes that 
 $\underline\L_f(p)\leq \overline\L_f(p)$  and  $\underline\L_f^=(p)\leq \overline\L^=_f(p)$  are trivial. 
\end{proof}

\begin{corollary}\label{L=LN:first}
Let  $M$ and $N$ be (equiregular) sub-Riemannian manifolds and let $ f:M\to N$ be  a  quasiconformal map. Let $p$ be  a point of differentiability for $f$. 
We have
\begin{equation}\label{eq:L=LN}
 \Lip_f(p)= 
 \Lip_{\mathcal N_p(f)}(e)
 \qquad \text{and} \qquad 
\ell_f(p)=   \ell_{\mathcal N_p(f)}(e).
\end{equation}
\end{corollary}
\begin{proof}
The proof follows from Lemma  \ref{uguaglianze} applied to $f$ and $f^{-1}$, and by observing that

\begin{equation}\label{Lvsl}
\ell_f(p) = 1/ \Lip_{f^{-1}}(f(p)), \text{ and } \mathcal N_p(f)^{-1}=\mathcal N_{f(p)}(f^{-1}).
\end{equation}

\end{proof}

\begin{corollary}\label{H_H=}
Let  $M$ and $N$ be (equiregular) sub-Riemannian manifolds and let $ f:M\to N$ be  a  quasiconformal map. Then for almost every $p\in M$
$$H_f(p)=H^=_f(p)
.$$
\end{corollary}
\begin{proof}
%
Note that in every geodesic metric space
$$\inf \{d_N(f(p),f(q)) : {d_M(p,q)\geq r}\}= 
{
\inf \{d_N(f(p),f(q)) : {d_M(p,q)= r}\}}.$$
Hence $H_f(p)\geq H^=_f(p)$
is immediate.

Regarding the opposite inequality, let $p$ be  a point of differentiability for $f$.
Consequently,
\begin{align*}
H_f(p)&\stackrel{def}{=}
\limsup_{r \to 0}\frac{
\sup \{d_N(f(p),f(q)):  {d_M(p,q)\leq r}   \}}
{
\inf \{d_N(f(p),f(q)) : {d_M(p,q)\geq r}\}}\\
&= 
\limsup_{r \to 0}\frac{
\sup \{d_N(f(p),f(q)):  {d_M(p,q)\leq r}   \}}
{
\inf \{d_N(f(p),f(q)) : {d_M(p,q)= r}\}}\\
&\le \limsup_{r \to 0}\frac{
r    }
{
\inf \{d_N(f(p),f(q)) : {d_M(p,q)= r}\}} \limsup_{r\to 0}\dfrac{\sup \{d_N(f(p),f(q)):  {d_M(p,q)\leq r\}}}{r}
\\
&= \limsup_{r \to 0}\frac{
r    }
{
\inf \{d_N(f(p),f(q)) : {d_M(p,q)= r}\}} 
\overline{\mathcal{ L}}_f(p)
\\
&= \limsup_{r \to 0}\frac{r    }
{
\inf \{d_N(f(p),f(q)) : {d_M(p,q)= r}\}} 
\underline{
\mathcal L}_f^=(p)\\
&= \limsup_{r \to 0}\frac{
r    }
{
\inf \{d_N(f(p),f(q)) : {d_M(p,q)= r}\}} \liminf_{r\to 0}\dfrac{\sup \{d_N(f(p),f(q)):  {d_M(p,q)= r\}}}{r}
\\
&\leq\limsup_{r \to 0}\frac{
\sup \{d_N(f(p),f(q)):  {d_M(p,q)= r}   \}}
{
\inf \{d_N(f(p),f(q)) : {d_M(p,q)= r}\}} \\
&\stackrel{def}{=} H_f^=(p),
\end{align*}
where in the last two steps we have used that $\overline{\mathcal{ L}}_f(p)=
\underline{
\mathcal L}_f^=(p)$ from Lemma \ref{uguaglianze} and the fact that $\limsup a_j \liminf b_j\leq \limsup(a_jb_j)$. 
\end{proof}

 \begin{proposition}\label{mug}
 Let $f:M\to N$ be a quasiconformal map between sub-Riemannian manifolds. 
 The function $p\mapsto \norm{\mathcal N_p(f)}$ is the minimal upper-gradient  of $f$.
 \end{proposition}
\begin{proof}
The function $p\mapsto \norm{\mathcal N_p(f)}$ is an upper-gradient of $f$ since 
 $\Lip_f(\cdot)$ 
 is such and 
$\Lip_f(p)=\norm{\mathcal N_p(f)}$ by Lemma \ref{uguaglianze} . 
Regarding the minimality, let $g$ be a weak upper-gradient of $f$.
We need to show that
\begin{equation}\label{g_maggiore_norm}
g(p)\geq \norm{\mathcal N_p(f)}, \qquad \text{ for almost all } p.
\end{equation}
Localizing, we take a unit horizontal vector field $X$.
For $p\in M$, let $\gamma_p$ be the curve defined by the flow of $X$, i.e.,
$$\gamma_p(t) :=\Phi_X^t(p),$$
which is defined for $t$ small enough.
We remark that the subfamilies of $\{\gamma_p\}_{p\in M}$ that have zero $Q$-modulus are of the form 
  $\{\gamma_p\}_{p\in E}$ with $E\subset M$ of zero $Q$-measure.
  Then, for every unit horizontal vector field $X$, there exists a set $\Omega_X\subseteq M$ of full measure such that for all $p\in \Omega_X$ we have
  $$\int_{\gamma_p|_{[0,\eps]}} g \geq d(f(\gamma_p(0)), f(\gamma_p(\eps))).$$
Since $\norm{X}\equiv1$, then each $ \gamma_p$ is parametrized by arc length. Thus
  $$\dfrac{1}{\eps} \int_0^\eps  g (\gamma_p(t)) \d t \geq  \dfrac{1}{\eps}    d(f(p), f(\Phi_X^\eps(p))).$$
 Assuming that $p$ is a Lebesgue point for $g$, taking the limit as $\eps \to 0$, and considering ultratangents with dilations 
$1/{\eps}$,
we have
\begin{eqnarray}
g(p ) &\geq& d_\omega (e , \mathcal N_p(f) [   \Phi_X^\eps(p) ]),  \nonumber
\\
&=& d_\omega (e , \mathcal N_p(f) \exp(\tilde X_p)), \qquad \forall p\in \Omega_X,
\label{g_d_omega}	
\end{eqnarray}
where $\tilde X_p$ is the vector induced on $ \mathcal N_p(M)$ by $   X_p$.

Set now $X_1, \ldots X_r$ an orthonormal frame of $\Delta$ and consider for all $\theta \in \mathbb S^{r-1}\subset \R^r$, the
unit horizontal
 vector field
$X^\theta := \sum_{i=1}^r \theta_i X_i$.
Fix $\{\theta_j\}_{j\in \N}$ a countable dense subset of 
$\mathbb S^{r-1}$
and define 
$\Omega:= \cap_j \Omega_{X^{\theta_j}}$, which has full measure.
Take $p\in \Omega$ and, recalling Remark \ref{realization}, take   $Y\in \Delta_p$ such that $\norm{Y}=1$ and 
$$ d_\omega (e , \mathcal N_p(f) \exp (\tilde Y))=     \norm{\mathcal N_p(f)}$$
By density, there exists a sequence $j_k$ of integers such that $\theta_{j_k}$ converges to some $\theta$ with the property that 
$Y=(X^\theta)_p$.
Therefore, by \eqref{g_d_omega} we conclude \eqref{g_maggiore_norm}.
\end{proof}

\subsection{Equivalence of metric definitions}

\begin{proposition}[Tangents of $1$-QC maps]\label{tan1qcsr}
Let $f:M\to N$ be  a  quasiconformal map between equiregular sub-Riemannian manifolds.
  Condition~\eqref{H} implies Condition~\eqref{S}.


\end{proposition}

\begin{proof}
For almost every $p\in M$, the  map $\mathcal N_p (f)$ exists and coincides with the ultratangent $f_\omega$ with respect to any nonprincipal ultrafilter and any sequence of dilations. Hence, we can apply Lemma \ref{Lemma:tan1qc} and deduce that spheres about the origin are sent to spheres about the origin. Therefore, the distortion $H_{\mathcal N_p (f)}(e)$ at the origin is $1$. Being $\mathcal N_p (f)$ an isomorphism,  the distortion is $1$ at every point, and in fact  $\mathcal N_p (f)$ is a similarity.

%
\end{proof}

\begin{corollary}\label{L=LN}
Let $f:M\to N$ be  a  quasiconformal map between equiregular sub-Riemannian manifolds.
Conditions~\eqref{S}, \eqref{L}, and \eqref{LN} are equivalent. 
\end{corollary}
\begin{proof}
For every point $p$ of differentiability for $f$, we have that 
$\mathcal N_p(f)$ is a similarity
if and only if 
$
 \Lip_{\mathcal N_p(f)}(e)=
 \ell_{\mathcal N_p(f)}(e),$ which by Corollary \ref{L=LN:first} is equivalent to   $\ell_f(p)=\Lip_f(p)$.
 \end{proof}

%
%


\begin{proposition}\label{L_H} Let $f:M\to N$ be a quasiconformal map between sub-Riemannian manifolds.
At every point $p\in M$ such that $\Lip_f(p)=\ell_f(p)$  one has that  $H^=_f(p)=1$.
Hence, Condition~\eqref{L} implies
Conditions  \eqref{H=}. 
\end{proposition}
\proof
Notice  that at every point in which $\Lip_f(p)=\ell_f(p)$ one has the existence of the limit  $$\lim_{ d(p,q)=r\to 0} \frac{d(f(p),f(q))}{r}.$$
Consequently, at those points one has  
\[H^=_f(p) =\lim_{r \to 0}
\frac{ \frac{
\sup \{d_Y(f(p),f(q)):  {d_X(p,q)= r}   \}}
{r}}{\frac{
\inf \{d_Y(f(p),f(q)) : {d_X(p,q)= r}\}}{r}} = \dfrac{\Lip_f(p)}{\ell_f(p)}=1.\qedhere\]

Therefore, we proved the equivalence of the metric definitions, i.e., Conditions
\eqref{H}, \eqref{H=},
  \eqref{S}, \eqref{L}, and \eqref{LN}. 

\subsection{Jacobians and Popp measure}\label{jacobians}
Let $(M,\mu_M)$ and  $(N,\mu_N)$ be metric measure spaces and let $f:M\to N$ be a homeomorphism.
We say that 
$\Jac_f:M\to \R$ is a {\em Jacobian} for $f$ with respect to the measures $\mu_M$ and $\mu_N$, if 
$f^* \mu_N = \Jac_f \mu_M$, which is equivalent to
the
change of variable formula:
\begin{equation}\label{change of variable}
\int_{f(A)} h \d \mu_N = 
\int_{ A} (h\circ f) \Jac_f \d \mu_M,
\end{equation}
for every $A\subset M$ measurable and every continuous function $h:N\to \R$.


If $M$ and $N$ are 
equiregular sub-Riemannian manifolds of Hausdorff dimension $Q$, we consider 
 $\mu_M$ and $\mu_N$ to be both either  the 
 $Q$-dimensional spherical Hausdorff measures 
 or the 
   Popp measures. 
See \cite{Montgomery,Barilari-Rizzi} for the definition of the Popp measure and  Example \ref{Ex:Popp_Carnot_2} for the case of step-2 Carnot groups.   
In these cases, we denote the corresponding Jacobians as
$\Jac_f^{\rm Haus}$ and $\Jac_f^{\rm Popp}$, respectively.
If $f$ is a quasiconformal map, such Jacobians  are
uniquely determined up to sets of measure zero.
In fact, by Theorem \cite[Theorem~4.9, Theorem~7.11]{Heinonen-Koskela} and \cite[Theorem~7.1]{Margulis-Mostow}, they can be espressed as volume derivatives.
Moreover,
by an elementary  calculation using just the definition one checks that
 the Jacobian satisfies the formula
\begin{equation}\label{Jac:inv}
\Jac_f(p) = 1/ \Jac_{f^{-1}}(f(p)).
\end{equation}

\begin{remark}\label{sameTangents}
We have that if 
$f:M\to N$ is   quasiconformal
and at almost every point $p$ its differential $\mathcal N_p(f)$   is a similarity, then  for almost every $p
\in M$ the Carnot groups
$
{\mathcal N}_p(M) $ and $ {\mathcal N}_{f(p)}(N)$
are isometric. Indeed,
if $\lambda_p$ is the dilation factor of  $\mathcal N_p(f)$,
then
the composition of 
${\mathcal N}_p(f)$
and the group dilation by $\lambda_p^{-1}$ gives an isometry.
As a consequence, $
{\mathcal N}_p(M) $ and $ {\mathcal N}_{f(p)}(N)$ are isomorphic as metric measure spaces when equipped with their Popp measures
$\vol_{ {\mathcal N}_{p}(M) }$ and $ \vol_{ {\mathcal N}_{f(p)} (N) }$, respectively.
In particular, for almost every $p
\in M$,
we have
\begin{equation}\label{54}
\vol_{ {\mathcal N}_{p}(M) }(B_{{\mathcal N}_{p}(M)} (e,1)) = \vol_{ {\mathcal N}_{f(p)} (N) }(B_{{\mathcal N}_{f(p)} (N)} (e,1))    .
\end{equation}
\end{remark}

\begin{example}\label{Ex:Popp_Carnot_2}
We recall in a simple case  the construction of the Popp measure.
Namely, we consider a Carnot group of step $2$, that is,  the Lie algebra is stratified as $V_1\oplus V_2$. 
Let $B\subseteq V_1\subseteq T_eG$ be the (horizontal) unit ball with respect to a sub-Riemannian metric tensor $g_1$ at the identity, which is the intersection of  the metric unit ball at the identity with $V_1$, in exponential coordinates.
The set $[B,B]:=\{[X,Y]:X,Y\in B\}$ is the unit ball of a unique scalar product $g_2$ on $V_2$.
The formula $g:=\sqrt{g_1^2+g_2^2}$ defines the unique  scalar product on $V_1\oplus V_2$ that make $V_1$ and $ V_2$ orthogonal and extend $g_1$ and $g_2$. 
Extending the scalar product on $T_e G$ by left translation,
one obtains a Riemannian metric tensor
$\tilde g$  on the Lie group $G$.
For such a Carnot group the Popp measure is by definition the Riemannian volume measure of $\tilde g$.
\end{example}

\begin{remark}\label{Popp_monotone} 
In Carnot groups the Popp measure is strictly monotone as a function of the distance, in the sense that if $d$ and $d'$ are two distances on the same Carnot group such that $d'\leq d$ and $d'\neq d$, 
then 
${\rm Popp}_{d'}\leq{\rm Popp}_{d}$ 
and
${\rm Popp}_{d'}\neq{\rm Popp}_{d}$.
Indeed, this claim follows easily from the construction of the measure. 
For simplicity of notation, we illustrate the proof
for Popp measures in
Carnot groups of step $2$
as we recalled  in Example \ref{Ex:Popp_Carnot_2}. 
%
If $B'$ is a set that   strictly contains $B$ then clearly  
 $[B,B]\subseteq  [B',B']$ and hence the unit ball for $g$ is strictly contained in the unit ball for $g'$.
 In other words, 
 the vector space $T_e G$ is equipped with two different (Euclidean) distances, say $\rho$ and $\rho'$, and by assumption, the identity
 $\id:(T_eG,\rho)\to (T_eG,\rho')$ is $1$-Lipschitz.
 Therefore, the Hausdorff measure with respect to $\rho$ is greater than the one with respect to $\rho'$.
At this point we recall that  the Hausdorff measure of a Eulidean space equals the Lebesque measure with respect to orthonormal coordinates.
In other words,  the Hausdorff measure
is equal to the measure induced by the top-dimensional form that takes value $1$ on any orthonormal basis, which is by definition the Riemannian volume form.
We therefore deduce that 
 the Riemannian volume measure of $\tilde g$ is less than 
 the Riemannian volume measure of $\tilde g'$.
 Hence,
   ${\rm Popp}_{d'}\leq{\rm Popp}_{d}$. Moreover, the equality holds only if 
   $\tilde g=\tilde g'$, which holds if and only if 
   $B'=B$.    
\end{remark} 

\begin{lemma}
\label{tangJP_S} 
Let $A:G\to G'$ be an isomophism of Carnot groups of Hausdorff dimension $Q$.
If either $\Jac_{A} (e)=(\Lip_A(e))^Q$ or $\Jac_{A} (e)=(\ell_A(e))^Q$,
then $A$ is a similarity.
\end{lemma}
\begin{proof}
Up to composing $A$ with a dilation, we assume that $\Lip_A(e)=1$, i.e., $A$ is $1$-Lipschitz. 
Then if  $\Jac_{A} (e)=(\Lip_A(e))^Q$ we have that 
$\Jac_{A} =1$, which means that the 
push forward via $A$ of the 
Popp measure on $G$ is the 
Popp measure on $G'$. 
Moreover, identifying the group structures via $A$, we assume that 
we are in the same group $G$ (algebraically) that is equipped with two different Carnot distances $d$ and $d'$ such that $d'\leq d$, since the identity $A=\id:(G,d)\to (G,d')$ is $1$-Lipschitz.
If $d'\neq d$, then 
by Remark \ref{Popp_monotone}
${\rm Popp}_{d'}\neq{\rm Popp}_{d}$, which contradicts the assumption.
We conclude that $d'= d$, i.e., $A=\id$ is an isometry.
The case when $\Jac_{A} (e)=(\ell_A(e))^Q$ is similar.  
\end{proof}

\subsection{A remark on tangent volumes}


We prove
that the Jacobian of a quasiconformal map coincides with the Jacobian of its tangent map 
almost everywhere.
We begin by recalling
the Margulis and Mostow's convergence \cite{Margulis-Mostow}.
Fix a point $p$ in a sub-Riemannian manifold $M$ and consider privileged  coordinates centered at $p$, see \cite[page 418]{Margulis-Mostow}.
Let 
$g$ be the sub-Riemannian metric tensor of  $M$.
Let 
$\delta_\eps$ be the  dilations associated to the privileged coordinates.
Notice that
$(\delta_\eps)_*g  $ is isometric via $\delta_\eps$ to $g$
and
$g_\eps := \frac{1}{\eps}(\delta_\eps)_*g  $ is isometric via $\delta_\eps$ to $\frac{1}{\eps} g$.
A
key fact is that $g_\eps$ converge to $g_0$, as $\eps\to 0$, which is a sub-Riemannian metric. (This convergence is the convergence of some orthonormal frames uniformly on compact sets).

Mitchell's theorem \cite{Mitchell} can be restated as the fact that
$(\R^n, g_0)$ is the tangent Carnot group ${\mathcal N}_p(M)$.
Margulis and Mostow  actually proved  that the maps
$\delta_\eps^{-1}\circ f \circ \delta_\eps$ converge uniformly, as $\eps\to 0$,  on compact sets to the map ${\mathcal N}_p(f)$.
Moreover, 
by functoriality of the  construction  of the Popp measure, we have that
$\vol^{g_\eps} \to \vol^{g_0} $, in the sense that if $\omega_\eps$ is the smooth function such that 
$\vol^{g_\eps} = \omega_\eps \mathcal L$,  then $\omega_\eps \to \omega_0$
 uniformly on compact sets.

\begin{proposition}\label{Jac=NJac}
Let $f:M\to N$ be  a   quasiconformal map between equiregular sub-Riemannian manifolds of Hausdorff dimension $Q$. For almost every $p\in M$
\begin{align*}
 \Jac_{{\mathcal N}_p(f)} (e) =\Jac_{f} (p).
\end{align*}
\end{proposition}

\proof
Denote by $B_r^{g_\eps}$   the ball at $0$ of radius $r$ with respect to the metric $g_\eps$.
We have
\begin{eqnarray}\label{18354228}
\eps^{-Q}    \vol^g (f(B_\eps^g))&=&
\vol^{\frac{1}{\eps} g} (f(B_1^{\frac{1}{\eps}g})) \\
\nonumber
&=&\vol^{\frac{1}{\eps} g} (    \delta_\eps  \circ \delta_\eps^{-1}\circ f \circ \delta_\eps (B_1^{g_\eps})\\
\nonumber&=&\vol^{  g_\eps} (       \delta_\eps^{-1}\circ f \circ \delta_\eps (B_1^{g_\eps}))\\
\nonumber &\to& 
\vol^{g_\infty} ({\mathcal N}_p(f)(B_1^{g_0})).
\end{eqnarray}
By \cite[Lemma 1 (iii)]{ghezzi-jean}, for all $q\in M$ we have the expantion
\begin{equation}\label{1828}
\vol_M(B(q,\epsilon))=\epsilon^Q \mathcal \vol_{\mathcal N_q(M)}(B_{\mathcal N_q(M)}(e,1))+o(\epsilon^Q).
\end{equation}
Using \eqref{18354228} and the latter, we conclude
\begin{eqnarray*}
\Jac_{{\mathcal N}_p(f)} (e) &= &
 \frac{ \vol_{\mathcal N_q(M)}(  N_q(f)( B_{\mathcal N_q(M)}(e,1)))}
 {  \vol_{\mathcal N_q(M)}(B_{\mathcal N_q(M)}(e,1)) }
 \\
\nonumber &=& \lim_{\eps \to 0} 
\frac{\vol_N(f(B(p,\eps)))}{ \eps^{Q}\mathcal \vol_{\mathcal N_p(M)}(B_{\mathcal N_p(M)}(e,1))}\\
 \nonumber
 &=& \lim_{\eps \to 0} 
\frac{\vol_N(f(B(p,\eps)))}{\vol_M(B(p,\eps))}\\
 \nonumber &=&
 \Jac_{f} (p).
\end{eqnarray*}
\qed

%
%
%
%
%
%
%

\subsection{Equivalence of the analytic definition}

\begin{lemma}\label{S_JP}  
Let $f:M\to N$ be  a   quasiconformal map between equiregular sub-Riemannian manifolds of Hausdorff dimension $Q$. 
If the differential $\mathcal N_p(f)$ of $f$ is a similarity for almost every $p\in M$, then
\begin{align*}
&\ell_f(p)^Q= 
\Jac_f^{\rm Popp}(p)= \Lip_f(p)^Q, \qquad \text{ for almost every } p\in M.
\end{align*}
\end{lemma}
\begin{proof}	
Let $p$ be a point where $\Jac^{\rm Popp}_f(p)$ is expressed as volume derivative. By definition, for all $\epsilon >0$, there exists $\bar r>0$ such that, if $q\in M$ is such that $d(q,p)\in (0,\bar r)$, then 
$$
 \frac{d(f(q),f(p))}{d(p,q)}\leq \Lip_f(p) + \epsilon.
$$
Hence for every $r\in (0,\bar r)$,
\begin{equation*}\label{1727}
 f(B(p,r))\subset B(f(p),r(\Lip_f(p) + \epsilon)).
\end{equation*}
So  one has
$$
 \frac{\vol_N(f(B(p,r)))}{\vol_M(B(p,r))}\leq \frac{\vol_N(B(f(p),r(\Lip_f(p) + \epsilon)))}{\vol_M(B(p,r))}.
$$
Letting $r\to 0$, using \eqref{1828} with $q=p$ and $q=f(p)$, and using \eqref{54}, we have
$$
\Jac_f^{\rm Popp} (p)\leq (\Lip_f(p) + \epsilon)^Q.
$$
Notice that equation~\eqref{54}  requires the assumption of the differential being a similarity.
Since $\epsilon$ is arbitrary,  $\Jac_f^{\rm Popp}(p)\leq \Lip_f(p)^Q$. 
Once we recall that 
$\Jac^{\rm Popp}_f(p)\cdot \Jac^{\rm Popp}_{f^{-1}}(f(p))=1$ and $\ell_f(p)\cdot \Lip_{f^{-1}}(f(p))=1$,
the same argument  applied to $f^{-1}$ yields 
$\ell_f(p)^Q\leq \Jac_f^{\rm Popp}(p)$.
With Corollary \ref{L=LN} we conclude.
\end{proof} 

For an arbitrary 
 quasiconformal map 
we expect the relation
\begin{align*}
&\ell_f(p)^Q\leq \Jac_f^{\rm Popp}(p)\leq \Lip_f(p)^Q
\end{align*}
to hold. However, our proof of Lemma \ref{S_JP} makes a crucial use of 
equation \eqref{54}, which is not true in general.

\begin{lemma}\label{JP_S}  
Let $f:M\to N$ be  a   quasiconformal map between equiregular sub-Riemannian manifolds of Hausdorff dimension $Q$.
 If for almost every $p\in M$
either
\begin{align*}
&\ell_f(p)^Q= \Jac_f^{\rm Popp}(p) 
\end{align*}
or
\begin{align*} 
 \Jac_f^{\rm Popp}(p)= \Lip_f(p)^Q,
\end{align*}
then 
${\mathcal N_p(f)}$ is a similarity, for almost every $p$.
\end{lemma}
\begin{proof}
In view of Proposition~\ref{Jac=NJac} and Corollary \ref{L=LN:first}, we have either 
$\ell_{\mathcal N_p(f)}(e)^Q= \Jac_{\mathcal N_p(f)}(e)$
or $
 \Jac_{\mathcal N_p(f)}(e)= \Lip_{\mathcal N_p(f)}(e)$.
 Therefore, by
Lemma \ref{tangJP_S}
we get that ${\mathcal N_p(f)}$ is a similarity.
\end{proof}


\subsection{Equivalence of geometric  definitions}\label{sec:geodef}

We recall the definition of the  modulus of  a family 
  $\Gamma$ of curves in a metric measure space $(M, \vol)$.  A Borel function $\rho\colon M\rightarrow [0,\infty]$ is said to be \textit{admissible} for $\Gamma$ if for every rectifiable $\gamma\in \Gamma$,
\begin{equation}
\label{admissibility}
\int_\gamma \rho\,ds\geq 1\text{.}
\end{equation}
The \textit{$Q$-modulus} of $\Gamma$ is 
\begin{equation*}
{\rm Mod}_Q(\Gamma) = \inf \left\{ \int_M \rho^Q\,\d\vol:\text{$\rho$ is admissible for $\Gamma$} \right\} \text{.}
\end{equation*}

\begin{proposition}\label{MP_L}
Let $f:M\to N$ be  a   quasiconformal map between equiregular sub-Riemannian manifolds of Hausdorff dimension $Q$.
Then 
 ${\rm Mod}_Q(\Gamma) = {\rm Mod}_Q(f(\Gamma))$
  for every family $\Gamma$ of curves in $M$ if and only if $\Lip_f^Q(p)=\Jac_f(p)$ for a.e. $p$.
\end{proposition}
\begin{proof}
This equivalence is actually a very general fact after the work of Cheeger \cite{Cheeger} and Williams \cite{Williams}.
Since locally  sub-Riemannian manifolds are doubling metric spaces that satisfy a Poincar\'e inequality, we have that 
 the pointwise Lipschitz constant
 $\Lip_{f}(\cdot)$ is the minimal upper gradient of the map $f$, see
 	Proposition~\ref{mug} and Lemma~\ref{uguaglianze}. 
 We also remark that 
  any quasiconformal map is in $W_{\rm loc}^{1,Q}$ and hence in the Newtonian space $N_{\rm loc}^{1,Q}$, see \cite{Balogh-Koskela-Rogovin}.
 By a result of Williams \cite[Theorem 1.1]{Williams},  
$\Lip_{f}(p)^Q \leq   J_{f} (p)  $, for almost every $p$,
if and only if
${\rm Mod}_Q(\Gamma) \leq {\rm Mod}_Q(f(\Gamma)),$
 for every family $\Gamma$ of curves in $M$.
 Hence, we get the inequality 
 $\Lip_{f}(p)^Q \leq   J_{f} (p)  $.
 
 Now consider the inverse map $f^{-1}$. Such a map satisfies the same assumptions of $f$.
 In particular, applying  to $f^{-1}$ William's result, we have that ${\rm Mod}_Q(\Gamma) \leq {\rm Mod}_Q(f^{-1}(\Gamma))$
 for every family $\Gamma$ of curves in $N$ if and only if $\Lip_{f^{-1}}(q)^Q \leq   J_{f^{-1}} (q)  $, for almost every $q\in N$.
 Writing $f^{-1}(\Gamma)=\Gamma'$ and $q=f(p)$ and using \eqref{Lvsl} and \eqref{Jac:inv}, 
 we conclude that ${\rm Mod}_Q(f(\Gamma')) \leq {\rm Mod}_Q(\Gamma')$
 for every family $\Gamma'$ of curves in $M$ if and only if $J_{f} (p)\leq \ell_f^Q (p) \leq \Lip_f^Q (p)  $, for almost every $p\in M$

%
%
%
%
\end{proof}

Let $M$ be an equiregular sub-Riemannian manifold of Hausdorff dimension $Q$. Let $\vol_M$ be the Popp measure of $M$.
For all $u\in W^{1,Q}_{\rm H}(M, \vol_M) $,   the $Q$-energy of $u$ is
$${\rm E}_Q(u) := \int_{M} |\nabla_{\rm H}u|^Q \d\vol_M.$$
\begin{remark}\label{rmk:E_I}
Since 
${\rm E}_Q(u) =\Lap(u ,u)$, 
if
the operator $\Lap$    is preserved,
then
 the $Q$-energy  is preserved. Namely, Condition  \eqref{LP} implies Condition~\eqref{EP}.
\end{remark}

\begin{proposition}\label{EP_CP}
For  a   quasiconformal map 
 $f:M\to N$   between equiregular sub-Riemannian manifolds 
 of Hausdorff dimension $Q$, Condition  \eqref{EP} implies Condition~\eqref{CP}.
\end{proposition}

\begin{proof}
Let $E,F\subset M$ compact sets in $M$.
We set $\mathcal S(E,F)$ to denote the family of all $u\in W^{1,Q}_{\rm H}(M)$ such that $u|_{E}=1$, $u|_{F}=0$ and $0\leq u\leq 1$. Recall that  the $Q$-capacity ${\rm Cap}_Q(E,F)$ is then defined as the infimum of the $Q$-energy $
{\rm E}_Q(u)
$
among all competitors  $u\in\mathcal S(E,F)$:
$$
{\rm Cap}_Q(E,F)= \inf \int_M  |\nabla_{\rm H}u|^Q \d\vol.
$$
Since $f$ satisfies \eqref{EP}, the map $v\ \mapsto v\circ f $  is  a bijection between $\mathcal S(f(E),f(F)) $ and $ \mathcal S(E,F)$
that preserves the $Q$-energy. 
Correspondingly, one has that
\begin{eqnarray*}
{\rm Cap}_Q(f(E),f(F))&=& \inf\{    {\rm E}_Q(v)  \;:\;  {v\in \mathcal S(f(E),f(F))}\}
\\
&=& \inf \{ {\rm E}_Q( v \circ f) \;:\; {v\in \mathcal S(f(E),f(F))} \}
\\
&=& \inf \{ {\rm E}_Q(u) \;:\; {u\in \mathcal S(E,F)} \}
\\
&=& {\rm Cap}_Q(E,F),
\end{eqnarray*}
completing the proof. 
\end{proof}

\begin{proposition}\label{JP_LP}
Let $f:M\to N$ be a 
quasiconformal map between equiregular sub-Riemannian manifolds of Hausdorff dimension $Q$. 
Either of Condition~\eqref{HS} and Condition~\eqref{JP} implies Condition~\eqref{LP}.
 \end{proposition}


\proof 
Let $p$ be a point of differentiability of $f$. Given an orthonormal basis $\{X_j\}$ of $H_pM$, 
from \eqref{HS} we have that
vectors
$$
Y_j:= \Lip_f(p)^{-1}(\d_{\rm H}f)_p X_j
$$
form an orthonormal basis of $H_qN$, with $q=f(p)$.
Then, for every open subset $V\subset N$ and for every $v\in W^{1,Q}_{\rm H}(N)$,
\begin{align*}
X_j (v\circ f)_p &= \d_{\rm H}(v\circ f)_p(X_j)\\
&= (\d_{\rm H} v)_q(\d_{\rm H}f)_p (X_j)\\
&= (\d_{\rm H}v)_q (\Lip_f(p) Y_j) = \Lip_f(p)(Y_j u)_q.
\end{align*}
Therefore,  for any $v, \phi \in W_{\rm H}^{1,Q}(V)$,
\begin{align*}
\langle \nabla_{\rm H} (v\circ f), \nabla_{\rm H} (\phi\circ f)\rangle_p &= \sum_j X_j(v\circ f)_p X_j(\phi\circ f)_p\\
&= L^2_f(p) \sum_j Y_j(v)_q Y_j(\phi)_q\\
&= L^2_f(p) \langle \nabla_{\rm H}v,\nabla_{\rm H}\phi\rangle_q.
\end{align*}
In particular
$$
|\nabla_{\rm H} (v\circ f)| = L_f(p) |(\nabla_{\rm H}v)_{f(\cdot)}|.
$$
So, using Condition~\eqref{JP} and writing $U=f^{-1}(V)$,
 \begin{eqnarray*}
I_Q(v\circ f,\phi\circ f; U)&=& \int_{U} |\nabla_{\rm H} (v\circ f)|^{Q-2} \langle \nabla_{\rm H} (v\circ f), \nabla_{\rm H} (\phi\circ f)\rangle \d\vol_M\\
&=&  \int_{U} L_f^{Q-2} | (\nabla_{\rm H} v)_{f(\cdot)}|^{Q-2} L_f^2 \langle \nabla_{\rm H} v, \nabla_{\rm H} \phi\rangle_{f(\cdot)} \d\vol_M  \\
&=& 
 \int_{U} \Jac_f | (\nabla_{\rm H} v)_{f(\cdot)}|^{Q-2}  \langle \nabla_{\rm H} v, \nabla_{\rm H} \phi\rangle_{f(\cdot)} \d\vol_M  \\
&=& \int_{V} |\nabla_{\rm H} v|^{Q-2} \langle \nabla_{\rm H} v,  \nabla_{\rm H} \phi\rangle \d\vol_N \\
&=& I_Q(v,\phi; V),
 \end{eqnarray*}
 where we used \eqref{change of variable}.
%
%
\qed

\begin{proposition}\label{LP_JP}
Let $f:M\to N$ be a 
quasiconformal map between equiregular sub-Riemannian manifolds of Hausdorff dimension $Q$. 
Then Condition~\eqref{LP} implies Condition~\eqref{JP}.
 \end{proposition}
 
 \begin{proof}
 We start with the following chain of equalities, where we use \eqref{LP}, the chain rule and the change of variable formula \eqref{change of variable}. 
 For every open subset $U\subset M$, denote $V=f(U)\subset N$. 
 For every $v,\phi\in W^{1,Q}_{\rm H}(V)$, 
 \begin{align*}
 \int_{V} |\nabla_{\rm H} v|^{Q-2} \langle \nabla_{\rm H} v,  \nabla_{\rm H} \phi\rangle \d\vol_N =   \int_{U} |\nabla_{\rm H} (v\circ f)|^{Q-2} \langle \nabla_{\rm H} (v\circ f), \nabla_{\rm H} (\phi\circ f)\rangle \d\vol_M  & \\
 =  \int_{U} |(\d_{\rm H}f)^{\rm T}_{f(\cdot)}(\nabla_{\rm H} v)_{f(\cdot)} |^{Q-2} \langle (\d_{\rm H}f)^{\rm T}_{f(\cdot)}(\nabla_{\rm H} v)_{f(\cdot)}, (\d_{\rm H}f)^{\rm T}_{f(\cdot)}(\nabla_{\rm H} \phi)_{f(\cdot)}\rangle \d\vol_M & \\
  = \int_{V}\Jac_{f^{-1}}(\cdot) |(\d_{\rm H}f)_{\cdot}^{\rm T} (\nabla_{\rm H} v)_{\cdot} |^{Q-2}  \langle  (\d_{\rm H}f)_{\cdot}^{\rm T} (\nabla_{\rm H} v)_{\cdot}, (\d_{\rm H}f)_{\cdot}^{\rm T} (\nabla_{\rm H} \phi)_{\cdot}\rangle \d\vol_N & \\
   = \int_{V}\Jac_{f^{-1}}(\cdot) |(\d_{\rm H}f)_{\cdot}^{\rm T} (\nabla_{\rm H} v)_{\cdot} |^{Q-2}  \langle (\d_{\rm H}f)_{f^{-1}(\cdot)} (\d_{\rm H}f)_{\cdot}^{\rm T} (\nabla_{\rm H} v)_{\cdot},  (\nabla_{\rm H} \phi)_{\cdot}\rangle \d\vol_N, &
 \end{align*}
 where $(\d_{\rm H}f)^{\rm T}_{q}$ denotes the adjoint of  $(\d_{\rm H}f)_{f^{-1}(q)}$ with respect to the metrics on $N$ and $M$ at $q$ and $f^{-1}(q)$ respectively.
We then proved that 
\begin{equation}\label{lunga}
 \int_{V} \langle |(\nabla_{\rm H} v)_\cdot|^{Q-2}  (\nabla_{\rm H} v)_\cdot - \Jac_{f^{-1}}(\cdot) |(\d_{\rm H}f)_{\cdot}^{\rm T} (\nabla_{\rm H} v)_{\cdot} |^{Q-2}
 (\d_{\rm H}f)_{f^{-1}(\cdot)} (\d_{\rm H}f)_{\cdot}^{\rm T} (\nabla_{\rm H} v)_{\cdot} ,  (\nabla_{\rm H} \phi)_\cdot\rangle \d\vol_N = 0 
 \end{equation}
 for every $v,\phi\in W^{1,Q}_{\rm H}(V)$ and for every open subset $V\subset N$. Note that \eqref{lunga} holds true for every measurable subset $V\subset N$.
 We claim that, for almost every $q\in N$,
  \begin{equation}\label{lunga-seconda}
 |(\nabla_{\rm H} v)_q|^{Q-2}  (\nabla_{\rm H} v)_q - \Jac_{f^{-1}}(q) |(\d_{\rm H}f)_{q}^{\rm T} (\nabla_{\rm H} v)_{q}|^{Q-2}
 (\d_{\rm H}f)_{f^{-1}(q)} (\d_{\rm H}f)_{q}^{\rm T} (\nabla_{\rm H} v)_{q}=0
   \end{equation}
 for every $v\in W^{1,Q}_{\rm H}(N)$.
Arguing by contradiction,  assume that there is a set $V\subset  N$ of positive measure where \eqref{lunga-seconda} fails for some $v\in W^{1,Q}_{\rm H}(N)$.
Choose any smooth frame $X_1,\dots,X_r$ of $HN$, and write the left hand side of \eqref{lunga-seconda} as  $\sum_{i=1}^r\psi_i X_i$, with
$\psi_i\in L^{Q}(N)$ for every $i=1,\dots,r$.
Then at least one of the $\psi_i$ must be different from zero in $V$.  Without loosing generality, say $\psi_1 \neq 0$ on $V$.
By possibly taking $V$ smaller, we may 
 assume 
that $\int_{V} \psi_1  \d\vol_N \neq 0$. Let  $\phi$ be the coordinate function $x_1$, that is $X_j \phi = \delta_1^j$. Substituting in the left hand side of \eqref{lunga}, we conclude
$$
\int_{V} \langle  (\psi_1,\dots,\psi_r), \nabla_{\rm H} \phi \rangle  \d\vol_N = \int_{V} \psi_1  \d\vol_N \neq 0,
$$
 which contradicts \eqref{lunga}. This completes the proof of  \eqref{lunga-seconda}.
 %
%
%

 Next, fix $q\in N$ a point of differentiability where  \eqref{lunga-seconda} holds. For every vector $\xi\in H_qN$, consider $v_\xi$ such that  $(\nabla_{\rm H} v_\xi)_{q}=\xi$.
 For every $\xi\in H_qN$ such that $|\xi|=1$, the following holds 
 $$
  \Jac_{f^{-1}}(q) |(\d_{\rm H}f)_{q}^{\rm T}\xi |^{Q-2}
 \langle(\d_{\rm H}f)_{f^{-1}(q)} (\d_{\rm H}f)_{q}^{\rm T}\xi,\xi\rangle = 1.
 $$
 Using \eqref{Jac:inv}, the equality above becomes
 $$
   |(\d_{\rm H}f)_{q}^{\rm T}\xi |^{Q-2}\langle (\d_{\rm H}f)_{q}^{\rm T}\xi, (\d_{\rm H}f)_{q}^{\rm T} \xi\rangle = \Jac_f (f^{-1}(q))
 $$
 which is equivalent to
 $$
  |(\d_{\rm H}f)_{q}^{\rm T}\xi|^{Q} = \Jac_f (f^{-1}(q))
 $$
 for every  $\xi$ on $H_qN$ of norm equal to one. From \eqref{horizontal-push-forward} we have
$
|(\d_{\rm H}f)_{q}^{\rm T}\xi(q) |^{Q}=| \mathcal N_q(f)_*^{\rm T}\xi(q) |^{Q} .
 $
 Therefore, at every point $q\in N$ of differentiability,
 $$
 \| \mathcal N_{f^{-1}(q)}(f)_*\|^Q= {\rm max} \{| \mathcal N_q(f)_*^{\rm T}\xi |^{Q}\,:\, \xi \in H_qN, |\xi|=1\}=  \Jac_f (f^{-1}(q)).
 $$
 By Lemma~\ref{uguaglianze} and writing $p=f^{-1}(q)$, we conclude $ \Lip_f(p)^Q= \Jac_f(p)$ for almost every $p\in M$, establishing  \eqref{JP}.
  \end{proof}


\subsection{The morphism property}\label{morphism-prop} 
\begin{proof}[Proof of Corollary~\ref{morphism-property}]
Let $v\in    W_{\rm H}^{1,Q}(N)$ and  $\phi \in    W_{{\rm H},0}^{1,Q}(N)\subset   W_{\rm H}^{1,Q}(N)$,
then from \eqref{LP} it follows
\[
 L_Q (v)(\phi)= 
 \Lap(v, \phi )=\Lap(v\circ f , \phi \circ f )
 =
L_Q ( v\circ f) \circ f^*  (\phi) . \qedhere\]
\end{proof}


\subsection{Equivalence of the two Jacobians} 

Given $M$ an equiregular sub-Riemannian manifold of Hausdorff dimension $Q$,
we prefer to work with the Popp measure $\vol_M$
rather than the spherical Hausdorff measure $\mathcal S_M^Q$ 
since
$\vol_M$ is always smooth whereas there are cases in which $\mathcal S_M^Q$ is not (see \cite{Agrachev_Barilari_Boscain:Hausdorff}).
However, one has the following formula (see
\cite[pages 358-359]{Agrachev_Barilari_Boscain:Hausdorff},  \cite[Section 3.2]{ghezzi-jean}).
\begin{equation}\label{F1}
\mathrm d \vol_M = 2^{-Q} \vol_{{\mathcal N}_p(M)} (B_{{\mathcal N}_p(M)} (e,1)) \mathrm d {\mathcal S}^Q_M,
\end{equation}
where we used the fact that the measure induced on 
${\mathcal N}_p(M)$ by $\vol_M$ is $
\vol_{{\mathcal N}_p(M)}$.

\begin{proposition}\label{equivalence_of_jacobians}
If $f:M\to N$ is  a   $1$-quasiconformal map between equiregular sub-Riemannian manifolds, then
for almost every $p\in M$,
$$\Jac_f^{\rm Popp}(p)=\Jac_f^{\rm Haus}(p).$$
\end{proposition}

\begin{proof}

Let $A\subseteq M$
be a measurable set.
Since $f^{-1}$ is 
\footnote{Here we need to invoke \cite[Corollary 6.5]{Margulis-Mostow} or \cite{Heinonen-Koskela}}
also $1$-quasiconformal, 
then we have
\eqref{54} with $p=f^{-1}(q)$ for almost all $q\in M$,
Then, using   twice \eqref{F1}, we have
\begin{eqnarray*}
2^{Q} (f^* \vol_N)(A) & = & 2^{Q} \vol_N (f(A))\\
  &=&   \int_{f(A)}   \vol_{ {\mathcal N}_q(N) }(B_{{\mathcal N}_q(N)} (e,1)) \,\d {\mathcal S}^Q_N(q)\\ 
&=&
\int_{f(A)}  \vol_{ {\mathcal N}_{f^{-1}(q)}(M) }(B_{{\mathcal N}_{f^{-1}(q)}(M)} (e,1))\, \d {\mathcal S}^Q_N(q)\\
 &=&
 \int_{A}  \vol_{ {\mathcal N}_p(M) }(B_{{\mathcal N}_p(M)} (e,1)) \Jac_f^{\rm Haus}(p) \, \d {\mathcal S}^Q_M(p)\\
  &=& 2^Q ( \Jac_f^{\rm Haus}  \vol_M)(A).
\end{eqnarray*}
Thus, we conclude that
$
\Jac_f^{\rm Popp} \vol_M
= f^* \vol_N 
=\Jac_f^{\rm Haus}\vol_M$.
\end{proof}



\section{Coordinates in sub-Riemannian manifolds}\label{sec:coord}
Given any system of coordinates near a point of a sub-Riemannian manifolds, we will identify special subsets of these coordinates,  that we call {\it horizontal}. By adapting a method of Liimatainen and Salo \cite{Salo}, we show that they can be constructed so that in addition they are also either harmonic or $Q$-harmonic (the more general construction of $p$-harmonic coordinates follows along the same lines, modifying appropriately the hypothesis). The construction of $Q$-harmonic coordinates is based upon a very strong hypothesis, namely that the sub-Riemannian structure supports regularity for  $Q$-harmonic functions. In contrast, the construction of horizontal harmonic coordinates rests on well known Schauder estimates. The key point of this section, and one of the main contributions of this paper, is that we can prove that  the smoothness of maps that preserve in a weak sense the horizontal bundles can be derived by the smoothness of the horizontal components alone.

\subsection{Horizontal coordinates}

\definition
Let $M$ be a sub-Riemannian manifold.
Let $x^1,\dots,x^n$ be a system of coordinates on an open set $U$ of $M$ and let $X_1,\dots,X_r$ be a frame of the horizontal distribution on $U$. We say that $x^1,\dots,x^r$ are {\em horizontal coordinates} with respect to  $X_1,\dots,X_r$ if  the matrix $(X_ix^j)(p)$, with $i,j=1,\dots,r$,   is invertible, for every $p\in U$.

\remark
It is clear that {\it any} system of coordinates $x^1,\ldots .,x^n$ around a point $p\in M$ can be reordered so that the first $r$ components become a system of horizontal coordinates.
\endremark

The next result states that the notion of horizontal coordinate does not depend on the choice of  frame.

\begin{proposition}
Assume that $x^1,\dots,x^n$ are coordinates such that $x^1,\dots,x^r$ are horizontal with respect to  the frame $X_1,\dots,X_r$. Then
\begin{itemize}
\item[(i)] $\nabla_{\rm H} x^1,\dots,\nabla_{\rm H} x^r$ are linearly independent and form a frame of $\Delta$.
\item[(ii)] If $X_1^\prime,\dots,X_r^\prime$ is another frame of $\Delta$, then $x^1,\dots,x^r$ are horizontal coordinates with respect to  $X_1^\prime,\dots,X_r^\prime$.
\end{itemize}
\end{proposition}
\begin{proof}
Since $O:=(X_ix^j)_{ij}$ is invertible and $X_1,\dots,X_r$ is a frame, then 
$$\nabla_{\rm H}x^i=\sum_{i=1}^r(X_kx^i)X_k=\sum_k O_i^kX_k$$ and $(i)$ follows.
Regarding $(ii)$, let $B$ be the matrix such that $X_i^\prime x^j=\sum_{k=1}^r B_i^kX_kx^j= (BO)_{ij}$. The conclusion follows from the invertibility of $BO$.
\end{proof}

\subsection{Horizontal harmonic coordinates}

Let $M$ be a sub-Riemannian manifold endowed with a   volume form $\vol$.  Our goal is to construct  horizontal coordinates in the neighborhood of any point $p\in M$, that are also in the kernel of the subLaplacian $L_2$,  defined in \eqref{sublaplacian},  associated to the sub-Riemannian structure and a volume form.


\begin{theorem}\label{hhc}
Let $M$ be an equiregular sub-Riemannian structure endowed with a smooth volume form $\vol$.
For any point $p\in M$ there exists a set of horizontal harmonic coordinates  defined in a neighborhood of $p$.
\end{theorem}

To prove this result we start by considering any system of  coordinates $x^1,\dots,x^n$ in a neighborhood of  $p\in M$. Without loss of generality we can assume that the vectors 
$\nabla_{\rm H} x^1,\ldots ,\nabla_{\rm H} x^r$ are linearly independent in a neighborhood of $p$, i.e., $x^1,..,x^r$ are horizontal coordinates.  Set $B_\e:=B_\epsilon(p)= \{q\in M\mid d(p,q)<\epsilon\}$. For $\epsilon >0$, let $u_\epsilon^1,\dots,u_\epsilon^n$ be the unique weak solution of the Dirichlet
problem
$$
\begin{cases}
L_2 u_\epsilon^i = 0\,\, \text{in } B_\epsilon,\,\, i=1,\dots,n\\
u_\epsilon^i=x^i \,\, \text{in } \partial B_\epsilon, \,\, i=1,\dots,n.
\end{cases}
$$
We will show that for $\e>0$ sufficiently small, the $n$-tuple $u^1_\e,\ldots ,u^r_\e,x^{r+1},\ldots ,x^n$ is a system of coordinates. Note that $u^1_\e,\ldots ,u^n_\e$ may fail to be a system of coordinates.

H\"ormander's hypoellipticity result \cite{Hormander} yields $u_\epsilon^i\in C^\infty(B_\epsilon)\cap W^{1,2}_{\rm H}(B_\epsilon)$. Consider now
$$w_\epsilon^i:=u_\epsilon^i-x^i\in C^\infty(B_\epsilon)\cap W_{{\rm H},0}^{1,2}(B_\epsilon).$$ 

\begin{lemma}\label{caccioppoli}
For $p\in K\subset\subset M$, the following  estimate holds
\begin{equation}\label{average-estimate}
\dashint_{B_\epsilon} |\nabla_{\rm H} w_\epsilon^i|^2  \,\d \vol \leq C^\prime \epsilon^2.
\end{equation}
for a constant $C'>0$ depending only on $K$, on the coordinates $x^1,..,x^n$, the Riemannian structure of $M$ and the volume form.
\end{lemma}
\begin{proof}  For every $i=1,\dots,n$, the function $w_\epsilon^i$ solves 
$$
\begin{cases}
L_2 w_\epsilon^i=-L_2 x^i=:g_i\\
w_\epsilon^i=0 \,\, {\text{in }} \partial B_\epsilon
\end{cases}
$$
The equation can be interpreted in a weak sense as
$$
\int_{B_\epsilon}\nabla_{\rm H}w_\epsilon^i\nabla_{\rm H} \phi\, {\rm d}\,{\rm vol}=\int_{B_\epsilon}g_i\phi \, {\rm d}\,{\rm vol}
$$
for every $\phi\in W_{{\rm H},0}^{1,2}(B_\epsilon)$. Choosing $\phi=w_\epsilon^i$ gives
$$
\int_{B_\epsilon}|\nabla_{\rm H} w_\epsilon^i|^2\, {\rm d}\,{\rm vol}= \int_{B_\epsilon}g_iw_\epsilon^i \,{\rm d}\,{\rm vol} \leq \left( \int_{B_\epsilon} g_i^2 \,{\rm d}\,{\rm vol}\right)^{1/2}\left(  \int_{B_\epsilon} (w_\epsilon^i)^2\,{\rm d}\,{\rm vol}\right)^{1/2}.
$$
Poincar\'e inequality for functions with compact support gives
$$
\int_{B_\epsilon} (w_\epsilon^i)^2 \,{\rm d}\,{\rm vol}\leq C \epsilon^2 \int_{B_\epsilon} |\nabla_{\rm H} w_\epsilon^i|^2 \,{\rm d}\,{\rm vol},
$$
whence
$$
\int_{B_\epsilon} |\nabla_{\rm H} w_\epsilon^i|^2\,{\rm d}\,{\rm vol} \leq  \left( \int_{B_\epsilon} g_i^2 \,{\rm d}\,{\rm vol}\right)^{1/2} \left( C \epsilon^2 \int_{B_\epsilon} |\nabla_{\rm H} w_\epsilon^i|^2 \,{\rm d}\,{\rm vol}\right)^{1/2}.
$$
We have
$$
\left( \int_{B_\epsilon} |\nabla_{\rm H} w_\epsilon^i|^2 \,{\rm d}\,{\rm vol}\right)^{1/2}\leq \epsilon C^{1/2}   \left( \int_{B_\epsilon} g_i^2 \,{\rm d}\,{\rm vol}\right)^{1/2}\leq \epsilon C^{1/2 }{\rm vol}(B_\epsilon)^{1/2}\left(\sup_{B_\epsilon}g_i^2\right)^{1/2}.
$$
This completes the proof of \eqref{average-estimate}.
\end{proof}
%
%
Next we need an interpolation inequality that allows us to bridge the $L^2$ estimates \eqref{average-estimate} and the $C^{1,\alpha}_{\rm H}$ estimates from \eqref{schauder-consequence} to produce $L^\infty$ bounds. The following is very similar to the analogue interpolation lemma in \cite{Salo}.
\begin{lemma}\label{interp} Let $p\in K\subset \subset M$ and let
 $h$ be a function defined on $B_\epsilon$. If there are constants $\mathcal A,\mathcal B>0$ such that for $\e>0$ sufficiently small one has 
\begin{itemize}
\item[(i)] $\|h\|^2_{L^2(B_\epsilon)}\leq \mathcal A\epsilon^2|B_\epsilon|^2$,\\
\item[(ii)] $\| h\|_{C^\alpha_{\rm H}(B_{ {\epsilon}/{2}})}\leq \mathcal B$,
\end{itemize}
then $\|h\|_{L^\infty(B_{\epsilon/4})}\leq o(1)$ as $\epsilon \to 0$, uniformly in $p\in K$.
\end{lemma}
\begin{proof}
Set $q\in B_{\frac{\epsilon}{4}}(p)$ so that $B_{\frac{\epsilon}{4}}(q)\subset B_{\frac{3\epsilon}{4}}(p)$. 
One has
\begin{align*}
\|h\|_{L^2(B_{\frac{\epsilon}{4}}(q))}&\geq \|h(q)\|_{L^2(B_{\frac{\epsilon}{4}}(q))}-\|h-h(q)\|_{L^2(B_{\frac{\epsilon}{4}}(q))}\\
&=|h(q)|\cdot|B_{\frac{\epsilon}{4}}(q))|^{\frac{1}{2}} - \left(\int_{B_{\frac{\epsilon}{4}}(q)} |h(\cdot)-h(q)|^2 \,{\rm d}\,{\rm vol}\right)^{\frac{1}{2}}\\
&\geq |h(q)|\cdot |B_{\frac{\epsilon}{4}}(q))|^{\frac{1}{2}} - \sup_{B_{\frac{\epsilon}{4}}(q)}\frac{|h(\cdot)-h(q)|}{d(\cdot,q)^\alpha}\left( \int_{B_{\frac{\epsilon}{4}}(q)} d(\cdot,q)^{2\alpha}  \,{\rm d}\,{\rm vol} \right)^{\frac{1}{2}}.
\end{align*}
We then obtain that there exists  constants 
$C_1$ and $C_2$,
depending only on the sub-Riemannian structure, the exponent $\alpha$, and the compact set $K$, 
such that
$$
\|h\|_{L^2(B_{\frac{\epsilon}{4}}(q))}\geq C_1\epsilon^{\frac{Q}{2}}|h(q)|-C_2\epsilon^{\alpha+\frac{Q}{2}}\|h\|_{C^\alpha_{\rm H}(B_{\frac{\epsilon}{4}}(q))}.
$$
Using the hypotheses $(i)$ and $(ii)$, we conclude for all $q\in B_{\frac{\epsilon}{2}(p)}$
\begin{align*}
|h(q)|& \leq C_1^{-1}\epsilon^{-\frac{Q}{2}}\left( \|h\|_{L^2(B_{\frac{\epsilon}{4}(q)})}+C_2\epsilon^{\alpha+\frac{Q}{2}}\|h\|_{C^\alpha_{\rm H}}(B_{\frac{\epsilon}{4}(q)})\right)\\
&\leq  C_1^{-1} \left\{ \mathcal A^{1/2} \e + \mathcal B  C_2\epsilon^\alpha\right\}= o(1)
\end{align*}
as $\epsilon\to 0$.
\end{proof}

In view of \eqref{average-estimate} and  \eqref{schauder-consequence} we can apply the previous lemma to $h=\nabla_{\rm H} w_\epsilon^i$ and infer
$$\sup_{B_{\frac{\epsilon}{4}}} | \nabla_{\rm H} u_\epsilon^i-\nabla_{\rm H} x^i|\leq o(1)$$ as $\epsilon \to 0.$
Since the matrix $(X_i x^j)_{ij}$ for $i,j=1,\ldots ,r$ is invertible in a neighborhood of $p$, then for $\e>0$ sufficiently small the same holds for the matrix $(X_i u^j_\e)_{ij}$. Consequently, the 
$n$-tuple $(u^1_\e,\ldots ,u^r_\e, x^{r+1},\ldots ,x^n)$ yields a system of coordinates in a neighborhood of $p$ and its  first  $r$ components are both horizontal and harmonic.
This concludes the proof of Theorem \ref{hhc}.
\qed

\subsection{Horizontal $Q$-harmonic coordinates} 

Throughout this section we will assume that $M$ is an equivariant sub-Riemannian structure, endowed with a smooth volume form $\vol$, that supports regularity for  $Q$-harmonic functions, in the sense of Definition~\ref{support}.

We will need an interpolation lemma analogue to Lemma \ref{interp}.

\begin{lemma}\label{Qinterp} Let $p\in K\subset \subset M$ and let
 $f$ be a function defined on $B_\epsilon$. If there are constants $\beta, \mathcal A,\mathcal B>0$ and $\alpha\in (0,1)$ such that for $\e>0$ sufficiently small one has 
\begin{itemize}
\item[(i)] $\|h\|_{L^Q(B_\epsilon)}\leq \mathcal A\epsilon^{1+\beta} $\\
\item[(ii)] $\| h\|_{C^\alpha_{\rm H}(B_{\frac{\epsilon}{2}})}\leq \mathcal B$,
\end{itemize}
then $\|h\|_{L^\infty(B_{\epsilon/4})}\leq o(1)$ as $\epsilon \to 0$, uniformly in $p\in K$.
\end{lemma}
\begin{proof}
Using the notation and the argument in the proof of Lemma \ref{interp}, one concludes  that for any $q\in B_{\frac{\e}{4}}(p)$ one has 
$$\|h\|_{L^Q(B_{\frac{\e}{4}}(q))} \ge |h(q)|\cdot|B_{\frac{\e}{4}}(q)|^{\frac{1}{Q}} - \|h\|_{C^\alpha_{\rm H}(B_{\frac{\e}{4}}(q)) } \e^{\al+\frac{Q}{p}}.$$
The proof follows immediately from the latter and from the hypothesis.
\end{proof}

\begin{theorem}\label{hQhc}
Let $M$ be an equiregular sub-Riemannian structure endowed with a smooth volume form $\vol$ that supports regularity for  $Q$-harmonic functions.
For any point $p\in M$ there exists a set of horizontal  coordinates  defined in a neighborhood of $p$ that are $Q$-harmonic.
\end{theorem}

\begin{proof} We follow the argument outlined in the special case of Theorem \ref{hhc}. For $p\in K\subset \subset M$ and $\e>0$ to be determined later, we consider  weak  solutions $u_\e^i \in W^{1,Q}_{\rm H}(B_\epsilon)$ to the Dirichlet problems
$$
\begin{cases}
L_Q u_\epsilon^i = 0\,\, \text{in } B_\epsilon,\,\, i=1,\dots,n\\
u_\epsilon^i=x^i \,\, \text{in } \partial B_\epsilon, \,\, i=1,\dots,n,
\end{cases}
$$
where $x^1,\ldots,x^n$ is an arbitrary set of coordinates near $p$. These solutions exist and are unique in view of the convexity of the $Q$-energy. The $C^{1,\alpha}_{\rm H}$ estimates assumptions guarantee that $u_\e^i\in C^{1,\alpha}_{\rm H,loc}(B_\e)\cap W^{1,Q}_{\rm H,loc}(B_\e)$.
Arguing as in Lemma~\ref{caccioppoli}, we set
$$w_\epsilon^i:=u_\epsilon^i-x^i\in C^{1,\alpha}_{\rm H,loc}(B_\e)\cap W_{{\rm H},0}^{1,Q}(B_\epsilon)$$ and 
observe that
$$\int_{B_\e} |\nabla_{\rm H} u^i_\e|^{Q-2} X_k u^i_\e X_k w^i_\e \,{\rm d}\,{\rm vol}=0.$$
As a consequence one has
\begin{eqnarray}
\int_{B_\e} |\nabla_{\rm H} w_\e^i|^Q  \,{\rm d}\,{\rm vol} \le &&
\int_{B_\e} (|\nabla_{\rm H} u^i_\e|+ |\nabla_\e x^i|)^{Q-2} | \nabla_{\rm H} w^i_\e|^2 \,{\rm d}\,{\rm vol}\notag\\
\le && \int_{B_\e} (|\nabla_{\rm H} u^i_\e|^{Q-2} X_k u^i_\e- |\nabla_{\rm H} x^i|^{Q-2} X_k x^i) X_k w_\e^i 
 \,{\rm d}\,{\rm vol}\notag \\
 = &&  \int_{B_\e}- X_k^*(  |\nabla_{\rm H} x^i|^{Q-2} X_k x^i)  w_\e^i 
 \,{\rm d}\,{\rm vol}\notag \\
 \le && \bigg( \int_{B_\e}  |L_Q  x^i|^{Q} \,{\rm d}\,{\rm vol}\bigg)^{\frac{Q-1}{Q}}
 \bigg( \int_{B_\e}  | w^i_\e|^{Q} \,{\rm d}\,{\rm vol}\bigg)^{\frac{1}{Q}}\notag \\
\text{ (applying Poincar\'e inequality)} \ \ \   \le &&  C \e \bigg( \int_{B_\e}  |L_Q  x^i|^{Q} \,{\rm d}\,{\rm vol}\bigg)^{\frac{Q-1}{Q}}
 \bigg( \int_{B_\e}  |\nabla_{\rm H} w^i_\e|^{Q} \,{\rm d}\,{\rm vol}\bigg)^{\frac{1}{Q}}\notag \\
  \le && C' \e^{Q}||\nabla_{\rm H} w^i||_{L^Q(B_\e)},
 \end{eqnarray}
  for  constants $C,C'>0$ depending only on $Q,K$, on the coordinates $x^1,\ldots,x^n$, the sub-Riemannian structure, and the volume form.
From the latter  it immediately follows that 
\begin{equation}\label{Qaverage-estimate}
 \|\nabla_{\rm H} w_\epsilon^i\|_{L^Q(B_\e)} \leq C^{''} \epsilon^{1+\frac{1}{Q-1}}
\end{equation}
Arguing as in Theorem \ref{hhc}, and  applying the $C^{1,\alpha}_{\rm H}$ estimates, \eqref{Qaverage-estimate} and the interpolation Lemma \ref{Qinterp},
 one has that for $\e>0$ sufficiently small the matrix $(X_i u^j_\e)_{ij}$, for $i,j=1,\ldots ,r$ is invertible in a neighborhood of $q$. On the other hand, this implies that for each $i=1,\ldots,r$ one has  that $|\nabla_{\rm H} u^i_\e|$ is a $C^\alpha_{\rm H}$ function bounded away from zero in a neighborhood of $p$, and hence by part (2) of
 Definition~\ref{support}  and by Proposition~\ref{smoothness} one has that $u^1_\e,\ldots,u^r_\e, x^{r+1},\ldots,x^n$ is a smooth system of coordinates in a neighborhood of $p$, with $u^1_\e,\ldots,u^r_\e$ both horizontal and $Q$-harmonic.
\end{proof}

\subsection{Regularity from horizontal regularity}

Let $\gamma$ be an horizontal curve in $M$. Let $x^1,\dots,x^n$ be coordinates on $M$ such that $x^1,\dots,x^r$ are horizontal coordinates with respect to an horizontal frame $X_1,\dots,X_r$. 
We write 
$$\gamma_{\rm H}=(x^1\circ \gamma,\dots,x^r\circ\gamma)\quad \text{ and }\gamma_V=(x^{r+1}\circ \gamma,\dots,x^n\circ \gamma).$$ Hence  $\gamma=(\gamma_{\rm H},\gamma_V)$ and $\dot{\gamma}=(\dot{\gamma}_{\rm H},\dot{\gamma}_V)$. There are functions $\beta_1,\dots,\beta_r$ so that 
$$
\dot{\gamma}=\sum_{j=1}^r \beta_j(X_j\circ \gamma).
$$
In coordinates we write $X_j=\sum_{k=1}^n X_j^k \frac{\partial}{\partial x^k}$. 
So
\begin{align*}
(\dot{\gamma}_{\rm H},\dot{\gamma}_V)&=\sum_{j=1}^r\beta_j(X_j\circ\gamma)\\
&=\sum_{j=1}^r\beta_j\sum_{k=1}^n (X_j^k\circ\gamma) \frac{\partial}{\partial x^k}\\
&=  \sum_{k=1}^r\sum_{j=1}^r \beta_j (X_j^k\circ\gamma) \frac{\partial}{\partial x^k}+ \sum_{k=r+1}^n\sum_{j=1}^r \beta_j (X_j^k\circ\gamma)\frac{\partial}{\partial x^k}.
\end{align*}
Set $O=(X_jx^i)_{ij}=X_j^i$. We have $\dot{\gamma_{\rm H}}=\sum_{k=1}^r\sum_{j=1}^r \beta_j (O_j^k\circ \gamma) \frac{\partial}{\partial x^k}=O\beta$, where we denoted $\beta=(\beta_1,\dots,\beta_r)$.
Since $O$ is invertible, $(\beta_1,\dots,\beta_r)=(O^{-1}\circ \gamma) \dot{\gamma_{\rm H}}$. Thus
\begin{equation}\label{vertical-derivatives}
\dot{\gamma}_V= \sum_{k=r+1}^n\sum_{j=1}^r \bigg[(O^{-1}\circ \gamma) \dot{\gamma_{\rm H}}\bigg]_j (X_j^k\circ\gamma)  \frac{\partial}{\partial x^k}.
\end{equation}
In particular, the following holds.
\begin{proposition}
Let $\gamma$ be an absolute continuous curve. If $\gamma_{\rm H}$ is smooth, then $\gamma$ is smooth.
\end{proposition}
\begin{proof}
By hypothesis $\gamma$ and $\dot{\gamma_{\rm H}}$ are absolute continuous. Then by \eqref{vertical-derivatives} also $\dot{\gamma}_V$ is absolute continuous. Thus $\dot\gamma$ is continuous. A bootstrap argument shows that $\gamma$ is smooth.
\end{proof}

In the following, we will consider maps that are {\em absolutely continuous on curves} (ACC$_Q$). 
 We recall that such maps
send almost every (with respect to the $Q$-modulus measure)
 rectifiable  curve into a rectifiable  curve (see \cite{Shanmugalingam_Newtonian} for more details). 
 In the case of a sub-Riemannian manifold $M$, 
 ACC maps defined on $M$ have the following property.
 Let $X$ be any horizontal vector field in $M$ and denote by $\phi_X^t$ the corresponding flow. 
 Then for  almost every $p\in M$ (with respect to Lebesgue measure), one has that  
 $t\to  f(\phi_X^t(p))$ is a rectifiable curve.

\begin{proposition}\label{smoothfromhorizontal}
Let $M$ and $N$ two sub-Riemannian manifolds. 
Let $f:M\to N$ be an ACC map.  
Let $k\geq 1$, $\alpha\in(0,1)$, and $p\geq 1$. 
If $f^1,\dots,f^r$ are in $C^{k,\alpha}_{\rm H,loc} (M)$  (resp. in $W^{k,p}_{\rm H,loc} (M)$), then $f^1,\dots,f^n$ is $C^{k,\alpha}_{\rm H,loc} (M)$ (resp. in $W^{k,p}_{\rm H,loc} (M)$).
\end{proposition}
\begin{proof}
Let $X$ be any horizontal vector field in $M$.
Notice that if $f^1,\dots,f^r$ are in $C^{k,\alpha}_{\rm H,loc} (M)$  (resp. in $W^{k,p}_{\rm H,loc} (M)$),
then $Xf^1,\dots,Xf^r$ are in $C^{k-1,\alpha}_{\rm H,loc} (M)$  (resp. in $W^{k-1,p}_{\rm H,loc} (M)$).
For  almost every $p\in M$, the curve 
$$
f(\phi_X^t(p))=:\gamma(p,t)=(\gamma_H(p,t),\gamma_V(p,t)),
$$
is an horizontal curve and hence 
  \eqref{vertical-derivatives} holds. Therefore, for  almost every $p$,
  we have
\begin{align*}
(X f^{m+1}(p),\dots,Xf^n(p))
&= 
\frac{d}{dt} \gamma_V(p,t)|_{t=0}\\
&=\sum_{j=1}^m \sum_{k=m+1}^n \bigg[(O^{-1}\circ \gamma(p,0))\dot{\gamma_{\rm H}} (p,0)\bigg]_j\, X_j^k(\gamma(p,0))  \frac{\partial}{\partial x^k}\\
&= \sum_{j=1}^m \sum_{k=m+1}^n \bigg[O^{-1}(f(p))(Xf^1(p),\dots,Xf^r(p))^{T}\bigg]_j (X_j^k\circ f)(p)  \frac{\partial}{\partial x^k}.
\end{align*}
 Since  the functions $X_j^k\circ f $ and $X f^{1} ,\dots,Xf^r$ are continuous  (resp.   in $L^p$), then the functions
$X f^{m+1} ,\dots,Xf^n$ are   continuous (resp.  in $L^p$),
for all horizontal $X$.
Hence, $f^1,\ldots, f^n\in C^1_{\rm H,loc} (M)$ (resp.  in $W^{ 1,p}_{\rm H,loc} (M)$)
and 
 then $X_j^k\circ f \in C^1_{\rm H,loc} (M)$ (resp.  in $W^{ 1,p}_{\rm H,loc} (M)$).
 Notice that, if $f^1,\ldots, f^n\in C^1_{\rm H,loc} (M)$ then 
 on any compact $K$ the functions
 $ \nabla_{\rm H} f^1,\ldots, \nabla_{\rm H} f^n$ are   bounded, say by a constant $C$, therefore,
 for all horizontal curve $\sigma:[0,1]\to K$,
 $${\rm Length}(f(\sigma))=\int_0^1 \|f_* \sigma' \| ds \le
 C \int_0^1 \|  \sigma' \| ds = C{\rm Length}(\sigma).$$
Hence,
$f^1,\ldots, f^n\in C^1_{\rm H,loc} (M)$ implies
 that
 $f$ is Lipschitz and therefore its components are in $C^\alpha$.
Bootstrapping, we conclude that $f^1,\ldots, f^n$ is $C^{k,\alpha}_{\rm H,loc} (M)$ (resp.  in $W^{ k,p}_{\rm H,loc} (M)$).
\end{proof}


\section{Regularity of $1$-quasiconformal maps}\label{sec:regular}

In this section we prove Theorem \ref{duballe}.
Let us first clarify the definition of the function spaces involved.
Given two  equiregular sub-Riemannian manifolds $M,N$, we say that a homeomorphism $f$ is in $C^{1,\alpha}_{\rm H,loc}(M,N)\cap W^{2,2}_{\rm H,loc}(M,N)$ if, 
in any (smooth) coordinate system of $N$,
the components of $f$ 
   belong to $C^{1,\alpha}_{\rm H,loc}(M)\cap W^{2,2}_{\rm H,loc}(M)$.
\subsection{Every 1-quasiconformal map in $C^{1,\alpha}_{\rm H,loc}(M,N)\cap W^{2,2}_{\rm H,loc}(M,N)$  is conformal}\label{unapizza}



We now show that, 
assuming that a 1-quasiconformal map 
has the basic regularity, then the map is smooth.
The proof is independent from the results in 
Section~\ref{sec:coord}.
Namely, we do not need to assume any 
regularity theory for $Q$-Laplacian.

\begin{proof}[Proof of Theorem \ref{duballe}.(i)] Denote by $\vol_M$ and $\vol_N$ the Popp measures of $M$ and $N$.
For $p\in M$, consider any system of smooth coordinates $y^1,...,y^n$ in a neighborhood of $f(p)\in N$. Set $f^i:=y^i\circ f$ and $h^i:=L_Q(y^i)\in C^{\infty}(N)$.
From Corollary \ref{morphism-property}.(i), it follows that for all $u\in  
C^{\infty}_0(M)$
\begin{equation*}\label{pullback}
\int_{M} L_Q  (f^i) \,u \,\d\vol_M = \int_{M} h^i \circ f \ J_f^{\rm Popp} \,u\, \d\vol_M.
\end{equation*}
For $i=1,...,n$, set $H^i:=h^i\circ f \ J_f^{\rm Popp}$.
Since the Popp measures are smooth and 
 $f\in C^{1,\alpha}_{\rm H,loc}(M,N)$, we have that
 $J_f^{\rm Popp}\in C^{\alpha}_{\rm H,loc}(M)$ and therefore
 $H^i\in C^{\alpha}_{\rm H,loc}(M)$. At this point we have that $L_Q f^i\in C^{\alpha}_{\rm H,loc}(M)$ and that $f^i\in C^{1,\alpha}_{\rm H,loc}(M)\cap W^{2,2}_{\rm H,loc}(M)$. 
 Notice that $|\nabla_{\rm H}f^i|$ is bounded away from $0$,
since $f$ is bi-Lipschitz. 
Therefore, Proposition~\ref{smoothness} applies, yielding that $  f \in C^{2,\alpha}_{\rm H,loc}(M,N).$ 
 The proof follows by bootstrap using the Schauder estimates in Proposition~\ref{schauder}.
\end{proof}


\subsection{Regularity of $Q$-harmonic functions implies conformality}\label{sec:smoothness_conclusion}
We now reduce the smoothness assumption by  using horizontal $Q$-harmonic coordinates, see Section~\ref{sec:coord}. 
To ensure their existence and  to use them we  need to assume that the manifolds support the regularity theory for $Q$-Laplacian as defined in Definition~\ref{support}.



\begin{proof}[Proof of Theorem \ref{duballe}.(ii)]


We shall use Proposition~\ref{smoothfromhorizontal}. 
Since sub-Riemannian manifolds are $Q$-regular, by 
\cite{Heinonen-Koskela} any quasiconformal map is ACC$_Q$ (see also \cite[Corollary 6.5]{Margulis-Mostow})
In view of Theorem \ref{hQhc}, consider  $u_1,\ldots,u_n$ a system of local coordinates around  a point $f(p)\in M$
  for which the horizontal coordinates  $u^1,\ldots,u^r$  are
  $Q$-harmonic. 

In view of the morphism property (Corollary \ref{morphism-property})
the pull-backs $f_i=u_i\circ f$, for $i=1,\ldots,r$ are $Q$-harmonic functions in a neighborhood of $p\in M$.
By the $Q$-harmonic regularity assumption, 
both $u^i$
 and $f^i=u^i\circ f$ are in $C^{1,\alpha}_{\rm H,loc}(M) 
$, for $i=1,\ldots,r$. 
 Apply Proposition~\ref{smoothfromhorizontal} to $f$ with $k=1$ and get $ f \in C^{1,\alpha}_{\rm H,loc}(M,N)$.
 Since also $f^{-1}$ is 1-quasiconformal, the same argument shows that  $ f^{-1} \in C^{1,\alpha}_{\rm H,loc}(N,M)$.
 In particular, the map $f$ is bi-Lipschitz and $f^1 ,\ldots,f^n $  is a local system of bi-Lipschitz coordinates.
       In particular, $|\nabla_{\rm H} f^1|,\ldots, |\nabla_{\rm H} f^n|$ are   bounded away from zero. 
        Because of the $Q$-regularity hypothesis, 
        we have that 
        $f^1,\ldots,f^r$ are  in $W^{2,2}_{\rm H,loc}(M)
       $. 
       Invoking  Proposition~\ref{smoothfromhorizontal} once more, we have that 
               $f^1,\ldots,f^n$ are  in $W^{2,2}_{\rm H,loc}(M)
       $. 
\end{proof}

We remark that in the setting of Carnot groups 
both the existence of horizontal  $Q$-harmonic
coordinates
and the Lipschitz regularity of
1-quasiconformal
can be proven directly without using any PDE argument, see \cite{Pansu}.


\section{Liouville Theorem for contact sub-Riemannian manifolds}\label{contactsmooth}

\subsection{$Q$-Laplacian with respect to a divergence-free frame} 

In this section we intend to write the $Q$-Laplacian in a sub-Riemannian manifold using a horizontal frame that is not necessarily orthonormal, but is divergence-free with respect to some other volume form.
Recall that a vector field $X$ is divergence-free with respect to a volume form $\mu$ if its adjoint 
with respect to   $\mu$
equals $-X$.

Let $M$ be a  sub-Riemannian manifold equipped with a smooth volume form $\vol$.
Let $Y_1, \ldots, Y_r$ be an orthonormal frame for the horizontal distribution $HM$ of $M$.
Recall from 
\eqref{QLap:forte} that the $Q$-Laplacian of a twice differentiable function   is
\begin{eqnarray}\label{Coogee}
L_Q u    & = & \sum_i Y_i^*  \left(    \left( \sum_k (Y_k u)^{2} \right )^{\frac{Q-2}{2}}Y_i u\right)
\end{eqnarray}

Assume that there exists another frame $X_1, \ldots, X_r$ of $HM$ and another 
smooth volume form $\mu $ such that 
each $X_i$ is divergence-free with respect to 
$\mu $.
If $g$ is the sub-Riemannian metric of $M$, let 
$$g_{ij} := g(X_i, X_j)  \in C^\infty (M).$$
For all $x\in M$, let $g^{ij}(x)$ be the inverse matrix of $g_{ij}(x)$ and define the family of scalar products on $\R^r$ as
$$ \tilde g_x (v,w) := v_i g^{ij}(x) w_j,\qquad x\in M,\, v,w\in \R^r  .$$
Then there exists 
$a_i^j \in C^\infty (M)$ such that
\begin{equation}\label{4321}
Y_i= a_i^jX_j.
\end{equation}
So $\delta_{ij}= a_i^k a_j^l g_{kl} $ and $g^{ij} = a^i_k a ^j_k$.

Let $\omega$ be the smooth function such that 
$  \vol = \omega \mu $.
Since $X_i$ are divergence-free with respect to $\mu $, the 
adjoint vector fields with respect to $\vol$ of $Y_i$ are such that
$$Y_i^* u= X^*_j ( a _i^j u)  =  
-\omega^{-1}X_j (\omega  a _i^j u ) .
$$  
We use the notation
$$ \nabla_0u := (X_1 u, \ldots , X_r u ).$$
Noticing that $\sum_k (Y_k u)^{2}  = 
\tilde g (
\nabla_0u,
\nabla_0u),$
the expression \eqref{Coogee} becomes 
\begin{equation*}
(L_Q u)(x) = -\omega(x)^{-1} X_i A_i(x,\nabla_0 u),
\end{equation*}
where
\begin{equation}\label{Coogee2}
A_i(x, \xi) :=
\omega (x) \,\tilde g_x(\xi, \xi)^{\frac{Q-2}{2}} g^{ik}(x)\xi_k, \qquad \text{ for } \xi \in \R^r , x\in M.
\end{equation}
The derivatives of such functions are
$$\partial_{x_j}A_i(x,\xi)
=
\partial_{x_j} \omega 
\tilde g(\xi, \xi)^{\frac{Q-2}{2}} g^{ik}\xi_k
+
\omega\tfrac{Q-2}{2}
\tilde g(\xi, \xi)^{\frac{Q-2}{2}-1} 
 \partial_{x_j} g^{l,l'} \xi_l \xi_{l'}
 g^{ik}\xi_k
+
\omega
\tilde g(\xi, \xi)^{\frac{Q-2}{2}} \partial_{x_j} g^{ik}\xi_k
$$
and
$$  \partial_{\xi_j} A_i(x,\xi)  
=
\omega \left(  (Q-2) \tilde g(\xi, \xi)^{\frac{Q-4}{2}} g^{lj} g^{ik} \xi_l\xi_k
+ 
 \tilde g(\xi, \xi)^{\frac{Q-2}{2}} g^{ij}\right).
$$
Hence, 
$$ \partial_{\xi_j} A_i(x,\xi) \eta_i \eta_j 
=
\omega  \left(  (Q-2) \tilde g(\xi, \xi)^{\frac{Q-4}{2}}  \tilde g(\xi, \eta)^2
+ 
 \tilde g(\xi, \xi)^{\frac{Q-2}{2}}  \tilde g(\eta, \eta)\right).
$$
Using Cauchy-Schwarz inequality, 
the equivalence of norms in $\R^r$, and the smoothness of the functions $\omega$ and $g^{ij}$'s, the functions $A_i$ in \eqref{Coogee2} satisfy the following estimates:
on each compact set of $M$, for some
$\lambda, \Lambda>0$ depending only on $Q$, and for every $\chi  \in \R^r$,
\begin{equation}
\label{123412}
\lambda |\xi|^{{Q-2}} |\chi|^2\le  \partial_{\xi_j} A_i(x,\xi) \chi_i \chi_j \le \Lambda |\xi|^{{Q-2}} |\chi|^2
\end{equation}
 and
\begin{equation}
\label{9234287}
|\partial_{x_j}A_i(x,\xi)|\le \Lambda |\xi|^{{Q-1}}.\end{equation}

Summarizing, we have the following.
\begin{proposition}\label{6.6}
Let $M$ be a  sub-Riemannian manifold and consider  $\vol$ and $\mu$  two  smooth volume forms on $M$.
Assume there is  
a horizontal frame $X_1,\ldots, X_r$ on $M$ of
 vector fields that are  divergence-free with respect to $\mu$.
If $u$ 
is a function on $M$ that is $Q$-harmonic with respect to $\vol$,
 then $u$ satisfies 
\begin{equation*}
\sum_{i=1}^{r} X_i A_i(x,\nabla_0 u)=0,
\end{equation*}
for some $A_i$ for which \eqref{123412} and \eqref{9234287} hold.
\end{proposition}

\begin{remark}
In the above we used two different structures of metric measure space on the same manifold $M$. These are $(M,g,\vol)$ and  $(M,g_0,\mu)$, where $g_0$ is the metric for which $X_1,\dots,X_m$ form an orthonormal frame.
 For each of these structures we may define  corresponding Sobolev spaces 
$W^{p,q}_{\rm H}(M,g,\vol)$ and $W^{p,q}_{\rm H}(M,g_0,\mu)$. Similarly, we consider spaces
 $C^{1,\alpha}_{\rm H}(M,g)$ and $C^{1,\alpha}_{\rm H}(M,g_0)$. Since the 
 the matrix $(a_{i}^j)$
 in \eqref{4321}
  and its inverse have
  locally Lipschitz coefficients,
   it follows that on compact sets $\Omega\subset M$ the space
   $W^{p,2}_{\rm H}(\Omega,g,\vol)$ is biLipschitz to $W^{p,2}_{\rm H}(\Omega,g_0,\mu)$ for $p=1,2$, and  $C^{1,\alpha}_{\rm H}(\Omega,g)$ is biLipschitz to $C^{1,\alpha}_{\rm H}(\Omega,g_0)$.
\end{remark}

\subsection{Darboux coordinates on contact manifolds}\label{sec:contact}
On every contact manifold, the existence of a frame of divergence-free  vector fields with respect to some measure is ensured  by 
Darboux Theorem.
More generally, every sub-Riemannian manifold that is contactomorphic to a unimodular (e.g., nilpotent) Lie group equipped with a horizontal left-invariant distribution admits such a frame. The reason is that 
left-invariant vector fields are divergence-free with respect to the Haar measure of the group.
We shall recall now 
Darboux Theorem and we recall the standard contact structures, which are  those of the Heisenberg groups. 

Darboux Theorem states, see \cite{Etnyre_intro}, 
 that every two contact manifolds of the same dimension are locally contactomorphic. In particular, any 
contact  $2n+1$-manifold is locally contactomorphic to the standard contact structure  on $\R^{2n+1}$, a frame of which is given by 
\begin{equation}
\label{Heis:campi}X_i:=\p_{x_i}-\frac{x_{n+i}}{2}\p_{x_{2n+1}},\qquad  \qquad X_{n+i}:= \p_{x_{n+i}}+\frac{x_i}{2}\p_{x_{2n+1}}, 
\end{equation}
where $i=1,\dots,n$. For future reference we will also set $X_{2n+1}=\p_{x_{2n+1}}$.
This frame is   left-invariant for a specific Lie group structure, which we denote by $ \mathbb H^n$: 
the Heisenberg group.

\begin{corollary}(of Darboux Theorem)\label{dardar}
Let $M$ be a contact sub-Riemannian $2n+1$-manifold equipped with a volume form $\vol$.
There are local coordinates $x_1,\dots, x_{2n+1}$ in which
the horizontal distribution 
 is given by the vector fields in \eqref{Heis:campi}, which are divergence-free with respect to the Lebesgue measure $ \mathcal{L}$, and 
 there exists  $\omega\in C^\infty$ such that 
$\omega^{-1}\in C^\infty$ and $\d\vol=\omega \d \mathcal{L}$.
\end{corollary}

\subsection{Riemannian approximations}
Let us consider a contact  $2n+1$ manifold $M$, with subRiemannian metric $g_0$ and volume form $\vol$. Let 
$Y_1,...,Y_{2n}$ denote a $g_0$-orthonormal horizontal  frame in a neighborhood $\Omega\subset M$, and denote by $Y_{2n+1}$ the  Reeb vector field. For every $\e\in (0,1)$ we may define a $1$-parameter family of Riemannian metrics $g_\e$ on $M$ so that the frame $Y_1,..., Y_{2n}, \e Y_{2n+1}$ is orthonormal. Denote by $Y_1^\e,...,Y_{2n+1}^\e$ such $g_\e$-orthonormal frame.  For  $\e\ge 0$ and  $\delta\ge 0$ we will consider the family of regularized $Q$-Laplacian operators
\begin{eqnarray}\label{Coogee01}
L_Q^{\e,\delta} u    & := & \sum_{i=1}^{2n+1} Y_i^{\e *}  \left( \left(   \delta+  \sum_k (Y_k^\e u)^{2} \right )^{\frac{Q-2}{2}}Y_i^\e u\right)
\end{eqnarray}
Invoking Corollary \ref{dardar}, and applying  the same arguments as in Proposition \ref{6.6}, one can see that such  $Q$-Laplacian operators $L_Q^{\e,\delta} $,  can be written in the form 
\begin{equation}\label{zhong-e-eq}
L_{Q}^{\e,\delta} u = \sum_{i=1}^{2n+1} X_i^\e A_i^{\e, \delta} (x,\nabla_\e u)=0,
\end{equation}
where $X_i^\e=X_i$ for $i=1,...,2n$ and $X_{2n+1}^\e=\e X_{2n+1}$, with $X_1,...,X_{2n+1}$ as in \eqref{Heis:campi}.  Here we have set $\nabla_\e f=(X_1^\e f, ..., X_{2n+1}^\e f)$. The case $\e=\delta=0$ in \eqref{zhong-e-eq} reduces to the subelliptic $Q$-Laplacian. The components $A_i^{\e, \delta}$ in
\eqref{zhong-e-eq} are defined as in \eqref{Coogee2}, starting with the $g_\e$ metric, i.e., for every $\xi\in \R^{2n+1}$ and $x\in \Omega$,
 \begin{equation}\label{Coogee21}
A_i^{\e, \delta}(x, \xi) :=
\omega (x) \,(\delta+\tilde g_{\e,x}(\xi, \xi))^{\frac{Q-2}{2}} g^{ik}_\e (x)\xi_k.\end{equation}
By the same token as in \eqref{123412}, one has that there exists $\lambda,\Lambda>0$ depending  only on $Q$, such that the estimates
\begin{eqnarray}
\label{zhong-structure}
&
 \lambda (\delta+|\xi|^2)^{\frac{Q-2}{2}} |\chi|^2\le \sum_{i,j=1}^{2n+1} \partial_{\xi_j} A_i^{\e, \delta}(x,\xi) \chi_i \chi_j \le \Lambda (\delta+|\xi|^2)^{\frac{Q-2}{2}} |\chi|^2. &\\
\label{zhong-structure2}
&
|\partial_{x_j}A_i^{\e, \delta}(x,\xi)|\le \Lambda (\delta+|\xi|^2)^{\frac{Q-1}{2}}.&
\end{eqnarray}
hold for all $\e\ge 0$ and $\delta\ge 0$ and for all $\xi\in \R^{2n+1}$ and $\chi\in \R^{2n+1}$.

In the next section we prove that 
contact sub-Riemannian manifolds support regularity for $Q$-harmonic functions.
Hence, together with Theorem \ref{duballe}, this result will yield Theorem~\ref{thm:contact}. 

\subsection{$C^{1,\alpha}$ estimates after Zhong}

In this section we consider weak solutions $u\in W^{1,Q}_{\rm H, loc}(\Omega)$ of  $L_Q^0 u=0$, where $L_Q^0$ denotes the $Q$-Laplacian operator corresponding to a subRiemannian metric $g_0$ (not necessarily left-invariant)
%
in an open set $\Omega\subset \mathbb H^n$, endowed with its Haar measure, which coincides with the Lebesgue measure in $\R^{2n+1}$.
%
We prove the following theorem

\begin{theorem}\label{zhong-main} 
The following two properties hold:
\begin{enumerate} \item  
For every open $U\subset\subset \Omega$ and for every $\ell >0$, there exist constants $\alpha\in (0,1), C>0$ 
such that
for each 
$u\in W^{1,Q}_{\rm H, loc}(\Omega)$  weak solution of $L_Q^0 u=0$
with
$||u||_{W^{1,Q}_{\rm H}( U)} <\ell$, one has
\begin{equation*} 
 ||u||_{C^{1,\alpha}_{\rm H}( U)} \le C.
\end{equation*}
\item
For every open $U\subset \subset \Omega$
 and for every $\ell , \ell' >0$,
there exists a constant $C>0$ 
such that
for each 
$u\in W^{1,Q}_{\rm H, loc}(\Omega)$  weak solution of $L_Q^0 u=0$
with
$||u||_{W^{1,Q}_{\rm H}( K)} <\ell$ 
and
$\tfrac{1}{\ell'}<|\nabla_{\rm H} u|<\ell'$ on $U$,
one has
\begin{equation*} 
||u||_{W^{2,2}_{\rm H}(U)} 
\le C.
\end{equation*}
\end{enumerate}
\end{theorem}

This result is due to Zhong \cite{Zhong}, in the case when $g_0$ is a left invariant subRiemannian metric in $\mathbb H^n$. The proof in \cite{Zhong} breaks down with the additional dependence on $x$, in the coefficients of the equation as expressed in Proposition \ref{6.6}.  In fact,   in one of the approximations used  in \cite{Zhong}, the argument relies on the existence of explicit barrier functions, which one does not have in our setting. To deal with this issue we follow the method recently used in \cite{CCGM} where a Riemannian approximation scheme was employed to carry out the corresponding regularization for evoluationary equations. Apart from this aspect the arguments in \cite{Zhong} apply to the present setting as well. Note that the H\"older regularity of the solution $u$ is considerably simpler (see for instance \cite{CDG}).

\begin{remark} The proof in \cite{Zhong} applies to any Carnot group of step two, and likewise the conclusion of Theorem \ref{zhong-main} continues to hold in this more general setting.
\end{remark}

\subsubsection{Riemannian approximation}  Throughout the rest of the section we will assume $\delta>0$ and let $u$ denote a solution of $L_Q^0 u=0$ in $\Omega\subset \mathbb H^n$. For $\e>0$ we consider $ W^{k,p}_{\e, {\rm loc}}$ and $C^{k,\al}_\e$ to be the  Sobolev and H\"older spaces corresponding to the frame $X_1^\e,...,X_{2n+1}^\e$.
%
%
Observe that by virtue of classical elliptic theory  (see for instance \cite{LU} ) for $\delta>0$ one has that the weak solutions $u^\e\in W^{1,Q}_{\e, {\rm loc}}(\Omega)$ of \eqref{zhong-e-eq} 
are in fact smooth in $\Omega$. For a fixed ball $D\subset \subset \Omega$ and for any $\e\ge 0$, standard PDE arguments (see for instance \cite{HKM}) yield  the existence and unicity of the solution to the Dirichlet problem
\begin{equation}\label{zhong-diri}
\begin{cases}
L_{Q}^{\e,\delta}u^\e=0  \text{ in }D\\
u^\e-u \in W^{1,Q}_{\e,0}(D).
\end{cases}
\end{equation}
Although the smoothness of $u^\e$ may degenerate as $\e\to 0$ and $\delta\to 0$, we will show that the estimates on the H\"older norm of the gradient do not depend on these parameters and hence will hold uniformly in the limit. Note that in view of the Caccioppoli inequality and of the uniform bounds on the H\"older norm of $u^\e$ as $\e\to 0$ (such bounds depend only on the stability of the Poincar\'e inequality and on the doubling constants of the Riemannian Heisenberg groups $(\mathbb H^n, g_\e)$ which are stable in view of \cite{CCR}), one has that for any $K\subset \subset D$ there exists a constant $M_{K,Q}>0$ depending only on $Q, K$ such that 
\begin{equation*}
||\nabla_\e u^\e||_{L^Q(K)}\le M_{K,Q}.
\end{equation*}
The next proposition addresses the non trivial uniform bounds.
\begin{proposition}\label{zhong-prop} 
The following two properties hold:
\begin{enumerate} \item  
For every open $U\subset\subset D$ and for every $\ell >0$, there exist constants $\alpha\in (0,1), C>0$ 
such that
if
$u^\e\in W^{1,Q}_{\e, {\rm loc}}(D)\cap C^{\infty}(D)$  is the unique  solution of \eqref{zhong-diri} 
with
$||u||_{W^{1,Q}_{\rm H}( D)} <\ell$, then one has
\begin{equation*}
 ||u^\e||_{C^{1,\alpha}_\e( U)} \le C,\qquad \forall \eps>0.
\end{equation*}
\item
For every  open $U\subset \subset D$
 and for every $\ell , \ell' >0$,
there exists a constant $C>0$ 
such that
if
$u^\e\in W^{1,Q}_{\e, {\rm loc}}(D)\cap C^{\infty}(D)$ is the  unique  solution of \eqref{zhong-diri} 
with $||u||_{W^{1,Q}_{\rm H}( D)} <\ell$, and
$\tfrac{1}{\ell'}<|\nabla_\e u^\e|<\ell'$ on $U$, then 
one has
\begin{equation*}
||u^\e||_{W^{2,2}_\e(U)} 
\le C,\qquad \forall \eps>0.
\end{equation*}
\end{enumerate}
\end{proposition}

The main regularity result Theorem \ref{zhong-main} then follows from Proposition \ref{zhong-prop}, by means of Ascoli-Arzela theorem and the uniqueness of the Dirichlet problem \eqref{zhong-diri} when $\e=0$.

The proof of Proposition \ref{zhong-prop} follows very closely the arguments in \cite{Zhong}. For the reader's convenience we reproduce them in the two sections below. For the sake of notation's simplicity, and without  any loss of generality, we will just present the proof in the case $n=1$.

\subsubsection{Uniform Lipschitz regularity} The aim of this section is to establish Lipschitz estimates that are uniform as $\e\to 0$, on a open ball $B\subset \subset D$. 

\begin{theorem}\label{lipschitzestimate}
Let  $u^\e\in W^{1,Q}_{\e, {\rm loc}}(D)\cap C^{\infty}(D)$ be the unique  solution of \eqref{zhong-diri}.   
If $B\subset  2B \subset \subset D$  then there exists $C>0$, depending only on $Q, \Lambda, \lambda$ of \eqref{zhong-structure} and \eqref{zhong-structure2}, such that
\begin{equation*}
\sup_B |\nabla_\e u^\e| \le C \left( \frac{1}{\mathcal{L}(2B)}\int_{2B} (\delta+ |\nabla_\e u^\e|^2)^{\frac{Q}{2}}  \right)^{\frac{1}{Q}},
\end{equation*}
where $2B$ denotes the ball with the same center of $B$ and twice the radius.
\end{theorem}

The proof of this theorem is developed across several lemmata in this section.

For $\e,\delta>0$ and $i=1,2,3$ set $v_i=X_i^\e u^\e$ and observe that by differentiating \eqref{zhong-e-eq} along  $X^\e_i$, $i=1,2,3$  one has
\begin{multline}\label{zhong-d-eq1}
\sum_{i,j=1}^3 X_i^\e \bigg( A_{i,\xi_j}^{\e,\delta} (x, \nabla_\e u^\e) X_j^\e v_1\bigg) +\sum_{i=1}^3  X_i^\e \bigg( A_{i,\xi_2}^{\e,\delta} (x,\nabla_\e u^\e) X_3 u^\e  \bigg) + X_3\bigg( A_2^{\e,\delta} (x,\nabla_\e u^\e)\bigg) 
\\+\sum_{i=1}^3X_i^\e \bigg(A^{\e,\delta}_{i,x_1}(x,\nabla_\e u^\e) -\frac{x_2}{2}A^{\e,\delta}_{i,x_3}(x,\nabla_\e u^\e)  \bigg)=0; 
\end{multline}
\begin{multline}\label{zhong-d-eq2}
\sum_{i,j=1}^3 X_i^\e  \bigg( A_{i,\xi_j}^{\e,\delta}  (x, \nabla_\e u^\e) X_j^\e v_2 \bigg) -\sum_{i=1}^3 X_i^\e \bigg( A_{i,\xi_1}^{\e,\delta} (x,\nabla_\e u^\e) X_3 u^\e  \bigg) -X_3\bigg( A_1^{\e,\delta} (x,\nabla_\e u^\e)\bigg) \\+\sum_{i=1}^3X_i^\e \bigg(A^{\e,\delta}_{i,x_2}(x,\nabla_\e u^\e) +\frac{x_1}{2}A^{\e,\delta}_{i,x_3}(x,\nabla_\e u^\e)  \bigg)=0; \end{multline}
and
\begin{multline}\label{zhong-d-eq3}
\sum_{i,j=1}^3 X_i^\e \bigg(A_{i,\xi_j}^{\e,\delta}  (x, \nabla_\e u^\e) X_j^\e v_3 \bigg)
+\e\sum_{i=1}^3 X_i^\e \bigg(A^{\e,\delta}_{i,x_3}(x,\nabla_\e u^\e) \bigg)=0.
\end{multline}

\begin{remark}
Note that the terms containing $X_3$ in the equations above are not bounded as $\e\to 0$ in the $g_\e$ metric. In the following it will be crucial to obtain estimates that are stable as $\e\to 0$.
\end{remark}

The following results were originally proved for the case with no dependence of $x$, in \cite[Theorem 7]{MR2336058}, \cite[Lemma 5.1]{MR2531368} and then again in \cite{Zhong} with a more direct argument bypassing the difference quotients method. The proofs in our setting are very similar and we omit most of the details.

\begin{lemma}\label{zhong-l1}
For every $\beta \ge 0$ and $\eta\in C^{\infty}_0(B)$ one has 
\begin{multline*}
\int_B (\delta+|\nabla_\e u^\e|^2)^{\frac{Q-2}{2}} |\nabla_\e v_3|^2 |v_3|^\beta \eta^2  \d\mathcal{L} \le \bigg(\frac{2\Lambda}{\lambda (\beta+1)} +2\Lambda
\bigg) \int_B (\delta+|\nabla_\e u^\e|^2)^{\frac{Q-2}{2}} |\nabla_\e \eta|^2 |v_3|^{\beta +2} \d\mathcal{L}
\\+ 2 \e^2\Lambda \bigg(1+\frac{1}{\lambda (\beta+1)^2}\bigg) \int_B  (\delta+|\nabla_\e u^\e|^2)^{\frac{Q}{2}}  |v_3|^\beta \eta^2 \d\mathcal{L}.
\end{multline*}
\end{lemma}
\begin{proof} Multiply both sides of 
\eqref{zhong-d-eq3} 
by  $\phi=\eta^2 |X_3^\e u^\e|^{\beta} X_3^\e u^\e$
and integrate over $B$. The result follows in a standard way from Young's inequality and from the structure conditions \eqref{zhong-structure}. \end{proof}

Note that dividing both sides of
the inequality above
by $\e^{\beta+2}$ and letting $\beta\to 0$ one recovers the Manfredi-Mingione original lemma (see for instance \cite[Lemma 3.3]{Zhong}).

\begin{lemma}\label{zhong-l2}
For every $\beta \ge 0$ and $\eta\in C^{\infty}_0(B)$ one has 
\begin{multline*}
\int_B (\delta+|\nabla_\e u^\e|^2)^{\frac{Q-2+\beta}{2}} \sum_{i,j=1}^{3} |X_i^\e X_j^\e u^\e|^2  \eta^2  \d\mathcal{L} \le C (\beta+1)^4 \int_B (\delta+|\nabla_\e u^\e|^2)^{\frac{Q-2+\beta}{2}} |X_3 u^\e|^2 \eta^2  \d\mathcal{L} \\+
C \int_B (\eta^2+|\nabla_\e \eta|^2) (\delta+|\nabla_\e u^\e|^2)^{\frac{Q+\beta}{2}}  \d\mathcal{L}
+ C \int_B  \eta^2 (\delta+|\nabla_\e u^\e|^2)^{\frac{Q+\beta+1}{2}}  \d\mathcal{L},
\end{multline*}
for some constant $C=C(\lambda, \Lambda)>0$.

\end{lemma}
\begin{proof} Following the arguments in \cite{Zhong}, we multiply both sides of  \eqref{zhong-d-eq1}, \eqref{zhong-d-eq2} and \eqref{zhong-d-eq3}  by  $\phi=\eta^2  (\delta+|\nabla_\e u^\e|^2)^{\frac{\beta}{2}} v_i $ for $i=1,2,3$
and integrate over $B$ to obtain
\begin{multline*} \int_B \sum_{i,j=1}^3 A_{i,\xi_j}^{\e,\delta} (x, \nabla_\e u^\e) X_j^\e v_1 X_i^\e (\eta^2  (\delta+|\nabla_\e u^\e|^2)^{\frac{\beta}{2}} v_1 )    \d\mathcal{L}
\\+\int_B \sum_{i=1}^3   A_{i,\xi_2}^{\e,\delta}(x,\nabla_\e u^\e) X_3 u^\e   X_i^\e(\eta^2  (\delta+|\nabla_\e u^\e|^2)^{\frac{\beta}{2}} v_1 )    \d\mathcal{L}\\ + 
\int_B A_2^{\e,\delta} (x,\nabla_\e u^\e)X_3 (\eta^2  (\delta+|\nabla_\e u^\e|^2)^{\frac{\beta}{2}} v_1 )  \d\mathcal{L}
\\+\sum_{i=1}^3 \big[ A^{\e,\delta}_{i,x_1}(x,\nabla_\e u^\e) -\frac{x_2}{2}A^{\e,\delta}_{i,x_3}(x,\nabla_\e u^\e)  \big]  X_i^\e (\eta^2  (\delta+|\nabla_\e u^\e|^2)^{\frac{\beta}{2}} v_1 )\d\mathcal{L}=0; 
\end{multline*}

\begin{multline*} \int_B \sum_{i,j=1}^3 A_{i,\xi_j}^{\e,\delta} (x, \nabla_\e u^\e) X_j^\e v_2 X_i^\e (\eta^2  (\delta+|\nabla_\e u^\e|^2)^{\frac{\beta}{2}} v_2 )    \d\mathcal{L}
\\-\int_B \sum_{i=1}^3   A_{i,\xi_1}^{\e,\delta}(x,\nabla_\e u^\e) X_3 u^\e   X_i^\e(\eta^2  (\delta+|\nabla_\e u^\e|^2)^{\frac{\beta}{2}} v_2 )    \d\mathcal{L}\\ -
\int_B A_1^{\e,\delta} (x,\nabla_\e u^\e)X_3 (\eta^2  (\delta+|\nabla_\e u^\e|^2)^{\frac{\beta}{2}} v_2 )  \d\mathcal{L}
\\+\sum_{i=1}^3 \big[ A^{\e,\delta}_{i,x_2}(x,\nabla_\e u^\e) +\frac{x_1}{2}A^{\e,\delta}_{i,x_3}(x,\nabla_\e u^\e)  \big]  X_i^\e (\eta^2  (\delta+|\nabla_\e u^\e|^2)^{\frac{\beta}{2}} v_2 )\d\mathcal{L}=0; 
\end{multline*}
and
\begin{multline*} \int_B\sum_{i,j=1}^3 A_{i,\xi_j}^{\e,\delta} (x, \nabla_\e u^\e) X_j^\e v_3X_i^\e (\eta^2  (\delta+|\nabla_\e u^\e|^2)^{\frac{\beta}{2}} v_3 )    \d\mathcal{L}
\\
+\e\sum_{i=1}^3 A^{\e,\delta}_{i,x_3}(x,\nabla_\e u^\e) X_i^\e (\eta^2  (\delta+|\nabla_\e u^\e|^2)^{\frac{\beta}{2}} v_3 )    \d\mathcal{L}
=0.
\end{multline*}
Combining the previous identities we have
\begin{multline}\label{1.21}
 \sum_{i,j,k=1}^3\int_B A_{i,\xi_j}^{\e,\delta} (x, \nabla_\e u^\e) X_j^\e v_k X_i^\e ( (\delta+|\nabla_\e u^\e|^2)^{\frac{\beta}{2}} v_k )   \eta^2  \d\mathcal{L}
\\= -2 \sum_{i,j,k=1}^3 \int_B  \eta A_{i,\xi_j}^{\e,\delta} (x, \nabla_\e u^\e) X_j^\e v_k   X_i^\e \eta  (\delta+|\nabla_\e u^\e|^2)^{\frac{\beta}{2}} v_k   \d\mathcal{L}- \mathcal A - \mathcal B - \mathcal C,
\end{multline}
where
\begin{multline*}\mathcal A= \int_B \sum_{i=1}^3   A_{i,\xi_2}^{\e,\delta}(x,\nabla_\e u^\e) X_3 u^\e   X_i^\e(\eta^2  (\delta+|\nabla_\e u^\e|^2)^{\frac{\beta}{2}} v_1 )    \d\mathcal{L} 
\\-\int_B \sum_{i=1}^3   A_{i,\xi_1}^{\e,\delta}(x,\nabla_\e u^\e) X_3 u^\e   X_i^\e(\eta^2  (\delta+|\nabla_\e u^\e|^2)^{\frac{\beta}{2}} v_2 )    \d\mathcal{L},
\end{multline*}
\begin{multline*}
\mathcal B= \int_B A_2^{\e,\delta} (x,\nabla_\e u^\e)X_3 (\eta^2  (\delta+|\nabla_\e u^\e|^2)^{\frac{\beta}{2}} v_1 )  \d\mathcal{L}
\\ -
\int_B A_1^{\e,\delta} (x,\nabla_\e u^\e)X_3 (\eta^2  (\delta+|\nabla_\e u^\e|^2)^{\frac{\beta}{2}} v_2 )  \d\mathcal{L},
\end{multline*}
and 
\begin{multline*}
\mathcal C= \sum_{i=1}^3 \int_B \big[ A^{\e,\delta}_{i,x_1}(x,\nabla_\e u^\e) -\frac{x_2}{2}A^{\e,\delta}_{i,x_3}(x,\nabla_\e u^\e)  \big]  X_i^\e (\eta^2  (\delta+|\nabla_\e u^\e|^2)^{\frac{\beta}{2}} v_2 )\d\mathcal{L}
\\+ \sum_{i=1}^3\int_B  \big[ A^{\e,\delta}_{i,x_2}(x,\nabla_\e u^\e) +\frac{x_1}{2}A^{\e,\delta}_{i,x_3}(x,\nabla_\e u^\e)  \big]  X_i^\e (\eta^2  (\delta+|\nabla_\e u^\e|^2)^{\frac{\beta}{2}} v_2 )\d\mathcal{L}
\\+\e\sum_{i=1}^3 \int_B A^{\e,\delta}_{i,x_3}(x,\nabla_\e u^\e) X_i^\e (\eta^2  (\delta+|\nabla_\e u^\e|^2)^{\frac{\beta}{2}} v_3 )    \d\mathcal{L}.
\end{multline*}

In view of the structure conditions \eqref{zhong-structure} one has
that in the left hand side of \eqref{1.21} the following inequalities hold
\begin{multline*}
\int_B (\delta+|\nabla_\e u^\e|^2)^{\frac{Q-2+\beta}{2}} \sum_{i,j=1}^{3} |X_i^\e X_j^\e u^\e|^2  \eta^2  \d\mathcal{L} \\ \le C\sum_{i,j,k=1}^3\int_B A_{i,\xi_j}^{\e,\delta} (x, \nabla_\e u^\e) X_j^\e v_k X_i^\e ( (\delta+|\nabla_\e u^\e|^2)^{\frac{\beta}{2}} v_k )   \eta^2  \d\mathcal{L}
\end{multline*}
and for all $\tau>0$,
\begin{multline*}
\sum_{i,j,k=1}^3 \int_B  \eta A_{i,\xi_j}^{\e,\delta} (x, \nabla_\e u^\e) X_j^\e v_k   X_i^\e \eta  (\delta+|\nabla_\e u^\e|^2)^{\frac{\beta}{2}} v_k   \d\mathcal{L} \\
\le \int_B (\delta+|\nabla_\e u^\e|^2)^{\frac{Q-2+\beta}{2}} \sum_{i,j=1}^{3} |X_i^\e X_j^\e u^\e|^2  \eta |\nabla_\e \eta| |\nabla_\e u^\e|  \d\mathcal{L} 
\\
\le \tau \int_B (\delta+|\nabla_\e u^\e|^2)^{\frac{Q-2+\beta}{2}} \sum_{i,j=1}^{3} |X_i^\e X_j^\e u^\e|^2  \eta^2  \d\mathcal{L} + \tau^{-1} \int_B (\delta+|\nabla_\e u^\e|^2)^{\frac{Q+\beta}{2}} |\nabla_\e \eta|^2  \d\mathcal{L}.
\end{multline*}

Next we estimate $\mathcal A, \mathcal B$ and $\mathcal C$. Using the structure conditions \eqref{zhong-structure} we obtain
\begin{multline*} |\mathcal A|  \le C\int_B   (\delta+|\nabla_\e u^\e|^2)^{\frac{Q-1+\beta}{2}}  |X_3u| \eta |\nabla_\e \eta| \d\mathcal{L} \\+ C(\beta+1) \int_B   (\delta+|\nabla_\e u^\e|^2)^{\frac{Q-2+\beta}{2}} |X_3 u|   \sum_{i,j=1}^{3} |X_i^\e X_j^\e u^\e|^2  \eta^2 \d\mathcal{L}.
\end{multline*}
For $\mathcal B$ we integrate by parts and obtain
\begin{multline*}
\mathcal B=  2\int_B A_2^{\e,\delta} (x,\nabla_\e u^\e) \eta X_3 \eta (\delta+|\nabla_\e u^\e|^2)^{\frac{\beta}{2}} v_1   \d\mathcal{L} +\beta \int_B A_2^{\e,\delta} (x,\nabla_\e u^\e)\eta^2  (\delta+|\nabla_\e u^\e|^2)^{\frac{\beta-2}{2}}  X_3 X_k^\e u^\e X_k^\e u^\e v_1  \d\mathcal{L}
\\+\int_B A_2^{\e,\delta} (x,\nabla_\e u^\e) \eta^2  (\delta+|\nabla_\e u^\e|^2)^{\frac{\beta}{2}}X_3 v_1   \d\mathcal{L}
 -2 \int_B A_1^{\e,\delta} (x,\nabla_\e u^\e)\eta X_3 \eta  (\delta+|\nabla_\e u^\e|^2)^{\frac{\beta}{2}} v_2  \d\mathcal{L} \\-\beta \int_B A_1^{\e,\delta} (x,\nabla_\e u^\e) \eta^2  (\delta+|\nabla_\e u^\e|^2)^{\frac{\beta-2}{2}} X_3 X_k^\e u^\e X_k^\e u^\e v_2 )  \d\mathcal{L}- \int_B A_1^{\e,\delta} (x,\nabla_\e u^\e) \eta^2  (\delta+|\nabla_\e u^\e|^2)^{\frac{\beta}{2}} X_3 v_2  \d\mathcal{L}
 \\=2\int_B A_2^{\e,\delta} (x,\nabla_\e u^\e) \eta X_3 \eta (\delta+|\nabla_\e u^\e|^2)^{\frac{\beta}{2}} v_1   \d\mathcal{L} -\beta \int_B X_k^\e \bigg(A_2^{\e,\delta} (x,\nabla_\e u^\e)\eta^2  (\delta+|\nabla_\e u^\e|^2)^{\frac{\beta-2}{2}}   X_k^\e u^\e v_1\bigg) X_3  u^\e  \d\mathcal{L}
\\-\int_B X_1^{\e}\bigg(A_2^{\e,\delta} (x,\nabla_\e u^\e) \eta^2  (\delta+|\nabla_\e u^\e|^2)^{\frac{\beta}{2}}\bigg) X_3 u^\e   \d\mathcal{L}
 -2 \int_B A_1^{\e,\delta} (x,\nabla_\e u^\e)\eta X_3 \eta  (\delta+|\nabla_\e u^\e|^2)^{\frac{\beta}{2}} v_2  \d\mathcal{L} \\+\beta \int_B X_k^\e \bigg(A_1^{\e,\delta} (x,\nabla_\e u^\e) \eta^2  (\delta+|\nabla_\e u^\e|^2)^{\frac{\beta-2}{2}}  X_k^\e u^\e v_2 \bigg) X_3  u^\e \d\mathcal{L} +\int_B X_2^\e \bigg(A_1^{\e,\delta} (x,\nabla_\e u^\e) \eta^2  (\delta+|\nabla_\e u^\e|^2)^{\frac{\beta}{2}} \bigg)X_3 u^\e \d\mathcal{L}.
\end{multline*}
The structure conditions \eqref{zhong-structure} then yield
\begin{multline*}
|\mathcal B| \le C(\beta+1)^2 \int_B  (\delta+|\nabla_\e u^\e|^2)^{\frac{Q-2+\beta}{2}}  |X_3 u^\e|  \sum_{i,j=1}^{3} |X_i^\e X_j^\e u^\e|  \eta^2 \d\mathcal{L}\\ +C(\beta+1)
\int_B (\delta+|\nabla_\e u^\e|^2)^{\frac{Q-1+\beta}{2}}  |X_3 u^\e|  \eta |\nabla_\e \eta|  \d\mathcal{L}
\\+C\int_B (\delta+|\nabla_\e u^\e|^2)^{\frac{Q+\beta}{2}}  \eta |X_3  \eta|  \d\mathcal{L}.
\end{multline*}
Regarding $\mathcal C$, the structure conditions immediately imply
\begin{multline*}
|\mathcal C|\le C\int_B (\delta+|\nabla_\e u^\e|^2)^{\frac{Q+\beta}{2}}  \eta |\nabla_\e  \eta|  \d\mathcal{L} + C(\beta+1) \int_B  (\delta+|\nabla_\e u^\e|^2)^{\frac{Q-1+\beta}{2}} \sum_{i,j=1}^{3} |X_i^\e X_j^\e u^\e|  \eta^2 \d\mathcal{L}.
\end{multline*}
The proof  of the lemma now follows immediately from the previous inequalities and from Young inequality.
 \end{proof}

The next step provides a crucial reverse H\"older-type inequality.
\begin{lemma}\label{zhong-l3}
For every $\beta \ge 2$ and $\eta\in C^{\infty}_0(B)$ one has 
\begin{multline*}
\int_B (\delta+|\nabla_\e u^\e|^2)^{\frac{Q-2}{2}} |X_3^\e u^\e|^\beta \sum_{i,j=1}^{3}  |X_i^\e X_j^\e u^\e|^2  \eta^{\beta+2}  \d\mathcal{L}  
\\  \le C \  (\beta+1)^2 ||\nabla_\e \eta||_{L^\infty(B)}^2 \bigg( \e^2\int_B
 (\delta+|\nabla_\e u^\e|^2)^{\frac{Q}{2}} |X_3^\e u^\e|^{\beta-2} \sum_{i,j=1}^{3}  |X_i^\e X_j^\e u^\e|^2 ( \eta^{\beta+2} +\eta^\beta)   \d\mathcal{L}
   \\+  \e^\beta\int_B (\delta+|\nabla_\e u^\e|^2)^{\frac{Q+\beta}{2}} \eta^{\beta}  d\mathcal{L} \bigg).
\end{multline*}
\end{lemma}
Note  that dividing by $\e^{\beta}$ and letting $\e\to 0$ one recovers Zhong's estimate.
\begin{proof} Differentiating \eqref{zhong-e-eq} along $X_1^\e$, recalling that $[X_1^\e,X_2^\e]=X_3$,  and multiplying by a test function $\phi \in C^{\infty}_0(B)$ yields
\begin{equation}\label{zhong-3.6}
\int_B X_1^\e A_i^{\e, \delta}(x,\nabla_\e u^\e) X_i^\e \phi  \d\mathcal{L} = \int_B X_3 A_2^{\e,\delta}(x,\nabla_\e u^\e) \phi  \d\mathcal{L}.
\end{equation}
Next set $\phi=\eta^{\beta+2} |X_3^\e u^\e|^\beta X_1^\e u^\e$ in the previous identity to obtain in the left-hand side
\begin{multline*}
\int_B X_1^\e A_i^{\e, \delta}(x,\nabla_\e u^\e) X_i^\e \phi  \d\mathcal{L} =
\int_B A_{i,\xi_j}^{\e,\delta}(x,\nabla_\e u^\e)  X_1^\e X_j^\e u^\e X_1^\e X_i^\e u^\e\eta^{\beta+2} |X_3^\e u^\e|^\beta  \d\mathcal{L}
\\ -\int_B   X_1^\e A_2^{\e,\delta}(x,\nabla_\e u^\e) 
 X_3u^\e \eta^{\beta+2} |X_3^\e u^\e|^\beta  \d\mathcal{L}  \\ + \beta \int_B X_1^\e A_i^{\e, \delta}(x,\nabla_\e u^\e) X_i^\e X_3^\e u^\e |X_3^\e u^\e|^{\beta-2}  X_3^\e u^\e  X_1^\e u^\e \eta^{\beta+2}  \d\mathcal{L} \\ + (\beta+2) \int_B X_1^\e A_i^{\e, \delta}(x,\nabla_\e u^\e) X_i^\e \eta  |X_3^\e u^\e|^{\beta} X_1^\e u^\e \eta^{\beta+1}  \d\mathcal{L}.
\end{multline*}
Substituting in \eqref{zhong-3.6} and using the structure conditions \eqref{zhong-structure} one obtains
\begin{multline}\label{zhong-1.22}
\int_B (\delta+|\nabla_\e u^\e|^2)^{\frac{Q-2}{2}} |X_3^\e u^\e|^\beta \sum_{j=1}^{3}  |X_1^\e X_j^\e u^\e|^2  \eta^{\beta+2}  \d\mathcal{L}  
\\ \le \int_B   X_1^\e A_2^{\e,\delta}(x,\nabla_\e u^\e) 
X_3u^\e  \eta^{\beta+2} |X_3^\e u^\e|^\beta  \d\mathcal{L}  
\\ + \beta \int_B |X_1^\e A_i^{\e, \delta}(x,\nabla_\e u^\e) X_i^\e X_3^\e u^\e |X_3^\e u^\e|^{\beta-1} X_1^\e u^\e \eta^{\beta+2}|  \d\mathcal{L} \\ + (\beta+2) \int_B |X_1^\e A_i^{\e, \delta}(x,\nabla_\e u^\e) X_i^\e \eta  |X_3^\e u^\e|^{\beta} X_1^\e u^\e \eta^{\beta+1} |  \d\mathcal{L}
\\+ \int_B |X_3 A_2^{\e,\delta}(x,\nabla_\e u^\e) \eta^{\beta+2} |X_3^\e u^\e|^\beta X_1^\e u^\e | \d\mathcal{L}
\end{multline}
\begin{multline*}
\le 
 \sum_{h,k=1}^3 \bigg|\int_B   X_h^\e A_k^\e(x,\nabla_\e u^\e) 
X_3u^\e  \eta^{\beta+2} |X_3^\e u^\e|^\beta  \d\mathcal{L}  \bigg|
\\ + \beta \int_B |\nabla_\e A_i^{\e, \delta}(x,\nabla_\e u^\e) ||\nabla_\e X_3^\e u^\e | |X_3^\e u^\e|^{\beta-1} |\nabla_\e  u^\e| \eta^{\beta+2}  \d\mathcal{L} \\ + (\beta+2) \int_B |\nabla_\e A_i^{\e, \delta}(x,\nabla_\e u^\e)| |\nabla_\e \eta |  X_3^\e u^\e|^{\beta} |\nabla_\e  u^\e| \eta^{\beta+1}  \d\mathcal{L}
\\+ \int_B\sum_{j=1}^2  |X_3 A_j^{\e,\delta} (x,\nabla_\e u^\e) | \eta^{\beta+2} |X_3^\e u^\e|^\beta |\nabla_\e u^\e | \d\mathcal{L}
= I_1+I_2+I_3+I_4.
\end{multline*}
In a similar fashion, differentiating  \eqref{zhong-e-eq} along $X_2^\e$ and $X_3^\e$, and using the test function 
$\phi=\eta^{\beta+2} |X_3^\e u^\e|^\beta X_h^\e u^\e$ with $h=2,3$, one arrives at a similar estimate for $X_h^\e X_j^\e u^\e$ in the left-hand side. The combination of such estimate and \eqref{zhong-1.22} yields
\begin{equation*}
\int_B (\delta+|\nabla_\e u^\e|^2)^{\frac{Q-2}{2}} |X_3^\e u^\e|^\beta \sum_{i,j=1}^{3}  |X_i^\e X_j^\e u^\e|^2  \eta^{\beta+2}  \d\mathcal{L}   \le I_1+I_2+I_3+I_4.
\end{equation*}
Next, for any $\tau>0$,  we estimate each single component $|I_k|$ in the following way
\begin{multline}\label{mustbelikethis}
|I_h|\le \tau \int_B (\delta+|\nabla_\e u^\e|^2)^{\frac{Q-2}{2}} |X_3^\e u^\e|^\beta \sum_{i,j=1}^{3}  |X_i^\e X_j^\e u^\e|^2  \eta^{\beta+2}  \d\mathcal{L}  \\+ \e^2 \frac{C(\beta+1)^2 ||\nabla_\e \eta||_{L^\infty(B)}^2}{\tau}
\int_B
 (\delta+|\nabla_\e u^\e|^2)^{\frac{Q}{2}} |X_3^\e u^\e|^{\beta-2} \sum_{i,j=1}^{3}  |X_i^\e X_j^\e u^\e|^2 ( \eta^{\beta+2} +\eta^\beta)   \d\mathcal{L}
   \\+ \tau^{-1}\e^2 \int_B (\delta+|\nabla_\e u^\e|^2)^{\frac{Q+\beta}{2}} \eta^{\beta}  d\mathcal{L} ,
\end{multline}
from which the conclusion will follow immediately. We begin by looking at $I_1$. 
%
\begin{itemize} \item {\bf Estimate of $I_1$.} Proceeding as in \cite{Zhong} we integrate by parts to obtain
\begin{multline}\label{I}  \int_B   X_h^\e A_k^{\e,\delta} (x,\nabla_\e u^\e) 
X_3u^\e  \eta^{\beta+2} |X_3^\e u^\e|^\beta  \d\mathcal{L}  = - \int_B   A_k^{\e,\delta} (x,\nabla_\e u^\e) 
X_h^\e \bigg( X_3u^\e  \eta^{\beta+2} |X_3^\e u^\e|^\beta\bigg)  \d\mathcal{L}  \\
\\
= -\e^{-1} (\beta+1)  \int_B   A_k^{\e,\delta} (x,\nabla_\e u^\e)\eta^{\beta+2} |X_3^\e u^\e|^\beta X_3^\e X_h^\e u^\e  \d\mathcal{L} 
\\-(\beta+2) \int_B A_k^{\e,\delta} (x,\nabla_\e u^\e)\eta^{\beta+1} X_h^\e \eta |X_3^\e u^\e|^\beta X_3u^\e  \d\mathcal{L}= \mathcal I + \mathcal {II}.
\end{multline}
\begin{itemize} \item{\bf Estimate of  $\mathcal I$.} Using Young inequality one has
\begin{multline*}
\left| \e^{-1} (\beta+1)  \int_B   A_k^{\e,\delta} (x,\nabla_\e u^\e)\eta^{\beta+2} |X_3^\e u^\e|^\beta X_3^\e X_h^\e u^\e  \d\mathcal{L}\right| \\
\le  \e^{-1} (\beta+1) \int_B (\delta+|\nabla_\e u^\e|^2)^{\frac{Q-1}{2}} |X_3^\e u^\e|^\beta |\nabla_\e X_3^\e u^\e|  \eta^{\beta+2}  \d\mathcal{L} \\ 
\le  \tau ||\nabla_\e \eta||_{L^\infty(B) }^{-2}  \e^{-2} (\beta+1) \int_B (\delta+|\nabla_\e u^\e|^2)^{\frac{Q-2}{2}} |X_3^\e u^\e|^\beta |\nabla_\e X_3^\e u^\e|^2  \eta^{\beta+4}  \d\mathcal{L} \\
+ \frac{  (\beta+1)||\nabla_\e \eta||_{L^\infty(B)}^2 }{\tau} \int_B (\delta+|\nabla_\e u^\e|^2)^{\frac{Q}{2}} |X_3^\e u^\e|^\beta  \eta^{\beta}  \d\mathcal{L} =\mathcal A  + \mathcal B.
 \end{multline*}
Next, we invoke Lemma \ref{zhong-l1} to estimate the first integral $\mathcal A$ as 
\begin{multline*}
\int_B (\delta+|\nabla_\e u^\e|^2)^{\frac{Q-2}{2}} |X_3^\e u^\e|^\beta |\nabla_\e X_3^\e u^\e|^2  \eta^{\beta+4}  \d\mathcal{L} \\
 \le \bigg(\frac{4\Lambda}{\lambda (\beta+6)} +2\Lambda
\bigg) \int_B (\delta+|\nabla_\e u^\e|^2)^{\frac{Q-2}{2}}\eta^{\beta+2} |\nabla_\e \eta|^2 |X_3^\e u^\e|^{\beta +2} \d\mathcal{L}
\\+ 2 \e^2\Lambda \bigg(1+\frac{2}{\lambda (\beta+3)^2}\bigg) \int_B  (\delta+|\nabla_\e u^\e|^2)^{\frac{Q}{2}}  |X_3^\e u^\e|^\beta \eta^\beta \d\mathcal{L}
\\ \le  ||\nabla_\e \eta||_{L^\infty(B)}^2 \bigg(\frac{4\Lambda}{\lambda (\beta+6)} +2\Lambda
\bigg) \int_B (\delta+|\nabla_\e u^\e|^2)^{\frac{Q-2}{2}}\eta^{\beta+2}  |X_3^\e u^\e|^{\beta}  |X_3^\e u^\e|^{2} \d\mathcal{L}
\\+ 2 \e^2\Lambda \bigg(1+\frac{2}{\lambda (\beta+3)^2}\bigg) \int_B  (\delta+|\nabla_\e u^\e|^2)^{\frac{Q}{2}}  |X_3^\e u^\e|^{\beta-2}  |X_3^\e u^\e|^{2} \eta^\beta \d\mathcal{L}
\end{multline*}
(using the fact that $ |X_3^\e u^\e| \le \e \sum_{i,j=1}^3 |X_i^\e X_j^\e u^\e|$ one concludes)
\begin{multline*} \le\e^2  ||\nabla_\e \eta||_{L^\infty(B)}^2 \bigg(\frac{4\Lambda}{\lambda (\beta+6)} +2\Lambda
\bigg) \int_B (\delta+|\nabla_\e u^\e|^2)^{\frac{Q-2}{2}}\eta^{\beta+2}  |X_3^\e u^\e|^{\beta}  \sum_{i,j=1}^3 |X_i^\e X_j^\e u^\e|^{2} \d\mathcal{L}
\\+ 2 \e^4\Lambda \bigg(1+\frac{2}{\lambda (\beta+3)^2}\bigg) \int_B  (\delta+|\nabla_\e u^\e|^2)^{\frac{Q}{2}}  |X_3^\e u^\e|^{\beta-2}   \sum_{i,j=1}^3 |X_i^\e X_j^\e u^\e|^{2} \eta^\beta \d\mathcal{L}.
\end{multline*}
To estimate $\mathcal B$  we simply observe that 
\begin{multline*}
|\mathcal B|\le \frac{  \e^2 (\beta+1)||\nabla_\e \eta||_{L^\infty(B)}^2 }{\tau} \int_B (\delta+|\nabla_\e u^\e|^2)^{\frac{Q}{2}} |X_3^\e u^\e|^{\beta-2}  \eta^{\beta}  \sum_{i,j=1}^3 |X_i^\e X_j^\e u^\e|^{2}  \d\mathcal{L}.
\end{multline*}

 In conclusion we have proved 
 \begin{multline*}
| \mathcal I | \le  \tau  (\beta+1) \Bigg[
  \bigg(\frac{4\Lambda}{\lambda (\beta+6)} +2\Lambda
\bigg) \int_B (\delta+|\nabla_\e u^\e|^2)^{\frac{Q-2}{2}}\eta^{\beta+2}  |X_3^\e u^\e|^{\beta}  \sum_{i,j=1}^3 |X_i^\e X_j^\e u^\e|^{2} \d\mathcal{L}
\\+ 2  ||\nabla_\e \eta||_{L^\infty(B)}^{-2} \e^2\Lambda \bigg(1+\frac{2}{\lambda (\beta+3)^2}\bigg) \int_B  (\delta+|\nabla_\e u^\e|^2)^{\frac{Q}{2}}  |X_3^\e u^\e|^{\beta-2}   \sum_{i,j=1}^3 |X_i^\e X_j^\e u^\e|^{2} \eta^2 \d\mathcal{L}\Bigg] 
\\+\frac{  \e^2 (\beta+1)||\nabla_\e \eta||_{L^\infty(B)}^2 }{\tau} \int_B (\delta+|\nabla_\e u^\e|^2)^{\frac{Q}{2}} |X_3^\e u^\e|^{\beta-2}  \eta^{\beta}  \sum_{i,j=1}^3 |X_i^\e X_j^\e u^\e|^{2}  \d\mathcal{L}.
 \end{multline*}

 \item{\bf Estimate of  $\mathcal {II}$.} Observe that, in view of Young's inequality, one has
 \begin{multline*}
 |\mathcal {II}| \le  \tau  \int_B (\delta+|\nabla_\e u^\e|^2)^{\frac{Q-2}{2}} \eta^{\beta+2} |\nabla_\e \eta|  |X_3^\e u^\e|^\beta |X_3u^\e|^2  \d\mathcal{L}  \\+ \frac{(\beta+2)^2}{\tau} \int_B (\delta+|\nabla_\e u^\e|^2)^{\frac{Q}{2}} \eta^{\beta} |\nabla_\e \eta|^2  |X_3^\e u^\e|^\beta  d\mathcal{L}
 \\\le   \tau  \int_B (\delta+|\nabla_\e u^\e|^2)^{\frac{Q-2}{2}} \eta^{\beta+2} |\nabla_\e \eta|  |X_3^\e u^\e|^\beta  \sum_{i,j=1}^{3}  |X_i^\e X_j^\e u^\e|^2  \d\mathcal{L}  \\+ \e^2 \frac{(\beta+2)^2)||\nabla_\e \eta||_{L^\infty(B)}^2}{\tau} \int_B (\delta+|\nabla_\e u^\e|^2)^{\frac{Q}{2}} \eta^{\beta}  |X_3^\e u^\e|^{\beta-2}  \sum_{i,j=1}^{3}  |X_i^\e X_j^\e u^\e|^2  d\mathcal{L}.
 \end{multline*}
\end{itemize}
This concludes the estimate of $I_1$, as in \eqref{mustbelikethis}.

 \item {\bf Estimate of $I_2$.} To estimate $I_2$ we will note that in view of the structure conditions \eqref{zhong-structure} there exists a constant $C$ depending on $B$ (essentially $\max_B |x_i|$) such that
 \begin{multline*}
 |\int_B |\nabla_\e A_i^{\e, \delta}(x,\nabla_\e u^\e) ||\nabla_\e X_3^\e u^\e | |X_3^\e u^\e|^{\beta-1} |\nabla_\e  u^\e| \eta^{\beta+2}  \d\mathcal{L}| \\ \le \int_B (\delta+|\nabla_\e u^\e|^2)^{\frac{Q-2}{2}} \sum_{i,j=1}^{3}  |X_i^\e X_j^\e u^\e| |\nabla_\e u^\e | ||\nabla_\e X_3^\e u^\e | |X_3^\e u^\e|^{\beta-1}\eta^{\beta+2}  d\mathcal{L} \\+ C \int_B(\delta+|\nabla_\e u^\e|^2)^{\frac{Q-1}{2} } |\nabla_\e u^\e | |\nabla_\e X_3^\e u^\e | |X_3^\e u^\e|^{\beta-1}\eta^{\beta+2}  d\mathcal{L} \\
 \le \int_B (\delta+|\nabla_\e u^\e|^2)^{\frac{Q-1}{2}} \sum_{i,j=1}^{3}  |X_i^\e X_j^\e u^\e| ||\nabla_\e X_3^\e u^\e | |X_3^\e u^\e|^{\beta-1}\eta^{\beta+2}  d\mathcal{L} \\+ C \int_B(\delta+|\nabla_\e u^\e|^2)^{\frac{Q}{2} }|\nabla_\e X_3^\e u^\e | |X_3^\e u^\e|^{\beta-1}\eta^{\beta+2}  d\mathcal{L}.
 \end{multline*}
 Note that the second integral occurs only because of the dependence of $A_i$ on the space variable $x$.
 The first integral is estimated exactly as in \cite{Zhong}, by means of Young's inequality and Lemma \ref{zhong-l1}.  In fact one has
 \begin{multline*}
  \int_B (\delta+|\nabla_\e u^\e|^2)^{\frac{Q-1}{2}} \sum_{i,j=1}^{3}  |X_i^\e X_j^\e u^\e| ||\nabla_\e X_3^\e u^\e | |X_3^\e u^\e|^{\beta-1}\eta^{\beta+2}  d\mathcal{L} \\  \le 
\e^{-2}  || \nabla_\e \eta||_{L^\infty}^{-2}  \tau   \int_B (\delta+|\nabla_\e u^\e|^2)^{\frac{Q-2}{2}}   |\nabla_\e X_3^\e u^\e |^2 |X_3^\e u^\e|^{\beta}\eta^{\beta+4}  d\mathcal{L} \\
  +C \e^2\beta^2||\nabla_\e \eta||_{L^\infty}^{2} \tau^{-1}   \int_B (\delta+|\nabla_\e u^\e|^2)^{\frac{Q}{2}} \sum_{i,j=1}^{3}  |X_i^\e X_j^\e u^\e|^2  |X_3^\e u^\e|^{\beta-2}\eta^{\beta}  d\mathcal{L} 
\\
  \le
 C \e^{-2} || \nabla_\e \eta||_{L^\infty}^{-2}  \tau  (\beta+2)^4 \bigg(\frac{2\Lambda}{\lambda (\beta+1)} +2\Lambda
\bigg) \int_B (\delta+|\nabla_\e u^\e|^2)^{\frac{Q-2}{2}} \eta^{\beta+2} |\nabla_\e \eta|^2 |X_3^\e u^\e|^{\beta +2} \d\mathcal{L}
\\+ 2 C\beta^2||\nabla_\e \eta||_{L^\infty}^{-2} \tau  \Lambda \bigg(1+\frac{1}{\lambda (\beta+1)^2}\bigg) \int_B  (\delta+|\nabla_\e u^\e|^2)^{\frac{Q}{2}}  |X_3^\e u^\e|^\beta \eta^{\beta+4} \d\mathcal{L}
\\
+C\e^{2}\beta^2||\nabla_\e \eta||_{L^\infty}^{2} \tau^{-1}   \int_B (\delta+|\nabla_\e u^\e|^2)^{\frac{Q}{2}} \sum_{i,j=1}^{3}  |X_i^\e X_j^\e u^\e|^2  |X_3^\e u^\e|^{\beta-2}\eta^{\beta}  d\mathcal{L} .
 \end{multline*}
 Estimate \eqref{mustbelikethis} then follows once one assumes (without loss of generalization) that $||\nabla_\e \eta||_{L^\infty}\ge 1$  and using the fact that $ |X_3^\e u^\e| \le \e \sum_{i,j=1}^3 |X_i^\e X_j^\e u^\e|$.
 
 For the second integral we first use Young inequality and obtain
 \begin{multline*}
  \int_B(\delta+|\nabla_\e u^\e|^2)^{\frac{Q}{2} }|\nabla_\e X_3^\e u^\e | |X_3^\e u^\e|^{\beta-1}\eta^{\beta+2}  d\mathcal{L} \\ \le   \tau \e^{-2}  \int_B (\delta+|\nabla_\e u^\e|^2)^{\frac{Q-2}{2}}  |\nabla_\e X_3^\e u^\e |^2 |X_3^\e u^\e|^{\beta}\eta^{\beta+4}  d\mathcal{L} \\
  + \e^2\tau^{-1}   \int_B (\delta+|\nabla_\e u^\e|^2)^{\frac{Q+2}{2}}  |X_3^\e u^\e|^{\beta-2}\eta^{\beta}  d\mathcal{L} .
 \end{multline*}
 Invoking Lemma \ref{zhong-l1} and Young inequality one then has
 \begin{multline}\label{ericclapton}
  \int_B(\delta+|\nabla_\e u^\e|^2)^{\frac{Q}{2} }|\nabla_\e X_3^\e u^\e | |X_3^\e u^\e|^{\beta-1}\eta^{\beta+2}  d\mathcal{L} \\ \le 
  \tau \e^{-2} \bigg[ \bigg(\frac{2\Lambda}{\lambda (\beta+1)} +2\Lambda
\bigg) \int_B (\delta+|\nabla_\e u^\e|^2)^{\frac{Q-2}{2}} |\nabla_\e \eta|^2 |v_3|^{\beta +2} \d\mathcal{L}
\\+ 2 \e^2\Lambda \bigg(1+\frac{1}{\lambda (\beta+1)^2}\bigg) \int_B  (\delta+|\nabla_\e u^\e|^2)^{\frac{Q}{2}}  |v_3|^\beta \eta^2 \d\mathcal{L}
  \bigg]
  \\+ \tau^{-1}\e^2 \int_B (\delta+|\nabla_\e u^\e|^2)^{\frac{Q+2}{2}}  |X_3^\e u^\e|^{\beta-2}\eta^{\beta}  d\mathcal{L} 
  \end{multline}
   \begin{multline*}
 \le 
  \tau \e^{-2} \bigg[ \bigg(\frac{2\Lambda}{\lambda (\beta+1)} +2\Lambda
\bigg) \int_B (\delta+|\nabla_\e u^\e|^2)^{\frac{Q-2}{2}} |\nabla_\e \eta|^2 |v_3|^{\beta +2} \d\mathcal{L}
\\+ 2 \e^2\Lambda \bigg(1+\frac{1}{\lambda (\beta+1)^2}\bigg) \int_B  (\delta+|\nabla_\e u^\e|^2)^{\frac{Q}{2}}  |v_3|^\beta \eta^2 \d\mathcal{L}
  \bigg]
  \\+ \tau^{-1}  \frac{\beta-2}{\beta} \int_B (\delta+|\nabla_\e u^\e|^2)^{\frac{Q}{2}}  |X_3^\e u^\e|^{\beta}\eta^{\beta}  d\mathcal{L} 
  \\+ \tau^{-1} \frac{2}{\beta} \e^\beta \int_B (\delta+|\nabla_\e u^\e|^2)^{\frac{Q+\beta}{2}} \eta^{\beta}  d\mathcal{L} .
 \end{multline*}
 From the latter, estimate \eqref{mustbelikethis} follows once one recalls that  $ |X_3^\e u^\e| \le \e \sum_{i,j=1}^3 |X_i^\e X_j^\e u^\e|$.

%
%
  \item {\bf Estimate of $I_3$.} Using the structure conditions \eqref{zhong-structure} one has
\begin{multline*}  
 (\beta+2) \int_B |\nabla_\e A_i^{\e, \delta}(x,\nabla_\e u^\e)| |\nabla_\e \eta |  X_3^\e u^\e|^{\beta} |\nabla_\e  u^\e| \eta^{\beta+1}  \d\mathcal{L}
\\ \le
(\beta+2) \int_B (\delta+|\nabla_\e u^\e|^2)^{\frac{Q-2}{2} }  \sum_{i,j=1}^{3}  |X_i^\e X_j^\e u^\e| 
 |  X_3^\e u^\e|^{\beta} |\nabla_\e  u^\e| \eta^{\beta+1}  |\nabla_\e \eta | \d\mathcal{L} \\+
C (\beta+2) \int_B (\delta+|\nabla_\e u^\e|^2)^{\frac{Q-1}{2}}  |  X_3^\e u^\e|^{\beta} |\nabla_\e  u^\e| \eta^{\beta+1}   |\nabla_\e \eta |\d\mathcal{L} 
\end{multline*}
The second integrand in the right hand side is estimated as in \eqref{ericclapton}. To estimate the first integral we use Young inequality to obtain
\begin{multline*}
   \int_B (\delta+|\nabla_\e u^\e|^2)^{\frac{Q-2}{2} }  \sum_{i,j=1}^{3}  |X_i^\e X_j^\e u^\e| 
 |  X_3^\e u^\e|^{\beta} |\nabla_\e  u^\e| \eta^{\beta+1}  |\nabla_\e \eta | \d\mathcal{L} \\
  \le \tau  \int_B (\delta+|\nabla_\e u^\e|^2)^{\frac{Q-2}{2} }  \sum_{i,j=1}^{3}  |X_i^\e X_j^\e u^\e| ^2
 |  X_3^\e u^\e|^{\beta}  \eta^{\beta+2}   \d\mathcal{L} \\
 +C\tau^{-1} \int_B (\delta+|\nabla_\e u^\e|^2)^{\frac{Q}{2} }
 |  X_3^\e u^\e|^{\beta} \eta^{\beta}  |\nabla_\e \eta |^2 \d\mathcal{L} 
\end{multline*}
and consequently invoke $ |X_3^\e u^\e| \le \e \sum_{i,j=1}^3 |X_i^\e X_j^\e u^\e|$ to conclude that \eqref{mustbelikethis} holds.

   \item {\bf Estimate of $I_4$.} The structure conditions \eqref{zhong-structure} yield
 \begin{multline*}  
 \int_B\sum_{j=1}^2  |X_3 A_j^\e(x,\nabla_\e u^\e) | \eta^{\beta+2} |X_3^\e u^\e|^\beta |\nabla_\e u^\e | \d\mathcal{L} \\ \le
(\beta+2) \int_B (\delta+|\nabla_\e u^\e|^2)^{\frac{Q-2}{2} }  |\nabla_\e  X_3^\e u^\e| 
 |  X_3^\e u^\e|^{\beta} |\nabla_\e  u^\e| \eta^{\beta+2}  \d\mathcal{L} \\+
C  \int_B (\delta+|\nabla_\e u^\e|^2)^{\frac{Q-1}{2}}   |  X_3^\e u^\e|^{\beta} |\nabla_\e  u^\e| \eta^{\beta+2}  \d\mathcal{L},
\end{multline*}
which are estimated as for \eqref{I} and using $ |X_3^\e u^\e| \le \e \sum_{i,j=1}^3 |X_i^\e X_j^\e u^\e|$.
\qedhere
\end{itemize}
\end{proof}

The argument in the previous proof can be adapted to the case $\beta=0$ to obtain
\begin{corollary}\label{zhong-c3} There exists a constant $C>0$ depending only on $\lambda, \Lambda, Q$ such that 
for every $\eta\in C^{\infty}_0(B)$ with $0\le \eta\le 1$ one has 
\begin{multline*}
\int_B (\delta+|\nabla_\e u^\e|^2)^{\frac{Q-2}{2}}  \sum_{i,j=1}^{3}  |X_i^\e X_j^\e u^\e|^2  \eta^{2}  \d\mathcal{L}  
\\  \le C \  (1+||\nabla_\e \eta||_{L^\infty(B)}^2 +|| X_3 \eta||_{L^\infty(B)})\int_{\supp(\eta)} (\delta+|\nabla_\e u^\e|^2)^{\frac{Q}{2}}   d\mathcal{L} \bigg).
\end{multline*}
\end{corollary}

\begin{proof}
As with the previous proof we substitute $\phi=v_1 \eta^2$ into \eqref{zhong-3.6} to obtain
\begin{multline*}
\int_B A_{i\xi_j}^{\e,\delta}(x,\nabla_\e u^\e) X_1^\e X_j^\e u^\e X_1^\e X_i^\e u^\e \eta^2 \d\mathcal{L}  - \int_B X_1^\e A_2^{\e,\delta}(x,\nabla_\e u^\e) X_3 u^\e \eta^2  \d\mathcal{L} \\ + 2\int_B A_{i\xi_j}^{\e,\delta}(x,\nabla_\e u^\e) X_1^\e X_j^\e u^\e   v_1 \eta X_i^\e  \eta\d\mathcal{L} \\ = -\int_B  A_2^{\e,\delta}(x,\nabla_\e u^\e)X_3 (v_1 \eta^2)  \d\mathcal{L}.
\end{multline*}
Repeating this argument for $v_k$ with  $k=1,2,3$ and using the structure conditions \eqref{zhong-structure} yields
\begin{multline*}
\lambda \int_B (\delta+|\nabla_\e u^\e|^2)^{\frac{Q-2}{2}} \sum_{i,j=1}^{3}  |X_i^\e X_j^\e u^\e|^2  \eta^{2}  \d\mathcal{L}  
  \le -2 \int_B A_{i\xi_j}^{\e,\delta}(x,\nabla_\e u^\e) X_k^\e X_j^\e u^\e  v_k \eta X_i^\e \eta \d\mathcal{L}
  \\
  -\int_B  (A_2^{\e,\delta}(x,\nabla_\e u^\e) X_1^\e X_3 u^\e  - A_1^{\e,\delta} (x,\nabla_\e u^\e) X_2^\e X_3 u^\e)\eta^2  \d\mathcal{L} 
   \\
  -\int_B  (A_2^{\e,\delta}(x,\nabla_\e u^\e) X_1^\e \eta  - A_1^{\e,\delta} (x,\nabla_\e u^\e) X_2^\e \eta )\eta X_3 u^{\e,\delta}   \d\mathcal{L} \\-\int_B  A_2^{\e,\delta}(x,\nabla_\e u^\e)X_3 (v_1 \eta^2) - A_1^{\e,\delta} (x,\nabla_\e u^\e)X_3 (v_2 \eta^2)   \d\mathcal{L} = I_1+I_2+I_3+I_4.
\end{multline*}
Next we show that, modulo a constant $C$ as in the statement,  every term $I_l$ on the right hand side can be estimated by the expression
$$ \tau  \int_B (\delta+|\nabla_\e u^\e|^2)^{\frac{Q-2}{2}} \sum_{i,j=1}^{3}  |X_i^\e X_j^\e u^\e|^2  \eta^{2}  \d\mathcal{L}   +  \tau^{-1}  (1+||\nabla_\e \eta||_{L^\infty(B)}^2 +|| X_3 \eta||_{L^\infty(B)}) \int_{\supp(\eta)} (\delta+|\nabla_\e u^\e|^2)^{\frac{Q}{2}} \d\mathcal{L} 
$$
for some $1>>\tau>0$ arbitrarily small, from which the conclusion is immediate.   The estimate for $I_1$ follows immediately from Young inequality.  To estimate $I_2$ we invoke the structure conditions \eqref{zhong-structure}, then apply Young's inequality, Lemma \ref{zhong-l1}, and the observation $|X_3 u^\e|\le \sum_{i,j=1}^3 | X_i^\e X_j^\e u^\e|$ to  deduce
\begin{multline*}
I_2 \le \tau  \int_B (\delta+|\nabla_\e u^\e|^2)^{\frac{Q-1}{2}} |X_3 \nabla_\e u^\e| \eta^2    \d\mathcal{L}
\\ \le \tau  \int_B (\delta+|\nabla_\e u^\e|^2)^{\frac{Q-2}{2}} |X_3 \nabla_\e u^\e|^2 \eta^2 \d\mathcal{L} +C\tau^{-1} \int_{\supp(\eta)} (\delta+|\nabla_\e u^\e|^2)^{\frac{Q}{2}} \d\mathcal{L} 
\\
\le \tau \sup_B|\nabla_\e \eta|^2  \int_B (\delta+|\nabla_\e u^\e|^2)^{\frac{Q-2}{2}} |X_3 u^\e|^2 \eta^2 \d\mathcal{L} +C(\tau+\tau^{-1})\int_{\supp(\eta)} (\delta+|\nabla_\e u^\e|^2)^{\frac{Q}{2}} \d\mathcal{L} 
\\
\le \tau  \int_B (\delta+|\nabla_\e u^\e|^2)^{\frac{Q-2}{2}}\sum_{i,j=1}^3 | X_i^\e X_j^\e u^\e|^2 \eta^2 \d\mathcal{L} +C(\tau+\tau^{-1}) \int_{\supp(\eta)} (\delta+|\nabla_\e u^\e|^2)^{\frac{Q}{2}} \d\mathcal{L} 
\end{multline*}
The estimates on $I_3$ and $I_4$ proceed in a similar fashion, using Young's inequality, Lemma~\ref{zhong-l1}, and the observation $|X_3 u^\e|\le \sum_{i,j=1}^3 | X_i^\e X_j^\e u^\e|$.
\end{proof}

Note that the previous result immediately implies part (2) of Proposition \ref{zhong-prop}.

\bigskip

The following corollary is a straightforward consequence of 
Lemma \ref{zhong-l3} and the 
Young inequality applyed to the right hand side of inequality of the lemma.

\begin{corollary}\label{zhong-l3.1}
For every $\beta \ge 2$ and $\eta\in C^{\infty}_0(B)$ with $0\le \eta \le 1$, one has 
\begin{multline*}\int_B (\delta+|\nabla_\e u^\e|^2)^{\frac{Q-2}{2}} |X_3^\e u^\e|^\beta \sum_{i,j=1}^{3}  |X_i^\e X_j^\e u^\e|^2  \eta^{\beta+2}  \d\mathcal{L}  
\\  \le  \e^\beta C^\beta(\beta+1)^{4}  ||\nabla_\e \eta||_{L^\infty(B)}^\beta
\bigg(
 \int_B
 (\delta+|\nabla_\e u^\e|^2)^{\frac{Q-2+\beta}{2}} \sum_{i,j=1}^{3}  |X_i^\e X_j^\e u^\e|^2 \eta^\beta  \d\mathcal{L}
 \\+ \int_B (\delta+|\nabla_\e u^\e|^2)^{\frac{Q+\beta}{2}} \eta^{\beta}  d\mathcal{L} \bigg).
\end{multline*}
\end{corollary}

\begin{theorem}[Caccioppoli Inequality, \cite{Zhong}]  \label{bigbadcaccioppoli} For every $\beta \ge 2$ and $\eta\in C^{\infty}_0(B)$ with $0\le \eta \le 1$, one has
\begin{multline*}
\int_B (\delta+|\nabla_\e u^\e|^2)^{\frac{Q-2+\beta}{2}} \sum_{i,j=1}^{3} |X_i^\e X_j^\e u^\e|^2  \eta^2  \d\mathcal{L} \\
\le C (\beta+1)^{8} (||\nabla_\e \eta||_{L^\infty(B)}+ ||\eta X_3 \eta||_{L^\infty(B)})\int_{\supp(\eta)}  (\delta+|\nabla_\e u^\e|^2)^{\frac{Q+\beta}{2}}  \d\mathcal{L}
\\+ C \int_B  \eta^2 (\delta+|\nabla_\e u^\e|^2)^{\frac{Q+\beta+1}{2}}  \d\mathcal{L}.
\end{multline*}
\end{theorem}
\begin{proof} 
We use H\"older inequality and Corollary \ref{zhong-l3.1} to obtain
\begin{multline*}
 \int_B (\delta+|\nabla_\e u^\e|^2)^{\frac{Q-2+\beta}{2}} |X_3 u^\e|^2 \eta^2  \d\mathcal{L} \\
 \le  \bigg(
 \int_B (\delta+|\nabla_\e u^\e|^2)^{\frac{Q-2}{2}} |X_3 u|^{\beta+2} \eta^{\beta+2} \d\mathcal{L}\bigg)^{\frac{2}{\beta+2}}\bigg(
\int_{\supp(\eta)}  (\delta+|\nabla_\e u^\e|^2)^{\frac{Q+\beta}{2}}  \d\mathcal{L}\bigg)^{\frac{\beta}{\beta+2}}
\\
\le  \bigg(
\e^{-\beta}  \int_B  (\delta+|\nabla_\e u^\e|^2)^{\frac{Q-2}{2}} |X_3^\e u|^{\beta}  \sum_{i,j=1}^{3} |X_i^\e X_j^\e u^\e|^2 \eta^{\beta+2} \d\mathcal{L}\bigg)^{\frac{2}{\beta+2}}\bigg(
\int_{\supp(\eta)}  (\delta+|\nabla_\e u^\e|^2)^{\frac{Q+\beta}{2}}  \d\mathcal{L}\bigg)^{\frac{\beta}{\beta+2}}
\\
\le
  \Bigg[ C^\beta(\beta+1)^{4}  ||\nabla_\e \eta||_{L^\infty(B)}^\beta
\bigg(
 \int_B
 (\delta+|\nabla_\e u^\e|^2)^{\frac{Q-2+\beta}{2}} \sum_{i,j=1}^{3}  |X_i^\e X_j^\e u^\e|^2 \eta^\beta  \d\mathcal{L}
 \\+ \int_B (\delta+|\nabla_\e u^\e|^2)^{\frac{Q+\beta}{2}} \eta^{\beta}  d\mathcal{L} \bigg)
 \Bigg]^{\frac{2}{\beta+2}} \bigg(
\int_{\supp(\eta)}  (\delta+|\nabla_\e u^\e|^2)^{\frac{Q+\beta}{2}}  \d\mathcal{L}\bigg)^{\frac{\beta}{\beta+2}}.
\end{multline*}
Recalling Lemma \ref{zhong-l2} the previous estimate then yields
\begin{multline*}
\int_B (\delta+|\nabla_\e u^\e|^2)^{\frac{Q-2+\beta}{2}} \sum_{i,j=1}^{3} |X_i^\e X_j^\e u^\e|^2  \eta^2  \d\mathcal{L} \le C (\beta+1)^4 \int_B (\delta+|\nabla_\e u^\e|^2)^{\frac{Q-2+\beta}{2}} |X_3 u^\e|^2 \eta^2  \d\mathcal{L} \\+
C \int_B (\eta^2+|\nabla_\e \eta|^2) (\delta+|\nabla_\e u^\e|^2)^{\frac{Q+\beta}{2}}  \d\mathcal{L}
+ C \int_B  \eta^2 (\delta+|\nabla_\e u^\e|^2)^{\frac{Q+\beta+1}{2}}  \d\mathcal{L}
\\
\le 
C (\beta+1)^4 \Bigg[ C^\beta(\beta+1)^{4}  ||\nabla_\e \eta||_{L^\infty(B)}^\beta
\bigg(
 \int_B
 (\delta+|\nabla_\e u^\e|^2)^{\frac{Q-2+\beta}{2}} \sum_{i,j=1}^{3}  |X_i^\e X_j^\e u^\e|^2 \eta^\beta  \d\mathcal{L}
 \\+ \int_B (\delta+|\nabla_\e u^\e|^2)^{\frac{Q+\beta}{2}} \eta^{\beta}  d\mathcal{L} \bigg)
 \Bigg]^{\frac{2}{\beta+2}} \bigg(
\int_{\supp(\eta)}  (\delta+|\nabla_\e u^\e|^2)^{\frac{Q+\beta}{2}}  \d\mathcal{L}\bigg)^{\frac{\beta}{\beta+2}}
+ C \int_B  \eta^2 (\delta+|\nabla_\e u^\e|^2)^{\frac{Q+\beta+1}{2}}  \d\mathcal{L}.
\end{multline*}
The conclusion follows immediately from the latter and from Young inequality.
\end{proof}

\bigskip

\bigskip

\begin{lemma}
 Let  $u^\e\in W^{1,Q}_{\e, loc}(B)\cap C^{\infty}(B)$ be the unique  solution of \eqref{zhong-diri}. For every $\beta \ge 2$  set $w=(\delta+|\nabla_\e u^\e|^2)^{\frac{Q+\beta}{4}}$.
 If  $\eta\in C^{\infty}_0(B)$ with $0\le \eta \le 1$,  and $\kappa=Q/(Q-2)$, then one has
\begin{multline*}
\bigg(\int_B w^{2\kappa}  \eta^2  \d\mathcal{L} \bigg)^{\frac{1}{\kappa}}
\le C (\beta+1)^{8} (||\nabla_\e \eta||_{L^\infty(B)}+ ||\eta X_3 \eta||_{L^\infty(B)})\int_{\supp(\eta)}  w^2  \d\mathcal{L},
\end{multline*}
where $C>0$ is a constant depending only on $Q$.
\end{lemma}
\begin{proof}
It is well known that the Sobolev constant depends only on the constants in the Poincare' inequality and in the doubling inequality, both of which are stable in this Riemannian approximation scheme (see \cite{CCR}). Applying Sobolev inequality   yields
\begin{multline*}
\bigg(\int_B  |\eta w|^{2\kappa}  \d\mathcal{L}\bigg)^{\frac{1}{\kappa}} \le
C \int_B|\nabla_\e (\eta w)|^2 \eta^2  \d\mathcal{L} 
\le C ||\nabla_\e \eta||_{L^\infty(B) } \int_B w^2   \d\mathcal{L} + C\int_B |\nabla_\e w|^2 \eta^2  \d\mathcal{L},
\end{multline*}
for $\kappa=Q/(Q-2)$ and some constant $C$ depending only on $Q$.  Invoking Theorem \ref{bigbadcaccioppoli} we arrive at 
\begin{multline}\label{blahblah}
\bigg(\int_B  |\eta w|^{2\kappa}  \d\mathcal{L}\bigg)^{\frac{1}{\kappa}} \le
 C ||\nabla_\e \eta||_{L^\infty(B) } \int_B w^2   \d\mathcal{L} 
 \\ + C(\beta+1)^{8} (||\nabla_\e \eta||_{L^\infty(B)}+ ||\eta X_3 \eta||_{L^\infty(B)})\int_{\supp(\eta)}  (\delta+|\nabla_\e u^\e|^2)^{\frac{Q+\beta}{2}}  \d\mathcal{L}
\\+ C \int_B  \eta^2 (\delta+|\nabla_\e u^\e|^2)^{\frac{Q+\beta+1}{2}}  \d\mathcal{L}
\\\le C'  (\beta+1)^{8} (||\nabla_\e \eta||_{L^\infty(B)}+ ||\eta X_3 \eta||_{L^\infty(B)})\int_{\supp(\eta)}  w^2  \d\mathcal{L} + C' \int_B w^{2+\frac{2}{Q+\beta}}  \eta^2 \d\mathcal{L}.
\end{multline}
Using H\"older inequality one has  
\begin{multline*}
 \int_B w^{2+\frac{2}{Q+\beta}}  \eta^2 \d\mathcal{L} \le \bigg(\int_B w^{\frac{2}{Q+\beta} Q} \d\mathcal{L} \bigg)^{\frac{1}{Q}} \bigg( \int_B |w \eta|^{\frac{2Q}{Q-1}} \d\mathcal{L}
 \bigg)^{\frac{Q-1}{Q}}
  \\  \le  \bigg(
 \int_B  \eta^2 (\delta+|\nabla_\e u^\e|^2)^{\frac{Q}{2}}  \d\mathcal{L}
  \bigg)^{\frac{1}{Q}}
 \bigg(\int_B  |\eta w|^{2\kappa}  \d\mathcal{L}\bigg)^{\frac{1}{\kappa}} |B|^{\frac{1}{Q-1}}.
\end{multline*}
Since we can choose $|B|<1$, we can bring the first term on the left  hand side of \eqref{blahblah} to conclude the proof of the lemma.
\end{proof}

The proof of Theorem \ref{lipschitzestimate} now follows in a standard fashion, as described in \cite{Zhong},  from the Moser iteration scheme (see for instance \cite[Theorem 8.18]{MR1814364}) and from \cite[Lemma 3.38]{HKM}. Note that the constant involved in such iteration are stable as $\e\to 0$ (see \cite{CCR}).

\bigskip

\subsubsection{Uniform $C^{1,\al}_\e$ regularity} Throughout this section we will implicitly use the uniform (in $\e$) local Lipschitz regularity of solutions of \eqref{zhong-diri} and set for every $B(x_0,2 r_0)\subset B$, $k\in \R$, $l=1,2,3$, and $0<r<r_0/4<1$,
\begin{multline*}
\mu^\e(r)={\rm osc}_{B(x_0,r)} |\nabla_\e u^\e|; A_{l,k,r}^-=\{x\in B(x_0,r) \text{ such that } X_l^\e u^\e<k \} \\  \ \text{ and } A_{l,k,r}^+=\{x\in B(x_0,r) \text{ such that } X_l^\e u^\e>k \}.\end{multline*}

The proof of Proposition \ref{zhong-prop} and in particular of the $C^{1,\alpha}$ estimate
in part (1)
follows immediately from the following theorem, which is the main result of the section:
\begin{theorem}\label{finale}
 Let  $u^\e\in W^{1,Q}_{\e, {\rm loc}}(B)\cap C^{\infty}(B)$ be the unique  solution of \eqref{zhong-diri}.
 There exists a constant $s>0$ depending only on $Q,\lambda, \Lambda, r_0$ such that 
 $$\mu(r)\le (1-2^{-s}) \mu(4r) + 2^s(\delta+\mu(r_0)^2)^{\frac{Q}{2}} \bigg(\frac{r}{r_0}\bigg)^{\frac{1}{Q}},$$
 for all $0<r<r_0/8$.
\end{theorem}

Our first step in the proof of this theorem  consists in establishing a Caccioppoli inequality, in Proposition \ref{caccio-sub} for second order derivatives on super level sets $A_{l,k,r}^+$. This result will imply that the gradient $\nabla_\e u^\e$ is in a De Giorgi-type class and then Theorem \ref{finale} will follow from well known results in the literature.

We begin with some preliminary lemmata. We indicate by $|A|$ the Lebesque measure $\mathcal{L}(A)$ of a set $A$.

\begin{lemma}\label{caccio-sub-0}  Let  $u^\e\in W^{1,Q}_{\e, {\rm loc}}(B)\cap C^{\infty}(B)$ be the unique  solution of \eqref{zhong-diri}. For any $q\ge 4$ there exists a positive constant $C$ depending only on $q,\lambda, \Lambda$ such that for all $k\in \R$, $l=1,2,3$ and $0<r'<r<r_0/2$, $\eta\in C^{\infty}_0(B(x_0,r))$ such that $\eta=1$ on $B(x_0,r')$  one has
\begin{multline}\label{z4.3-00}
\int_{A^+_{l,k,r'}} (\delta+|\nabla_\e u^\e|^2)^{\frac{Q-2}{2}} |\nabla_\e \omega_l|^2 \eta^2 \d\mathcal{L}
\\ \le \int_{A^+_{l,k,r}} (\delta+|\nabla_\e u^\e|^2)^{\frac{Q-2}{2}} |\omega_l|^2 |\nabla_\e \eta|^2\d\mathcal{L}
+C(\delta+\mu(r_0)^2)^{\frac{Q}{2}} |A_{l,k,r}^+|^{1-\frac{2}{q}} + I_3
\end{multline}
where we have set $\omega_l=(X_l^\e u^\e-k)^+$ and 
\begin{equation}\label{int3}
I_3= \int_{B(x_0,r)} (\delta+|\nabla_\e u^\e|^2)^{\frac{Q-2}{2}} |\nabla_\e X_3 u^\e| |\omega_1|  \eta^2  \d\mathcal{L}.
\end{equation}
\end{lemma}
\begin{proof} As above, we study the case $l=1$, since $l=2,3$ is similar.
Select a  cut-off function $\eta\in C^{\infty}_0(B(x_0,r))$ such that $\eta=1$ on $B(x_0,r')$ and $|\nabla_\e \eta|\le M(r-r')^{-1}$, for some $M>0$ independent of $\e$. Substitute $\phi=\eta^2 \omega_1$ in the weak form of \eqref{zhong-d-eq1} to obtain
\begin{multline*}
\int_B A^{\e, \delta}_{i,\xi_j}(x, \nabla_\e u^\e) X_j^\e X_1^\e u^\e X_i^\e \omega_1 \eta^2  \d\mathcal{L} = 
-2 \int_B A^{\e, \delta}_{i,\xi_j}(x, \nabla_\e u^\e) X_j^\e X_1^\e u^\e X_i^\e \eta \eta \omega_1  \d\mathcal{L}
\\
- \int_B A^{\e,\delta}_{i,\xi_2}(x, \nabla_\e u^\e) X_3 u^\e X_i^\e (\omega_1 \eta^2 ) \d\mathcal{L}
\\
+ \int_B X_3 A_i^{\e, \delta} (x, \nabla_\e u^\e)  \eta^2 \omega_1 d\mathcal{L}
\\
- \int_B  \bigg(A^{\e,\delta}_{i,x_1}(x,\nabla_\e u^\e) -\frac{x_2}{2}A^{\e,\delta}_{i,x_3}(x,\nabla_\e u^\e)  \bigg) X_i^\e (\eta^2 \omega_1)  d\mathcal{L}.
\end{multline*}
Using Young inequality and the structure conditions \eqref{zhong-structure} one easily obtains the estimate
\begin{multline*}
\int_B (\delta+|\nabla_\e u^\e|^2)^{\frac{Q-2}{2}} |\nabla_\e \omega_1|^2 \eta^2 \d\mathcal{L}
\le C \int_B  (\delta+|\nabla_\e u^\e|^2)^{\frac{Q-2}{2}} |\nabla_\e \eta|^2 \omega_1^2  \d\mathcal{L}
\\
+C \int_B (\delta+|\nabla_\e u^\e|^2)^{\frac{Q-2}{2}} |X_3 u^\e|^2 \eta^2   \d\mathcal{L}
+ C \int_B (\delta+|\nabla_\e u^\e|^2)^{\frac{Q-2}{2}} |\nabla_\e X_3 u^\e| |\omega_1| \eta^2   \d\mathcal{L}
\\ + C\int_B (\delta+|\nabla_\e u^\e|^2)^{\frac{Q-1}{2}} \bigg( \omega_1 \eta^2+2 \omega_1 \eta |\nabla_\e \eta|+ \eta^2 |\nabla_\e \omega_1| \bigg) \d\mathcal{L}
\le I_1+I_2+I_3+I_4.
\end{multline*}

The terms $I_1$ and $I_3$ are already in the form needed for \eqref{z4.3}. To estimate $I_4$  we observe that for every $\tau>0$ one can estimate

\begin{multline*}
I_4 \le \tau \int_B (\delta+|\nabla_\e u^\e|^2)^{\frac{Q-2}{2}} |\nabla_\e \omega_1|^2 \eta^2 \d\mathcal{L}+\\
C\tau^{-1} \int_{A^+_{1,k,r}} (\delta+|\nabla_\e u^\e|^2)^{\frac{Q}{2}}  \d\mathcal{L}
+ C \int_B  (\delta+|\nabla_\e u^\e|^2)^{\frac{Q-2}{2}} |\nabla_\e \eta|^2 \omega_1^2  \d\mathcal{L},
\end{multline*}
thus leading to the correct left hand side for  \eqref{z4.3}. To estimate $I_2$ we argue as in \cite{Zhong} and invoke Theorem \ref{bigbadcaccioppoli} and Corollary \ref{zhong-l3.1} to show
\begin{multline*}
I_2 \le \bigg(  \int_{A^+_{1,k,r}} (\delta+|\nabla_\e u^\e|^2)^{\frac{Q-2}{2}}  \d\mathcal{L} \bigg)^{1-\frac{2}{q}} \bigg(  \int_B (\delta+|\nabla_\e u^\e|^2)^{\frac{Q-2}{2}} |X_3 u^\e|^q \eta^2   \d\mathcal{L}\bigg)^{\frac{2}{q}} 
\\ \le (\delta+\mu(r_0)^2)^{\frac{Q-2}{2}\frac{q-2}{q} }|A_{1,k,r}^+|^{1-\frac{2}{q}} 
 \bigg(  \int_{B(x_0,r_0/2)} (\delta+|\nabla_\e u^\e|^2)^{\frac{Q-2}{2}} |X_3 u^\e|^{q-2} |X_i^\e X_j^\e u^\e|^2  \d\mathcal{L}\bigg)^{\frac{2}{q}} 
 \\ \le (\delta+\mu(r_0)^2)^{\frac{Q-2}{2}\frac{q-2}{q}}|A_{1,k,r}^+|^{1-\frac{2}{q}} 
 \Bigg[   C^{q-2}(q-1)^{4}  r_0^{2-q}
\bigg(
 \int_{B(x_0,\frac{2}{3}r_0)} 
 (\delta+|\nabla_\e u^\e|^2)^{\frac{Q-4+q}{2}} \sum_{i,j=1}^{3}  |X_i^\e X_j^\e u^\e|^2  d\mathcal{L}
 \\+ \int_{B(x_0,\frac{2}{3} r_0)}  (\delta+|\nabla_\e u^\e|^2)^{\frac{Q+q-2}{2}}   d\mathcal{L} \bigg)
\Bigg]^{\frac{2}{q}} 
\\
\le C^q (q-1)^{12}  r_0^{-q} (\delta+\mu(r_0)^2)^{\frac{Q-2}{2}\frac{q-2}{q}}|A_{1,k,r}^+|^{1-\frac{2}{q}} 
 \Bigg[ 
\int_{B(x_0,r_0)}   (\delta+|\nabla_\e u^\e|^2)^{\frac{Q+q-2}{2}}  \d\mathcal{L}
\\+  \int_{B(x_0,r_0)}   \eta^2 (\delta+|\nabla_\e u^\e|^2)^{\frac{Q+q-3}{2}}  \d\mathcal{L}
+ \int_{B(x_0,r_0)}   (\delta+|\nabla_\e u^\e|^2)^{\frac{Q+q-2}{2}}   d\mathcal{L} 
\Bigg]^{\frac{2}{q}} 
\\
\le   C^q (q-1)^{12} r_0^{-q} (\delta+\mu(r_0)^2)^{\frac{Q}{2}} |A_{1,k,r}^+|^{1-\frac{2}{q}}.
\end{multline*}
%
\end{proof}

In order to obtain from the previous lemma a Cacciopoli inequality 
we only need to obtain an estimate of $I_3$. 
The proof of the previous lemma yields the following 
\begin{corollary}\label{I_3toG_0}
 In the hypothesis and notation of the previous lemma, one has  that for any $q\ge 4$ there exists a positive constant $C$ depending only on $q,\lambda, \Lambda$ such that for all $k\in \R$, $l=1,2,3$ and $\eta\in C^\infty_0(B(x_0,r))$, 
\begin{equation}
I_3 \le C (\delta+\mu(r_0)^2)^{\frac{Q-2}{4}} |A_{l,k,r}^+|^{\frac{1}{2}} G_0^{\frac{1}{2}}
\end{equation}
where 
$$G_0=  \int  (\delta+|\nabla_\e u^\e|^2)^{\frac{Q-2}{2}}  \omega_l^2 |\nabla_\e v_3 |^2  \eta^2 \d\mathcal{L}.$$
\end{corollary}
\begin{proof}
From H\"older inequality one has,
 \begin{multline}
 \int_{B(x_0,r)} (\delta+|\nabla_\e u^\e|^2)^{\frac{Q-2}{2}} |\nabla_\e X_3 u^\e| |\omega_1| \eta^2 \d\mathcal{L}
\\
 \le  C (\delta+\mu(r_0)^2)^{\frac{Q-2}{4}} |A_{l,k,r}^+|^{\frac{1}{2}}\bigg(\int_{B(x_0,r)}   (\delta+|\nabla_\e u^\e|^2)^{\frac{Q-2}{2}}  \omega_1^2 |\nabla_\e X_3 u^\e |^2  \eta^2 \d\mathcal{L}  \bigg)^{\frac{1}{2}}
 \end{multline}
\end{proof}

%


\begin{lemma}\label{step-1-00}
In the hypothesis and notations of  Lemma \ref{I_3toG_0},  for every $m\in \N$, $m\ge 1$ one has that there exists a constant $C$ depending on $m,  Q, \lambda, \Lambda$, such that
 \begin{equation}\label{z4.3-step2-000}
 G_0
 \le
  C 
  \sum_{h=0}^{m}
   K^{2-\frac{1}{2^{m+h}}} (\delta + \mu^2(r_0))^{ 1 + \frac{Q}{2^{m+h+2}}} 
 \end{equation}
where
$$K=\bigg(  \int_{B(x_0,r)} (\delta+|\nabla_\e u^\e|^2)^{\frac{Q-2}{2}}  \omega_l^2  (\eta^2+  |\nabla_\e \eta|^2)  \d\mathcal{L} + \int_{B(x_0,r)} (\delta+|\nabla_\e u^\e|^2)^{\frac{Q-2}{2}} |\nabla_\e \omega_l|^2 \eta^2  \d\mathcal{L}
\bigg)^{\frac{1}{2}}.$$
 \end{lemma}
\begin{proof} 
 In the following we will denote by $C$ a series of positive constants  depending only on $m, Q,\lambda, \Lambda$.
We study the case $l=1$, since $l=2$ is similar and $l=3$ is slightly easier. 

The bound \eqref{z4.3-step2-000}  follows from a bootstrap argument, whose main step is the subject of the following estimates.

For $\beta\ge 0$  and for any  cut-off function $\eta\in C^{\infty}_0(B(x_0,r))$,
let  $$G_\beta=  \int_{B(x_0,r)}  (\delta+|\nabla_\e u^\e|^2)^{\frac{Q-2}{2}}  \omega_l^2 |\nabla_\e v_3 |^2 |v_3|^\beta  \eta^2 \d\mathcal{L},$$
$$F_\beta=  \int_{B(x_0,r)}(\delta+|\nabla_\e u^\e|^2)^{\frac{Q}{2}} |v_3|^{\beta}  |\omega_l|^2  \eta^2 \d\mathcal{L},$$
where we recall that $\omega_l=(X_l^\e u^\e-k)^+$, for $l=1,2,3$. 

We claim that there exists a constant $C>0$, depending only on $ Q,\lambda, \Lambda$  such that \begin{equation}\label{caccio-sub-3-00}
G_{\beta} \le  \Bigg\{\begin{array}{ll}   C K \bigg( G_{2\beta+2}^{\frac{1}{2}} + F_{2\beta+2}^{\frac{1}{2}} + (\delta+\mu(r_0)^2)^\frac{1}{2}  F_{2\beta}^\frac{1}{2} \bigg), & \text{ if }\beta>0\\
C K \bigg( G_{2}^{\frac{1}{2}} + F_{2}^{\frac{1}{2}} + (\delta+\mu(r_0)^2)^{ 1+\frac{Q\sigma}{4}}  K ^{1-\sigma} \bigg),& \text{ if }\beta=0 \text{ and for any }\sigma\in [0,2), \end{array}
\end{equation}
and 
\begin{equation}\label{caccio-sub-3-01}
F_{\beta} \le \Bigg\{\begin{array}{ll}  C K (\delta+\mu(r_0)^2)^\frac{1}{2} F_{2\beta}^\frac{1}{2} & \text{ if }\beta>0,\\ C (\delta+ \mu(r_0)^2) K^{2} & \text{ if }\beta=0, \end{array}
\end{equation}
In particular, for every $\beta>0$ and $m\ge 2$, it will follow that one has  
\begin{equation}\label{stima di F}
F_\beta \le (CK)^{2(1-\frac{1}{2^m})}  (\delta+ \mu(r_0)^2)^{1-\frac{1}{2^m}} F_{2^m\beta}^{\frac{1}{2^{m}}}.
\end{equation}
Estimate
\eqref{caccio-sub-3-01} follows directly from H\"older inequality and from the gradient bounds in Theorem \ref{lipschitzestimate},
\begin{multline} F_\beta =  \int_{B(x_0,r)}(\delta+|\nabla_\e u^\e|^2)^{\frac{Q}{2}} |v_3|^{\beta}  |\omega_l|^2  \eta^2  \d\mathcal{L} \\
\le \bigg(   \int_{B(x_0,r)} (\delta+|\nabla_\e u^\e|^2)^{\frac{Q-2}{2}}  |\omega_l|^2 \eta^2   \d\mathcal{L} \bigg)^\frac{1}{2} 
\bigg(  \ \int_{B(x_0,r)}(\delta+|\nabla_\e u^\e|^2)^{\frac{Q  +2}{2}} |v_3|^{2\beta}  |\omega_l|^2    \eta^2 \d\mathcal{L}\bigg)^\frac{1}{2}
\\
\le C K (\delta+\mu(r_0)^2)^\frac{1}{2}  F_{2\beta}^\frac{1}{2} \end{multline}

To prove \eqref{caccio-sub-3-00} 
substitute $\phi=\eta^2 \omega_l^2 |v_3|^\beta v_3$ in the weak form of \eqref{zhong-d-eq3} to obtain
\begin{multline}\label{caccio-sub-1}
(\beta+1)\int_B A_{i \xi_j}^{\e,\delta} (x,\nabla_\e u^\e) X_j^\e v_3 X_i^\e v_3 \omega_l^2 |v_3|^\beta \eta^2\d\mathcal{L} \le   \int_B |A^{\e,\delta}_{i \xi_j}(x,\nabla_\e u^\e) | |X_j^\e v_3|   |X_i^\e [ \eta^2 \omega_l^2]  |  |v_3|^{\beta+1}  \d\mathcal{L}  \\ +
\e \int_B  | A^{\e,\delta}_{i, x_3} (x,\nabla_\e u^\e) X_i ^\e \bigg[ \eta^2 \omega_l^2 |v_3|^\beta v_3\bigg]  |\d\mathcal{L} = A+ B.
\end{multline}
The first term on the left hand side is estimated via Young's inequality \begin{equation*}
A\le C K \bigg( \int_B(\delta+|\nabla_\e u^\e|^2)^{\frac{Q-2}{2}} |v_3|^{2\beta+2} |\nabla_\e v_3|^2 |\omega_1|^2 \eta^2   \d\mathcal{L}
 \bigg)^{\frac{1}{2}}
\end{equation*}
For the second term we note that 
\begin{multline}
B\le  C\e\int_B (\delta+|\nabla_\e u^\e|^2)^{\frac{Q-1}{2}}  \bigg| \nabla_\e \bigg[ \eta^2 \omega_l^2 |v_3|^\beta v_3\bigg]  \bigg| |\d\mathcal{L}  \\
 \le C\e\int_B (\delta+|\nabla_\e u^\e|^2)^{\frac{Q-1}{2}} |v_3|^{\beta+1} |\omega_1|^2  \eta  |\nabla_\e \eta|  \d\mathcal{L} 
\\
+C\e\int_B (\delta+|\nabla_\e u^\e|^2)^{\frac{Q-1}{2}} |v_3|^{\beta} |\nabla_\e v_3| |\omega_1|^2 \eta^2   \d\mathcal{L} 
\\
+ C\e\int_B (\delta+|\nabla_\e u^\e|^2)^{\frac{Q-1}{2}} |v_3|^{\beta+1} |\nabla_\e \omega_1| |\omega_1| \eta^2   \d\mathcal{L}  =T_1+T_2+T_3.
\end{multline}
For any $\bar \e >0$,  Young inequality and \eqref{caccio-sub-3-01} yield the  estimate 
\begin{multline}\label{bringback}
T_2 \le \bar \e \int_B  (\delta+|\nabla_\e u^\e|^2)^{\frac{Q-2}{2}} |v_3|^{\beta} |\nabla_\e v_3|^2 |\omega_1|^2 \eta^2   \d\mathcal{L}  \\+ C_{\bar \e} 
 \int_B (\delta+|\nabla_\e u^\e|^2)^{\frac{Q}{2}} |v_3|^{\beta}  |\omega_1|^2 \eta^2   \d\mathcal{L} \\  \le  \bar \e \int_B  (\delta+|\nabla_\e u^\e|^2)^{\frac{Q-2}{2}} |v_3|^{\beta} |\nabla_\e v_3|^2 |\omega_1|^2 \eta^2   \d\mathcal{L}  \\+ C_{\bar \e} K (\delta+\mu(r_0)^2)^\frac{1}{2}  F_{2\beta}^\frac{1}{2}.
\end{multline}
The other two terms are estimated through H\"older inequality as
$$T_1+T_3\le K \Bigg( \int_B (\delta+|\nabla_\e u^\e|^2)^{\frac{Q}{2}} |v_3|^{2\beta+2}  |\omega_1|^2 \eta^2   \d\mathcal{L}\Bigg)^{\frac{1}{2}} .$$
%
In view of the structure conditions \eqref{zhong-structure}, of  \eqref{caccio-sub-1}, and of the estimates above for $A$ and $B$ one has
\begin{multline*}
\int_B  (\delta+|\nabla_\e u^\e|^2)^{\frac{Q-2}{2}}  \omega_1^2 |\nabla_\e v_3 |^2 |v_3|^\beta  \eta^2\d\mathcal{L}
\le K  \bigg( \int_B(\delta+|\nabla_\e u^\e|^2)^{\frac{Q-2}{2}} |v_3|^{2\beta+2} |\nabla_\e v_3|^2 |\omega_1|^2 \eta^2   \d\mathcal{L}
 \bigg)^{\frac{1}{2}} \\ 
 +  K \Bigg( \int_B (\delta+|\nabla_\e u^\e|^2)^{\frac{Q}{2}} |v_3|^{2\beta+2}  |\omega_1|^2 \eta^2   \d\mathcal{L}\Bigg)^{\frac{1}{2}} \\   + C_{\bar \e} K (\delta+\mu(r_0)^2)^\frac{1}{2}  F_{2\beta}^\frac{1}{2} + \bar \e \int_B  (\delta+|\nabla_\e u^\e|^2)^{\frac{Q-2}{2}} |v_3|^{\beta} |\nabla_\e v_3|^2 |\omega_1|^2 \eta^2   \d\mathcal{L} .
\end{multline*}
Bringing the last term on the right hand side over to the left hand side one obtains \eqref{caccio-sub-3-00} in the case $\beta>0$.
For the case $\beta=0$, the estimate on $T_2$ above can be improved. We let $\sigma\in [0,2)$ and observe that
\begin{multline}
T_2 \le \bar \e \int_B  (\delta+|\nabla_\e u^\e|^2)^{\frac{Q-2}{2}}  |\nabla_\e v_3|^2 |\omega_1|^2 \eta^2   \d\mathcal{L}  \\+ C_{\bar \e} 
 \int_B (\delta+|\nabla_\e u^\e|^2)^{\frac{Q}{2}}  |\omega_1|^2 \eta^2   \d\mathcal{L} \\ 
 \le \bar \e G_0+ C_{\bar \e}   (\delta+\mu(r_0)^2)
 \int_B (\delta+|\nabla_\e u^\e|^2)^{\frac{Q-2}{2}}  |\omega_1|^2 \eta^2   \d\mathcal{L}
 \\
  \le  \bar \e G_0+ C_{\bar \e} K^{2(1-\frac{\sigma}{2})}  (\delta+\mu(r_0)^2)
 \bigg(  \int_B (\delta+|\nabla_\e u^\e|^2)^{\frac{Q-2}{2}}  |\omega_1|^2 \eta^2   \d\mathcal{L}\bigg)^{\frac{\sigma}{2}}\\
  \le \bar \e G_0+ C_{\bar \e} K^{2(1-\frac{\sigma}{2})}  (\delta+\mu(r_0)^2)^{1+\frac{Q\sigma}{4}}.
\end{multline}

The latter concludes the proof of  the estimates \eqref{caccio-sub-3-00} and \eqref{caccio-sub-3-01}. At this point 
we can proceed with  the description of the  bootstrap argument needed to prove the bound on $G_0$.

In view of Lemma \ref{zhong-l1}, Corollary \ref{zhong-c3}, Corollary \ref{zhong-l3.1} and Theorem \ref{bigbadcaccioppoli} one has the following

\begin{equation}\label{hula-1} G_{\beta} \le C (\delta+\mu(r_0)^2)^{\frac{Q+\beta+2}{2}}|B(x_0,r_0)| \quad   \text{ and } \quad \quad
 F_{\beta} \le C (\delta+\mu(r_0)^2)^{\frac{Q+\beta+2}{2}}|B(x_0,r_0)|. \end{equation}

Combining \eqref{hula-1} with  \eqref{caccio-sub-3-00}
and \eqref{caccio-sub-3-01}
 yields for all $\beta> 0$ and $m\ge 1$, 
 \begin{multline}
 G_\beta \le CKG_{2\beta+2}^\frac{1}{2} + (CK)^{2-\frac{1}{2^m}} (\delta+\mu^2(r_0))^{\frac{1}{2}(1-\frac{1}{2^m})} F_{2^m(2\beta+2)}^{\frac{1}{2^{m+1}}}  \\
 + 
 (CK)^{2-\frac{1}{2^m}} (\delta+\mu^2(r_0))^{\frac{1}{2}+\frac{1}{2}(1-\frac{1}{2^m})} F_{2^m(2\beta)}^{\frac{1}{2^{m+1}}} 
 \\ \le  CKG_{2\beta+2}^\frac{1}{2} + (CK)^{2-\frac{1}{2^m}} (\delta+\mu^2(r_0))^{\frac{\beta+2}{2}+ \frac{1}{2^{m+1}} \frac{Q}{2} } |B(x_0,r_0)|^{\frac{1}{2^{m+1}}} \\
 \end{multline}

Iterating the latter $m$ times and 
setting $\beta_m = 2^m-2$ one obtains
$$
G_{\beta_2} \leq C \bigg[ K^{2(1-\frac{1}{2^m})} G^{\frac{1}{2^m}}_{\beta_{m+2}}+
 \sum_{h=1}^{m}
   K^{2(1-\frac{1}{2^{m+h}})} (\delta + \mu^2(r_0))^{2( 1 + \frac{Q}{2^{m+h+2}})}\bigg].$$

 From the latter, \eqref{stima di F}, and keeping in mind the starting point \eqref{caccio-sub-3-00}  corresponding to $\beta=0$,  one concludes that for any $\sigma\in [0,2)$,
 \begin{multline}
 G_0 \le C\bigg[  K^{2(1-\frac{1}{2^{m+1} })} G^{\frac{1}{2^{m+1}}}_{\beta_{m+2}} +  \sum_{h=1}^{m}
   K^{1-\frac{1}{2^{m+h}}} (\delta + \mu^2(r_0))^{ 1 + \frac{Q}{2^{m+h+2}}}  + F_2^\frac{1}{2}  +K^2  (\delta+\mu^2(r_0))\bigg].\\
\le   C\bigg[  K^{2(1-\frac{1}{2^{m+1} })} G^{\frac{1}{2^{m+1}}}_{\beta_{m+2}} +  \sum_{h=1}^{m}
   K^{2-\frac{1}{2^{m+h}}} (\delta + \mu^2(r_0))^{ 1 + \frac{Q}{2^{m+h+2}}} 
   \\
    + K^{1-\frac{1}{2^m}} 
(\delta+\mu^2(r_0))^{\frac{1}{2}-\frac{1}{2^{m+1}}}
 F_{2^{m+1}}^\frac{1}{2^{m+1}}  +K^{2-\sigma}  (\delta+\mu^2(r_0))^{1+\frac{Q\sigma}{4}} \bigg].
 \end{multline}
 Applying \eqref{hula-1} to the latter and letting $\sigma = \frac{1}{2^m}$, yields the estimate
 \begin{multline}\label{one more estimate on G_0}
 G_0\le C\bigg[
 K^{2-\frac{1}{2^m}} (\delta+\mu^2(r_0))^{1+\frac{Q+2}{2^{m+2}}} +
 \sum_{h=1}^{m}
   K^{2-\frac{1}{2^{m+h}}} (\delta + \mu^2(r_0))^{ 1 + \frac{Q}{2^{m+h+1}}}  \\
 +
 K^{2-\frac{1}{2^m}} (\delta+\mu^2(r_0))^{ \frac{1}{2}-\frac{1}{2^{m+1}} + \frac{Q+2^{m+1}+2}{2^{m+2}}} + K^{2-\frac{1}{2^m}}  (\delta+\mu^2(r_0))^{1+\frac{Q}{2^{m+2}}}.
 \bigg],
 \end{multline}
concluding the proof.
\end{proof}


\begin{proposition}[Caccioppoli inequality on super-level sets]\label{caccio-sub}  Let  $u^\e\in W^{1,Q}_{\e, {\rm loc}}(B)\cap C^{\infty}(B)$ be the unique  solution of \eqref{zhong-diri}. For any $q\ge 4$ there exists a positive constant $C$ depending only on $q,\lambda, \Lambda$ such that for all $k\in \R$, $l=1,2,3$ and $0<r'<r<r_0/2$ one has
\begin{multline}\label{z4.3}
\int_{A^+_{l,k,r'}} (\delta+|\nabla_\e u^\e|^2)^{\frac{Q-2}{2}} |\nabla_\e \omega_l|^2 \eta^2\d\mathcal{L}
\le C \int_{A^+_{l,k,r}} (\delta+|\nabla_\e u^\e|^2)^{\frac{Q-2}{2}} |\omega_l|^2 |\nabla_\e \eta|^2\d\mathcal{L}
\\ +C(\delta+\mu(r_0)^2)^{\frac{Q}{2}} |A_{l,k,r}^+|^{1-\frac{2}{q}},
\end{multline}
where we have set $\omega_l=(X_l^\e u^\e-k)^+$. 
\end{proposition}
\begin{proof} As above, we study the case $l=1$, since $l=2,3$ is similar. We will also call $\mathcal A$ the right hand side of 
\eqref{z4.3}. In view of  \eqref{z4.3-00} we only need to show $I_3\le \mathcal A$. 
From Lemma \ref{caccio-sub-0} one has
$$K\leq (\mathcal A + I_3)^{\frac{1}{2}}.$$
In view of 
 \eqref{z4.3-step2-000} and Corollary \ref{I_3toG_0} one obtains
 \begin{multline}\label{forse}
I_3
\le 
C (\delta+\mu(r_0)^2)^{\frac{Q-2}{4}} |A_{l,k,r}^+|^{\frac{1}{2}}\bigg( 
  \sum_{h=0}^{m}
   K^{2-\frac{1}{2^{m+h}}} (\delta + \mu^2(r_0))^{ 1 + \frac{Q}{2^{m+h+2}}} 
 \bigg)^\frac{1}{2}
\\
\le 
C
 (\delta+\mu(r_0)^2)^{\frac{Q}{4} }    |A_{l,k,r}^+|^{\frac{1}{2}}  
 \sum_{h=0}^{m}   \bigg(  {\mathcal A}^{\frac{1}{2}} +
 I_3^{\frac{1}{2}} 
\Bigg)^{1-\frac{1}{2^{m+h+1}}}(\delta+\mu^2(r_0))^{ \frac{Q}{2^{m+h+3}}},
\end{multline}

Next we observe that in view of Young inequality, for every $h=1,...,m$
\begin{multline}
C 
  (\delta+\mu(r_0)^2)^{\frac{Q}{4} }    |A_{l,k,r}^+|^{\frac{1}{2}} ({ \mathcal A^{\frac{1}{2}-\frac{1}{2^{m+h+2}}}
+   I_3^{\frac{1}{2}  -\frac{1}{2^{m+h+2}}}}\,\,\,)
(\delta+\mu^2(r_0))^{ \frac{Q}{2^{m+h+3}}} \\
\le \frac{1}{2} I_3 + \frac{1}{2}\mathcal A + C \bigg( 
  (\delta+\mu(r_0)^2)^{\frac{Q}{4} }    |A_{l,k,r}^+|^{\frac{1}{2}} (\delta+\mu^2(r_0))^{ \frac{Q}{2^{m+h+3}}}
\bigg)^{  \frac{2^{m+h+2}}{ 1+2^{m+h+1}}}
\\
\le \frac{1}{2} I_3 
+ \frac{1}{2}\mathcal A +C
(\delta+|\nabla_\e u^\e|^2)^{ \frac{Q}{2}}  |A_{l,k,r}^+|^{\frac{1}{2} (\frac{2^{m+h+2}}{ 1+2^{m+h+1}})}
\end{multline}

To complete the proof of \eqref{z4.3} we 
choose $m$ sufficiently large so that
$$ 1-\frac{2}{q} \le  \frac{1}{2} (\frac{2^{m+h+2}}{ 1+2^{m+h+1}}).$$
\end{proof}

A similar argument yields the corresponding result for sub-level sets:
\begin{corollary}  Let  $u^\e\in W^{1,Q}_{\e, {\rm loc}}(B)\cap C^{\infty}(B)$ be the unique  solution of \eqref{zhong-diri}. For any $q\ge 4$ there exists a positive constant $C$ depending only on $q,\lambda, \Lambda$ such that for all $k\in \R$, $l=1,2,3$ and $0<r'<r<r_0/2$ one has
\begin{multline}\label{z4.3-bis}
\int_{A^-_{l,k,r'}} (\delta+|\nabla_\e u^\e|^2)^{\frac{Q-2}{2}} |\nabla_\e \omega_l|^2 \d\mathcal{L}
\le C(r-r')^{-2} \int_{A^-_{l,k,r}} (\delta+|\nabla_\e u^\e|^2)^{\frac{Q-2}{2}} |\omega_l|^2 \d\mathcal{L}
\\ +C(\delta+\mu(r_0)^2)^{\frac{Q}{2}} |A_{l,k,r}^-|^{1-\frac{2}{q}},
\end{multline}
where we have set $\omega_l=(X_l^\e u^\e-k)^- $. 
\end{corollary}

From this point on, the rest of the argument does not rely on the function $u_\e$ being a solution of the  equation anymore but only on the Caccioppoli inequality above. The proof of Theorem \ref{finale}  is very similar to the Euclidean case as developed in \cite{LU}, and \cite{DiBenedetto}. It ultimately relies on the properties of De Giorgi classes in the general setting of metric spaces, as developed in \cite{KN1} and \cite{KMMP}. We recall that a function $f\in W^{1,2}_{\rm H}(B(x_0,r_0)\cap L^\infty(B(x_0,r_0)$ is in the De Giorgi class $DG^+(\chi, q, \gamma)$ if there exists constants $\chi, q, \gamma>0$ such that for every $0<r'<r<r_0/4<1/2$ and $k\in \R$  one has 
\begin{equation}\label{DG}
\int_{B(x_0,r')} |\nabla_\e w|^2 \d\mathcal{L}
 \le \gamma (r-r')^{-2} \int_{B(x_0,r)} w^2 \d\mathcal{L}
 +\chi |\{x\in B(x_0,r) \text{ such that }   w>0\}|^{1-\frac{2}{q}},
\end{equation}
where $\omega=(f-k)^+$. A function $f\in W^{1,2}_{\rm H}(B(x_0,r_0)\cap L^\infty(B(x_0,r_0)$ is in the De Giorgi class $DG^-(\chi, q, \gamma)$ if \eqref{DG} holds for  $\omega=(f-k)^- $. We set $DG(\chi, q, \gamma)=DG^+(\chi, q, \gamma)\cap DG^-(\chi, q, \gamma)$. It is well known, see for instance \cite{KMMP} and references therein, that functions in $DG$ satisfy a scale invariant Harnack inequality and the  following  oscillation bounds: {\it  If $f\in DG(\chi, q, \gamma)$ then there exists $s=s(q, \gamma, Q, r_0)>0$ such that}
\begin{equation*}
{\rm osc}_{B(x_0,r/2)} f \le (1-2^{-s}) {\rm osc}_{B(x_0,r)} f +\chi r^{1-\frac{Q}{q}}.
\end{equation*}
From the latter, the H\"older continuity follows immediately assuming $q$ is large enough. We need to show that \eqref{z4.3} and \eqref{z4.3-bis} imply $X_l^\e u^\e\in DG(\chi, q, \gamma)$. To do this we need to prove a result analogue to \cite[Proposition 4.1]{DiBenedetto}:
\begin{lemma}
 In the notation established above, there exists $\tau>0$ depending on $Q,\lambda,\Lambda,r_0$ such that if for at least one $k=1,2,3$,
$$|\bigg\{ x\in B(x,r) \text{ such that  } X_k^\e u^\e <\tfrac{1}{8}{\rm osc}_{B(x,2r)} 
|\nabla_\e u^\e| \bigg\} | \le \tau r^Q ,$$
then
$$ \sup_{B(x,\frac{r}{2})} X_k^\e u^\e \ge \frac{{\rm osc}_{B(x,2r) } 
|\nabla_\e u^\e|  }{100}.
$$
Analogously, if for at least one $k=1,2,3$,
$$|\bigg\{ x\in B(x,r) \text{ such that  } X_k^\e u^\e >-\frac{1}{8}{\rm osc}_{B(x,2r)} 
|\nabla_\e u^\e| \bigg\} | \le \tau r^Q,$$
then
$$ \sup_{B(x,\frac{r}{2})} X_k^\e u^\e \le -\frac{{\rm osc}_{B(x,2r) } 
|\nabla_\e u^\e|  }{100}.
$$
\end{lemma}
This result is proved exactly as in \cite[Proposition 4.1]{DiBenedetto} (see also \cite[Lemma 4.4]{Zhong}) and it yields essentially the equivalence 
$$(\delta+  \mu(2r)^2)^{\frac{Q-2}{2}} \approx (\delta+|\nabla_\e u^\e|^2)^{\frac{Q-2}{2}},
$$
for all $x\in B(x_0,r)$, when $|\nabla_\e u^\e|$ is small with respect to ${\rm osc}_{B(x,2r)} 
|\nabla_\e u^\e| $. This equivalence, together with  \eqref{z4.3} and \eqref{z4.3-bis} implies $X_l^\e u^\e\in DG(\chi, q, \gamma)$, thus concluding the proof of the H\"older regularity of the gradient in Theorem \ref{finale}. 
\qed

 \bibliography{general_bibliography_donotmodify}
\bibliographystyle{amsalpha}

\end{document}